\documentclass[3p,a4paper,fleqn, preprint]{elsarticle}

\def\tsc#1{\csdef{#1}{\textsc{\lowercase{#1}}\xspace}}
\tsc{WGM}
\tsc{QE}

\usepackage{mathtools,amssymb,amsthm}
\usepackage{algorithm}
\usepackage{booktabs}
\usepackage{multirow}
\usepackage{algpseudocode}
\usepackage{array}
\usepackage{caption}
\usepackage{tikz}
\usetikzlibrary{positioning, fit, backgrounds}
\usepackage{enumitem}
\usepackage{url}
\newtheorem{theorem}{Theorem}
\newtheorem{proposition}{Proposition}
\newtheorem{lemma}{Lemma}
\newtheorem{definition}{Definition}

\theoremstyle{definition}
\theoremstyle{remark}
\newtheorem*{remark}{Remark}

\begin{document}
\let\WriteBookmarks\relax
\def\floatpagepagefraction{1}
\def\textpagefraction{.001}

\cortext[cor1]{Corresponding author}

\title{A Convex Quasilinearization Method for Solving Nonlinear PDEs with Physics-Informed Neural Networks}
\address[tamu]{Texas A\&M University, College Station, Texas, U.S.A.}
\author[tamu]{Gbenga~T.~Awojinrin}
\author[tamu]{Abdul-Akeem~Olawoyin}
\author[tamu]{Rami~M.~Younis \corref{cor1}}
\ead{younis@tamu.edu}

\begin{abstract}
We present a numerical method for the forward solution of nonlinear partial differential equations (PDEs) in which Bellman-Kalaba quasilinearization reduces the nonlinear problem to a sequence of linear subproblems, each discretized by collocation onto a trial space that is linear in its parameters and solved by a single direct linear least-squares QR factorization.
The trial space, which we term \emph{linear-in-learnables} (\textsc{LiL}), comprises any representation whose trainable parameters enter linearly, including random-feature extreme learning machines, spectral polynomial bases, and trigonometric expansions, each realized as a physics-informed neural network. Subsequently, the method replaces the nonconvex gradient-based training that limits standard PINNs with a convex per-step solve.
We establish local Newton-Kantorovich convergence of the outer iteration to a residual-limited neighborhood under an explicit smallness condition, with the limiting accuracy governed by the best-approximation residual of the trial space rather than by an optimization tolerance.
The method, denoted \textsc{LiL-Q}, is assessed on seven benchmarks spanning scalar nonlinear PDEs (Bratu, viscous Burgers, Buckley-Leverett), coupled systems (plane-strain elasticity and the incompressible Navier-Stokes equations in two and three spatial dimensions), and steady-state Darcy flow with heterogeneous permeability.
Across these problems, \textsc{LiL-Q} converges in single-digit outer iterations in most cases and in no more than a few dozen even at the coarsest basis sizes, independent of parameter count.
When the exact solution lies in the span of the trial space, as in the manufactured elasticity case, the method recovers it to machine precision in a single solve.
On the Navier-Stokes benchmarks, it matches or exceeds published PINN solvers using up to two orders of magnitude fewer trainable parameters and without gradient-based optimization.
Conditioning of the collocation matrix is analyzed across all benchmarks and is shown to govern the round-off contribution to accuracy through a rank-revealing column-pivoted QR factorization.
Taken together, these results close much of the cost-and-reliability gap between physics-informed and classical solvers on problems where smooth global trial spaces are effective, while leaving a residual gap on problems where sparse-matrix discretizations retain a structural advantage.
The convexity of the per-step problem, together with the choice of trial space, governs both convergence reliability and attainable accuracy.
\end{abstract}

\maketitle

\section{Introduction}
\label{sec:introduction}

Nonlinear partial differential equations (PDEs) govern a wide range of phenomena across scientific and engineering disciplines.
Because analytical solutions are seldom available, the numerical solution of these equations has motivated decades of research into finite element, finite volume, and spectral methods, balancing accuracy, stability, and computational cost~\cite{LeVeque2007,StrangFix1973}.
Physics-informed neural networks (PINNs) emerged as an alternative paradigm, embedding PDEs directly into a neural network loss function and training the network parameters to satisfy the governing equations~\cite{Raissi2019,Karniadakis2021}.
The formulation is mesh-free, requires no labeled data for forward problems, and yields a differentiable solution representation that unifies forward and inverse problems within a single framework.
Since their introduction, PINNs have seen numerous methodological variants, including loss function modifications such as adaptive weighting~\cite{Wang2021,McClenny2023,Yu2022}, training strategies including curriculum and causal sequential approaches~\cite{Krishnapriyan2021}, architectural enhancements such as Fourier feature embeddings and hard constraint enforcement~\cite{Tancik2020,Lu2021}, domain decomposition methods~\cite{Jagtap2020cPINN}, variational and weak formulations~\cite{Kharazmi2021}, and Bayesian extensions for uncertainty quantification, among others.
Applications span fluid mechanics, heat transfer, subsurface flow, biomedical modeling, and materials science~\cite{Mao2020,Fuks2020,Haghighat2021,Cuomo2022}.
Despite this proliferation, PINNs have not become general-purpose tools for nonlinear PDE simulations. Training remains computationally expensive, convergence is not guaranteed even for linear problems~\cite{Shin2020,DeRyckMishra2024}, and systematic comparisons indicate that PINNs struggle to match the efficiency of classical solvers on standard benchmarks~\cite{McGreivyHakim2024,Grossmann2024}.

A central obstacle is the nonconvex nature of the PINN training problem. This nonlinear complexity arises from two distinct sources of nonlinearity: physical nonlinearity from the governing PDE, and architectural nonlinearity from the layered composition of nonlinear activation functions in the network.
When a nonlinear differential operator acts on a solution representation that is itself nonlinear in its trainable parameters, the resulting loss landscape can admit saddle points, local minima, and pathological training dynamics that have been extensively documented: spectral bias toward low-frequency components~\cite{Wang2022,Wang2021Eigenvector}, causality violation in which networks fit later-time solutions before resolving initial conditions~\cite{Wang2024Causality,Krishnapriyan2021}, gradient imbalance across loss terms~\cite{Wang2021}, failure to propagate information from boundaries into the interior~\cite{Krishnapriyan2021}, and the dominance of optimization error among error components~\cite{DeRyckMishra2024}.
Existing PINN variants address these pathologies through mechanisms that target one source at a time. Adaptive loss weighting, curriculum learning, and domain decomposition navigate the nonconvex landscape more effectively but do not alter its topology. \textit{Linearize-then-learn} strategies, including Newton-informed frameworks~\cite{Hao2024NeurIPS}, homotopy methods~\cite{Huang2022HomPINNs}, locally linearized approaches~\cite{Liu2024LLPINN}, and causal sequential training~\cite{Wang2024Causality,Mattey2022,Roy2024}, address the physical nonlinearity by decomposing the problem into linear or simpler subproblems.
However, when the solution remains parameterized by a standard neural network, the architectural nonlinearity persists and each subproblem remains nonconvex. The well-documented struggles of PINNs on linear PDEs confirm that architectural nonlinearity alone suffices to impede training~\cite{Shin2020,Wang2022}.
Conversely, architecturally linear representations such as Extreme Learning Machines~\cite{Dwivedi2020,DongLi2021,Calabro2021} perform well on linear PDEs but degrade on nonlinear ones, indicating that physical nonlinearity must also be addressed.
Neither partial strategy, applied in isolation, recovers convexity.
Read together, these two failure modes have produced an apparent tension in the empirical experience of the field: expressive nonlinear networks have the representational capacity for difficult PDE solutions but are hard to train to high accuracy, while architecturally linear alternatives train reliably but lack the expressivity needed beyond linear problems.

We present \textsc{LiL-Q}, a PINN training methodology that eliminates both sources of nonlinearity simultaneously and, in doing so, changes the character of the training problem rather than merely easing its navigation.
The nonlinear PDE is solved through a sequence of linear subproblems generated by Bellman-Kalaba quasilinearization~\cite{BellmanKalaba1965}, and the solution at each step is represented by a network whose trainable parameters enter linearly.
Because the governing equation is then linear and the representation is linear in its parameters, the training objective at each outer iteration is convex with a unique minimizer (under a mild rank condition on the collocation matrix). Gradient descent, learning-rate schedules, and adaptive loss reweighting are replaced by a single direct linear least-squares solve.
Convexity is a property of each linear subproblem rather than of the original nonlinear problem. The outer quasilinearization that links the subproblems retains the local convergence of the underlying Newton-Kantorovich iteration, which we establish in Theorem~\ref{thm:nk_convergence}.
The reliability of physics-informed training, on this view, is not a matter of better navigating a nonconvex landscape but of removing the nonconvexity altogether. Once that is done, convergence is deterministic and the attainable accuracy is dictated by the approximation power of the chosen representation rather than by the optimizer.

A related construction appeared as one component of the domain-decomposition framework of Dong and Li~\cite{DongLi2021}, whose Newton-LLSQ variant likewise linearizes the PDE and fits each linear subproblem with the output-layer coefficients of an Extreme Learning Machine (ELM).
There, it was a secondary element evaluated only with random features, and was reported to be less accurate than the authors' primary nonlinear solver.
Our analysis shows that this behavior reflects the uncontrolled approximation error of random features rather than the linearize-then-fit composition itself.
Casting the method over the \textit{full class of linear-in-learnables} representations, we prove that nested orthogonal and trigonometric bases confer guaranteed monotone error reduction (Proposition~\ref{prop:monotone}) and a spectral convergence rate that random features lack, and we establish a convergence theory in which the limiting accuracy is governed by the best-approximation residual of the subspace (Theorems~\ref{thm:stationary_residual},~\ref{thm:residual_bounds}), with local quadratic convergence under an explicit smallness condition (Theorem~\ref{thm:nk_convergence}).
The method further operates at a scale not addressed by these random-feature schemes, including coupled, multi-field, and four-dimensional systems, and solves for each iterate directly so that a single representation persists across the outer iterations.

We assess the approach on problems of increasing difficulty: the Bratu, viscous Burgers, and Buckley-Leverett equations (scalar nonlinear PDEs), plane-strain linear elasticity and the incompressible Navier-Stokes equations in two and three spatial dimensions (coupled systems), and steady-state Darcy flow with highly heterogeneous permeability.
For the scalar benchmarks, we include two intermediate formulations that each address only one source of nonlinearity (\textsc{NiL-Q} and \textsc{LiL-N}), alongside the standard PINN (\textsc{NiL-N}), isolating the contribution of each source. For the coupled and heterogeneous problems, we compare against published PINN baselines and classical numerical references.
Across benchmarks, \textsc{LiL-Q} converges in single-digit outer iterations in most cases, largely independent of the parameter count, recovers the exact solution to machine precision when it lies in the span of the basis, and on the Navier-Stokes systems matches or exceeds the accuracy of published PINN solvers using up to two orders of magnitude fewer trainable parameters and seconds to minutes of computation in place of hours of gradient-based training.

The remainder of this paper is organized as follows.
Section~\ref{sec:math_prelim} presents the mathematical background.
Section~\ref{sec:proposed_approach} develops the proposed approach and establishes properties of the iteration.
Section~\ref{sec:numerical_experiments} presents the numerical experiments.
Section~\ref{sec:discussion} discusses the results, and Section~\ref{sec:conclusion} concludes.

\section{Background and Preliminaries}
\label{sec:math_prelim}

We consider problems of the form,
\begin{equation}
\label{eq:general_pde}
\mathcal{N}(u) = 0, \qquad \mathbf{x} \in \bar{\Omega},
\end{equation}
where $\Omega \subset \mathbb{R}^d$ is the problem domain (spatial, temporal, or both), $u \colon \bar{\Omega} \to \mathbb{R}$ is the unknown field, and $\mathcal{N}$ is a nonlinear residual operator that encodes the governing PDE together with all auxiliary conditions (boundary, initial, or data-driven constraints).
This formulation encompasses any differential problem posed in residual form.

Classical discretization methods approximate~\eqref{eq:general_pde} by a system of algebraic equations solved by Newton-type iterations on a computational mesh, and decades of development have made these approaches efficient for a broad class of problems.
Physics-informed machine learning (PIML) offers an alternative in which the solution $u$ is approximated by a parametric model $\hat{u}(\mathbf{x};\,\boldsymbol{\theta})$, where $\boldsymbol{\theta} \in \mathbb{R}^P$ collects the trainable parameters and $P$ denotes their total number.
Let $\{\mathbf{x}_i\}_{i=1}^{N}$ denote collocation points distributed over $\bar{\Omega}$, including interior, boundary, and (where applicable) initial-condition points.
Evaluating the residual operator at each collocation point produces the discrete residual vector $\mathbf{R}(\boldsymbol{\theta}) \in \mathbb{R}^N$, whose $i$-th component is $R_i(\boldsymbol{\theta}) = \mathcal{N}(\hat{u}(\cdot;\,\boldsymbol{\theta}))(\mathbf{x}_i)$.
The PIML training problem seeks optimal parameters by minimizing the sum of squared residuals; i.e.,
\begin{equation}
\label{eq:piml_minimization}
\boldsymbol{\theta}^{*}
  = \operatorname*{arg\,min}_{\boldsymbol{\theta} \in \mathbb{R}^P} \mathcal{L}(\boldsymbol{\theta}),
\end{equation}
where,
\begin{equation}
\label{eq:loss_function}
\mathcal{L}(\boldsymbol{\theta}) \coloneqq \bigl\| \mathbf{R}(\boldsymbol{\theta}) \bigr\|_2^2.
\end{equation}
When the parametric model is a neural network, the resulting method is termed a physics-informed neural network.

\subsection{Neural Networks}
\label{sec:neural_networks}

A fully connected feedforward network with $L$ hidden layers with widths $\{n_\ell\}_{\ell=1}^{L}$ defines the approximation through the recursive composition,
\begin{equation}\label{eq:feedforward}
\begin{aligned}
\mathbf{z}^{(0)} &= \mathbf{x}, \\[2pt]
\mathbf{z}^{(\ell)} &= \sigma\bigl(\mathbf{W}^{(\ell)}\mathbf{z}^{(\ell-1)} + \mathbf{b}^{(\ell)}\bigr), \quad \ell = 1, \ldots, L, \\[2pt]
\hat{u}(\mathbf{x};\,\boldsymbol{\theta}) &= \mathbf{w}^{\mathsf{T}}\mathbf{z}^{(L)} + b,
\end{aligned}
\end{equation}
where $\mathbf{W}^{(\ell)} \in \mathbb{R}^{n_\ell \times n_{\ell-1}}$ and $\mathbf{b}^{(\ell)} \in \mathbb{R}^{n_\ell}$ are the weight matrix and bias vector at hidden layer $\ell$, the output layer is parameterized by $\mathbf{w} \in \mathbb{R}^{n_L}$ and $b \in \mathbb{R}$, and $\sigma \colon \mathbb{R} \to \mathbb{R}$ is a nonlinear activation function applied elementwise.
When the full parameter set $\boldsymbol{\theta} = \{\mathbf{W}^{(\ell)}, \mathbf{b}^{(\ell)}\}_{\ell=1}^{L} \cup \{\mathbf{w}, b\}$ is trained, the map $\boldsymbol{\theta} \mapsto \hat{u}(\mathbf{x};\,\boldsymbol{\theta})$ involves nested compositions of affine transformations and nonlinear activations. We refer to such networks as nonlinear-in-learnables (\textsc{NiL}).
Alternatively, if all hidden-layer parameters are held fixed and only the output-layer coefficients $\mathbf{w}$ are trained, the map from trainable parameters to network output becomes linear. We expand on this construction and its implications in Section~\ref{sec:proposed_approach}.

\subsection{The PINN Training Problem}
\label{sec:pinns}

Standard PINN training employs a \textsc{NiL} network as the solution proxy $\hat{u}(\mathbf{x};\,\boldsymbol{\theta})$ for a nonlinear PDE.
The loss~\eqref{eq:loss_function} is then a nonlinear function of $\boldsymbol{\theta}$ through both the PDE operator $\mathcal{N}$ acting on the network output and the implicit dependence of that output on nested compositions of affine maps and activations.
Training typically proceeds by gradient-based optimization such as first-order methods including Adam or quasi-Newton methods such as L-BFGS, each seeking parameters that satisfy the first-order stationarity condition for a local minimum.
From the least-squares structure of~\eqref{eq:loss_function}, this condition takes the form,
\begin{equation}
\label{eq:stationarity}
\nabla_{\boldsymbol{\theta}} \mathcal{L}(\boldsymbol{\theta}) = \mathbf{J}(\boldsymbol{\theta})^{\mathsf{T}} \mathbf{R}(\boldsymbol{\theta}) = \mathbf{0},
\end{equation}
where $\mathbf{J}(\boldsymbol{\theta}) \in \mathbb{R}^{N \times P}$ is the Jacobian of the residual vector with respect to the network parameters.
The stationarity condition~\eqref{eq:stationarity} constitutes a system of $P$ coupled nonlinear equations in which at least two nonlinearities are intertwined.
First, the residual $\mathbf{R}(\boldsymbol{\theta})$ depends nonlinearly on $\boldsymbol{\theta}$ because the PDE operator $\mathcal{N}$ acts on the network output $\hat{u}$, which is itself a nonlinear function of its parameters.
Second, each row of the Jacobian involves the Fr\'{e}chet derivative $\mathcal{N}'(\hat{u})$ evaluated at the current network output, coupling the linearization of the physics to the current iterate.
Third, the parameter sensitivity $\partial_{\boldsymbol{\theta}} \hat{u}$ inherits the nested nonlinearity of the network architecture.
Consequently, the training problem~\eqref{eq:loss_function} is a nonlinear least-squares problem whose loss landscape is nonconvex, admitting local minima, saddle points, and plateaus that impede gradient-based optimization~\cite{Rathore2024}.
Empirical studies have confirmed the practical consequences where spectral bias toward low-frequency components, stalled convergence in the presence of multiscale physics, and failure to reduce residuals to acceptable levels even after extensive training can occur~\cite{Wang2021,Krishnapriyan2021}.
We refer to this standard configuration, in which a \textsc{NiL} architecture is trained to directly approximate the solution of a nonlinear PDE, as \textsc{NiL-N}.

To quantify the computational cost of \textsc{NiL-N} training in a hardware-agnostic manner, we employ floating-point operation (FLOP) counts.
For a feedforward network with $P$ parameters evaluated at $N$ collocation points, a forward pass requires approximately $2NP$ FLOPs while the backward pass for gradient computation requires $4NP$ FLOPs, yielding the widely cited estimate of $6NP$ FLOPs per training iteration~\cite{Goodfellow2016}.
Automatic differentiation through the network for computing PDE-required derivatives adds further cost proportional to $NP$, with the precise constant depending on the operator structure.
We write the per-iteration cost as $C_{\textsc{NiL}}\, NP$, where $C_{\textsc{NiL}} \geq 6$ is a problem-dependent constant that absorbs the forward and backward pass cost, the derivative overhead, and optimizer-specific computations.
If training terminates after $K$ iterations, the total FLOP count for \textsc{NiL-N} is
\begin{equation}
\label{eq:flop_niln}
\mathcal{F}_{\textsc{NiL-N}} = \mathcal{O}(K \, N \, P).
\end{equation}

\section{Proposed Approach}
\label{sec:proposed_approach}

Standard PINN training (\textsc{NiL-N}) optimizes a \textsc{NiL} network directly on the nonlinear PDE, contending with the nonconvex loss landscape analyzed in Section~\ref{sec:pinns}.
The approach proposed here operates differently.
Rather than training on the nonlinear problem as posed, we first linearize the governing PDE about the current iterate using Bellman-Kalaba quasilinearization~\cite{BellmanKalaba1965}, producing a linear subproblem at each outer iteration.
We then approximate the solution of this subproblem not with a \textsc{NiL} network, but with a linear-in-learnables (\textsc{LiL}) representation. \textsc{LiL} representations are networks or prescribed basis expansions in which the trainable parameters enter the solution linearly.
Because the governing equation is linear (by quasilinearization) and the solution representation is linear in its trainable parameters (by construction), the training objective (mean square residual error) at each outer iteration is strictly convex, and neither the physical nonlinearity of the PDE nor the architectural nonlinearity of the network parameterization enters the loss landscape.
We term this approach \textsc{LiL-Q} (Linear-in-Learnables with Quasilinearization).
The following subsections develop the quasilinearization scheme (Section~\ref{sec:quasilinearization}), the \textsc{LiL} representation (Section~\ref{sec:lil_representations}), the resulting convex training formulation (Section~\ref{sec:linear_ls}), a worked example on the Bratu equation (Section~\ref{sec:bratu_example}), and the properties of the iteration (Section~\ref{sec:properties}). 

\subsection{Bellman-Kalaba Quasilinearization}
\label{sec:quasilinearization}
 
The technique of quasilinearization, introduced by Bellman and Kalaba~\cite{BellmanKalaba1965}, applies the Newton-Kantorovich method in function space to solve nonlinear differential equations.
The original formulation was motivated by the desire to obtain analytical or semi-analytical solutions to nonlinear boundary-value problems arising in optimal control and mathematical physics~\cite{Mandelzweig2001}. By linearizing the nonlinear operator about a current iterate, the method converts the nonlinear problem into a sequence of linear PDEs that can, in favorable cases, be solved in closed form (akin to an Exact Newton-Kantorovich method).
Given a nonlinear differential equation $\mathcal{N}(u) = 0$, the method generates a sequence of approximations $\{\hat{u}^{(k)}\}$ by linearizing $\mathcal{N}$ about the current solution iterate using its Fr\'{e}chet derivative, denoted as the linear operator $\mathcal{N}'(u)$ satisfying
\begin{equation}
\label{eq:frechet_definition}
\mathcal{N}(u + \epsilon v) = \mathcal{N}(u) + \epsilon\,\mathcal{N}'(u)(v) + O(\epsilon^2).
\end{equation}
Assuming Fr\'{e}chet differentiability, the operator $\mathcal{N}'(u)$ is linear in its argument, and for any functions $v$ and $w$ and scalars $\alpha$ and $\gamma$, 
\[ \mathcal{N}'(u)(\alpha v + \gamma w) = \alpha\,\mathcal{N}'(u)(v) + \gamma\,\mathcal{N}'(u)(w).\]
Applying this linearization about the current iterate $\hat{u}^{(k)}$ yields the quasilinearized iteration generated by the subproblem
\begin{equation}
\label{eq:quasilinear_subproblem}
\mathcal{N}'(\hat{u}^{(k)})\bigl(\hat{u}^{(k+1)}\bigr) = \mathcal{N}'(\hat{u}^{(k)})\bigl(\hat{u}^{(k)}\bigr) - \mathcal{N}(\hat{u}^{(k)}), \qquad k=1,2,\ldots,
\end{equation}
which is a linear equation for $\hat{u}^{(k+1)}$ with right-hand side depending only on the previous iterate.
The linearized equation is solved directly for the next iterate $\hat{u}^{(k+1)}$, rather than for an update $\delta u^{(k)} = \hat{u}^{(k+1)} - \hat{u}^{(k)}$ as is more common in standard Newton formulations. The two approaches are algebraically equivalent when solved exactly, but the direct formulation proves more natural in the present context because the full solution, rather than a correction, is represented at every iteration.
 
When solved exactly in an infinite-dimensional function space, the quasilinearization iteration recovers the classical Newton-Kantorovich method with local quadratic convergence under standard regularity assumptions~\cite{Deuflhard2011}.
In this work, we solve each quasilinear subproblem approximately using a finite-dimensional linear-in-learnables representation, described in the following subsection. The accuracy of this representation at each step strongly influences the convergence behavior of the outer iteration. A related iteration is the Inexact Newton method (IN) of~\cite{DemboEisenstatSteihaug1982} in which the linear subproblems are solved inexactly, and up to a prescribed residual error tolerance. IN methods guarantee local convergence properties under a the condition of a \textit{forcing function} that essentially tightens the tolerance asymptotically towards the zero residual solution. In the proposed methods, this forcing mechanism is absent, and leads us to develop the relationships formalized in Section~\ref{sec:properties} relating convergence properties to approximation capacity and numerical conditioning.

\subsection{Linear-in-Learnables Representations}
\label{sec:lil_representations}
 
The quasilinearized subproblem~\eqref{eq:quasilinear_subproblem} is a linear PDE for $\hat{u}^{(k+1)}$.
To preserve this linearity in the trainable parameters, we represent the solution at each step using a network in which all hidden-layer parameters are frozen and only the output-layer coefficients are trained.
In this setting, the hidden-layer outputs become fixed functions of the input, and the map from the remaining trainable coefficients to the network output is linear.
This characterization holds irrespective of the depth of the frozen portion. That is, whether the hidden representation comprises a single layer or multiple composed layers, fixing all hidden-layer weights and biases yields a network that is linear in its trainable parameters.
For a single hidden layer with $P$ nodes, the construction takes the explicit form,
\begin{equation}
\label{eq:linear_network}
\hat{u}(\mathbf{x};\,\boldsymbol{\beta}) = \sum_{p=1}^{P} \beta_p\,\sigma(\mathbf{w}_p^{\mathsf{T}}\mathbf{x} + b_p),
\end{equation}
where the hidden-layer weights $\{\mathbf{w}_p\}$ and biases $\{b_p\}$ are predetermined, $\sigma$ is a nonlinear activation function, and $\boldsymbol{\beta} \in \mathbb{R}^P$ collects the trainable output-layer coefficients.
The ELM framework~\cite{Huang2006} popularized this construction by sampling the hidden-layer parameters from a probability distribution and solving for $\boldsymbol{\beta}$ via linear least squares.
Deterministic initialization strategies, in which the hidden-layer parameters are set according to prescribed rules rather than sampled randomly, offer an alternative.
A third variant first trains the hidden-layer parameters in a conventional learning phase, then freezes them to form the linear network structure. This two-phase approach produces basis functions adapted to the problem at hand while retaining the computational advantages of the linear output-layer solve.
The choice of activation function is flexible: sigmoidal, hyperbolic tangent, sinusoidal, and Gaussian radial basis functions are all admissible.
While classical universal approximation results apply to fully-trained networks, the ELM theory~\cite{Huang2006} guarantees that a single-hidden-layer network with fixed random weights can approximate any continuous function on a compact domain to arbitrary accuracy given sufficiently many nodes. However, this guarantee does not ensure that increasing $P$ monotonically reduces the approximation error for a fixed target function.
 
The linearity of the map $\boldsymbol{\beta} \mapsto \hat{u}(\cdot;\,\boldsymbol{\beta})$ is the property that makes this construction suitable for the quasilinearization scheme, and this property is not unique to neural network architectures.
Any representation of the form,
\begin{equation}
\label{eq:lil_representation}
\hat{u}(\mathbf{x};\,\boldsymbol{\beta}) = \sum_{p=1}^{P} \beta_p\,\phi_p(\mathbf{x}),
\end{equation}
where the basis functions $\{\phi_p\}_{p=1}^{P}$ are prescribed prior to training, possesses the same linearity.
The linear network~\eqref{eq:linear_network} is a special case of~\eqref{eq:lil_representation} with $\phi_p(\mathbf{x}) = \sigma(\mathbf{w}_p^{\mathsf{T}}\mathbf{x} + b_p)$.
Other instantiations include orthogonal polynomial bases (Chebyshev, Legendre), spline bases, and wavelet families. For smooth target functions, these offer well-characterized convergence rates with monotonically decreasing approximation error as $P$ increases, though these guarantees require regularity that may not hold for all problems.
 
We refer to any representation satisfying~\eqref{eq:lil_representation} as \emph{linear-in-learnables} (\textsc{LiL}), encompassing both network-based and prescribed-basis instantiations.
The choice between the two categories involves tradeoffs. Network-based representations inherit the universal approximation property but may lack guaranteed monotonic error reduction with increasing $P$, and prescribed polynomial, compactly supported functions, or trigonometric bases offer deterministic convergence for smooth functions but require regularity assumptions on the target.
In either case, the defining property is that $\boldsymbol{\beta} \mapsto \hat{u}(\mathbf{x};\,\boldsymbol{\beta})$ is linear, which ensures that the residual of the quasilinearized subproblem~\eqref{eq:quasilinear_subproblem}, when evaluated at collocation points, depends linearly on the unknown coefficients.

\subsection{Convex Training Formulation}
\label{sec:linear_ls}
 
Substituting the \textsc{LiL} representation~\eqref{eq:lil_representation} into the quasilinearized subproblem~\eqref{eq:quasilinear_subproblem} and exploiting the linearity of $\mathcal{N}'(\hat{u}^{(k)})$ yields,
\begin{equation}
\label{eq:linearized_expansion}
\sum_{p=1}^{P} \beta_p^{(k+1)} \mathcal{N}'(\hat{u}^{(k)})(\phi_p) = \mathcal{N}'(\hat{u}^{(k)})(\hat{u}^{(k)}) - \mathcal{N}(\hat{u}^{(k)}).
\end{equation}
This is a functional equation in which the unknown coefficients $\boldsymbol{\beta}^{(k+1)}$ appear linearly.
Evaluating both sides at the $N$ collocation points $\{\mathbf{x}_i\}_{i=1}^{N} \subset \bar{\Omega}$ and stacking the resulting equations produces the linear system,
\begin{equation}
\label{eq:linear_system}
\mathbf{A}^{(k)} \boldsymbol{\beta}^{(k+1)} = \mathbf{f}^{(k)},
\end{equation}
where $\mathbf{A}^{(k)} \in \mathbb{R}^{N \times P}$ is the collocation matrix with entries,
\begin{equation}
\label{eq:collocation_matrix_entry}
A_{ip}^{(k)} = \bigl[\mathcal{N}'(\hat{u}^{(k)})(\phi_p)\bigr](\mathbf{x}_i),
\end{equation}
and $\mathbf{f}^{(k)} \in \mathbb{R}^N$ is the right-hand side vector with entries,
\begin{equation}
\label{eq:rhs_entry}
f_i^{(k)} = \bigl[\mathcal{N}'(\hat{u}^{(k)})(\hat{u}^{(k)}) - \mathcal{N}(\hat{u}^{(k)})\bigr](\mathbf{x}_i).
\end{equation}
Because $\mathcal{N}$ encodes both the governing PDE and all auxiliary conditions, the rows of $\mathbf{A}^{(k)}$ and $\mathbf{f}^{(k)}$ take different forms depending on the type of collocation point. At interior points, they encode the linearized PDE operator and at boundary or initial-condition points, they encode the corresponding auxiliary constraint.
When $N > P$, the system~\eqref{eq:linear_system} is overdetermined, and we solve it in the least-squares sense, paralleling the formulation~\eqref{eq:loss_function}, i.e.,
\begin{equation}
\label{eq:lil_least_squares}
\boldsymbol{\beta}^{(k+1)} = \arg\min_{\boldsymbol{\beta} \in \mathbb{R}^P} \; \mathcal{L}^{(k)}(\boldsymbol{\beta}),
\end{equation}
where,
\begin{equation}
\label{eq:lil_least_squares_loss}
\mathcal{L}^{(k)}(\boldsymbol{\beta}) \coloneqq \bigl\| \mathbf{A}^{(k)} \boldsymbol{\beta} - \mathbf{f}^{(k)} \bigr\|_2^2.
\end{equation}
The loss $\mathcal{L}^{(k)}$ is a convex quadratic in $\boldsymbol{\beta}$ and its Hessian $2(\mathbf{A}^{(k)})^{\mathsf{T}}\mathbf{A}^{(k)}$ is constant, independent of $\boldsymbol{\beta}$, and positive semi-definite, so $\mathcal{L}^{(k)}$ is convex and attains its global minimum value when $\mathbf{A}^{(k)}$ has full column rank. Equivalently when $\sigma_{\min}(\mathbf{A}^{(k)}) > 0$ the Hessian is positive definite and the minimizer is unique. Otherwise the minimum is attained on an affine subspace and a QR factorization for instance, returns the unique minimum-norm solution.
In either case the global optimum is reached by a single direct computation, with no dependence on initialization.
This stands in direct contrast to the \textsc{NiL-N} training problem~\eqref{eq:loss_function}, where the gradient~\eqref{eq:stationarity} is a nonlinear function of the parameters, the Hessian varies across parameter space, and the resulting loss landscape is nonconvex.
Setting the gradient to zero yields a linear normal equation, and the minimizer is computed via QR factorization of $\mathbf{A}^{(k)}$.
Each outer iteration of \textsc{LiL-Q} is therefore a convex optimization problem with a closed-form solution, in contrast to the nonconvex optimization problem that standard PINN training faces at every gradient step.
The complete \textsc{LiL-Q} procedure is summarized in Algorithm~\ref{alg:lilq}.
 
\begin{algorithm}
\caption{The \textsc{LiL-Q} method for solving~\eqref{eq:general_pde}.}
\label{alg:lilq}
\begin{algorithmic}[1]
\State \textbf{Input:} Nonlinear residual operator $\mathcal{N}$; basis functions $\{\phi_p\}_{p=1}^{P}$; collocation points $\{\mathbf{x}_i\}_{i=1}^{N} \subset \bar{\Omega}$; initial coefficients $\boldsymbol{\beta}^{(0)} \in \mathbb{R}^P$; update tolerance $\tau > 0$; maximum iterations $K_{\max}$.
\State \textbf{Output:} Coefficient vector $\boldsymbol{\beta}^*$ such that $\hat{u}(\mathbf{x}; \boldsymbol{\beta}^*) = \sum_{p=1}^{P} \beta_p^* \phi_p(\mathbf{x})$ approximates the solution of~\eqref{eq:general_pde}.
\State Form the initial approximation $\hat{u}^{(0)}(\mathbf{x}) = \sum_{p=1}^{P} \beta_p^{(0)} \phi_p(\mathbf{x})$.
\For{$k = 0, 1, 2, \ldots, K_{\max}-1$}
    \State \parbox[t]{0.90\linewidth}{\textbf{Assemble the collocation matrix and right-hand side:} For each collocation point $\mathbf{x}_i$, $i = 1, \ldots, N$, and each basis function $\phi_p$, $p = 1, \ldots, P$, compute
    \begin{equation*}
    A_{ip}^{(k)} = \bigl[\mathcal{N}'(\hat{u}^{(k)})(\phi_p)\bigr](\mathbf{x}_i),
    \end{equation*}
    and,
    \begin{equation*}
    f_i^{(k)} = \bigl[\mathcal{N}'(\hat{u}^{(k)})(\hat{u}^{(k)}) - \mathcal{N}(\hat{u}^{(k)})\bigr](\mathbf{x}_i).
    \end{equation*}
Stack these into the collocation matrix $\mathbf{A}^{(k)} \in \mathbb{R}^{N \times P}$ and right-hand side $\mathbf{f}^{(k)} \in \mathbb{R}^N$.}
    \vspace{0.2cm}
    \State \parbox[t]{0.90\linewidth}{\textbf{Solve the linear least-squares problem:}
    \begin{equation*}
    \boldsymbol{\beta}^{(k+1)} = \arg\min_{\boldsymbol{\beta} \in \mathbb{R}^P} \left\| \mathbf{A}^{(k)} \boldsymbol{\beta} - \mathbf{f}^{(k)} \right\|_2^2.
    \end{equation*}
    Compute $\boldsymbol{\beta}^{(k+1)}$ via QR factorization of $\mathbf{A}^{(k)}$.}
    \vspace{0.2cm}
    \State \textbf{Update the approximation:} Set $\hat{u}^{(k+1)}(\mathbf{x}) = \sum_{p=1}^{P} \beta_p^{(k+1)} \phi_p(\mathbf{x})$.
    \State \textbf{Check convergence:} If $\|\boldsymbol{\beta}^{(k+1)} - \boldsymbol{\beta}^{(k)}\|_2 < \tau$, terminate and return $\boldsymbol{\beta}^* = \boldsymbol{\beta}^{(k+1)}$.
\EndFor
\State \textbf{return} $ \boldsymbol{\beta}^{(K_{\max})}$.
\end{algorithmic}
\end{algorithm}
 
Because the system~\eqref{eq:linear_system} is overdetermined, the least-squares solution generally incurs a nonzero residual 
$\mathbf{R}_{\mathrm{lin}}^{(k)} := \mathbf{A}^{(k)}\boldsymbol{\beta}^{(k+1)} 
- \mathbf{f}^{(k)}$, reflecting the extent to which the chosen basis cannot exactly satisfy the quasilinearized PDE at all collocation points simultaneously.
This residual is determined by the approximation capacity of the basis and the conditioning of the collocation matrix, and not by the optimizer.
Because the outer quasilinearization iteration relies on each subproblem being solved with sufficient accuracy, the magnitude of this linear residual directly influences the convergence behavior of the overall \textsc{LiL-Q} scheme. This critical relationship is characterized formally in Section~\ref{sec:properties}.

It is worth contrasting this formulation with classical spectral collocation, with which it shares a reliance on global basis functions but differs fundamentally in how the discrete problem is posed and solved.
Spectral collocation enforces the residual to vanish \emph{exactly} at a set of carefully chosen nodes, producing a square system (i.e., $N = P$). For a nonlinear PDE this is a nonlinear algebraic system whose solution is a root-finding problem, typically addressed by Newton-like iteration on the full system, with the attendant sensitivity to initialization and the possibility of converging to spurious roots.
The \textsc{LiL-Q} formulation instead over-samples the domain ($N > P$) and solves the resulting overdetermined system in the least-squares sense~\eqref{eq:lil_least_squares}, and through the outer quasilinearization, each inner solve is a single convex linear least-squares problem rather than a nonlinear root-find.
The two approaches draw their accuracy from the approximation power of the global basis, which for smooth solutions decays spectrally (Section~\ref{sec:properties}).
What differs is the mechanism of solution and what it costs and buys.
Enforcing exact nodal satisfaction couples the accuracy of spectral collocation to the placement of the $P$ nodes and to the well-posedness of the square system, and for nonlinear problems inherits the difficulties of nonlinear algebraic solves. The least-squares formulation relaxes exact satisfaction, accepting a nonzero residual floor governed by the best-approximation error of the basis, in exchange for a convex inner solve that is independent of initialization, a system whose conditioning is handled by a single rank-revealing factorization, and a robustness to node placement that accommodates scattered points, irregular geometries, and coupled multi-field systems for which constructing a well-posed square collocation system could be challenging.
This trade is not necessarily universally advantageous. For instance, for smooth solutions on regular domains where a square collocation system is well-posed, exact nodal collocation can attain comparable accuracy with fewer points, since least-squares over-sampling is not without cost.
Rather, relaxing to least-squares is the choice that enables the convexity and robustness on which the present method rests, and which make it applicable to the nonlinear, heterogeneous, and high-dimensional problems considered in Section~\ref{sec:numerical_experiments}.
 
\subsection{Worked Example: Bratu Equation}
\label{sec:bratu_example}
 
Consider the Bratu boundary value problem,

\begin{equation}
\label{eq:bratu_example}
0 = \mathcal{N}(\hat{u})(\mathbf{x}_i) :=
\begin{cases}
\Delta \hat{u}(\mathbf{x}) + \lambda\, e^{\hat{u}(\mathbf{x})}, & \mathbf{x} \in \Omega, \\[4pt]
\hat{u}(\mathbf{x}) - u_b\left(\mathbf{x}\right), & \mathbf{x} \in \partial\Omega,
\end{cases}
\end{equation}
with Dirichlet boundary data $u_b$ and parameter $\lambda \neq 0$. In this case, the Fr\'{e}chet derivative is the parameterized operator,
\begin{equation}
\label{eq:bratu_expansion}
\mathcal{N}'(\hat{u}^{(k)})(v) = \lim_{\epsilon \to 0}{\frac{d}{d \epsilon} \mathcal{N}(\hat{u}^{(k)} + \epsilon v) } = 
\begin{cases}
\Delta v + \lambda e^{\hat{u}^{(k)}} v, & \mathbf{x} \in \Omega, \\[4pt]
v , & \mathbf{x} \in \partial\Omega,
\end{cases}
\end{equation}
Starting with an initial guess $\hat{u}^{(0)}$, quasilinearization generates a sequence of iterates, $\hat{u}^{(k)}$, defined as solutions to the sequence of linear subproblems,
\begin{equation}
\label{eq:bratu_linearized}
\begin{cases}
\Delta \hat{u}^{(k)} + \lambda e^{\hat{u}^{(k-1)}} \hat{u}^{(k)} = \lambda e^{\hat{u}^{(k-1)}} \bigl( \hat{u}^{(k-1)} - 1 \bigr), & \mathbf{x} \in \Omega, \\[4pt]
\hat{u}^{(k)} = u_b\left(\mathbf{x}\right), & \mathbf{x} \in \partial\Omega,
\end{cases}, \qquad k=1,2,\ldots.
\end{equation}
Typical convergence criteria for Newton-like iterations are a combination of contraction in the Newton update and sufficiently low nonlinear residual error. Our proposed method only relies on a stopping criterion based on contraction of the updates (i.e., stagnation or stall) or a maximum allowable number of iterations. As we show in the next subsection, this is because of a limit on the attainable residual error that is connected to the underlying representation of the iterate. We represent the iterates using the \textsc{LiL} expansion~\eqref{eq:lil_representation},
\begin{equation}
\label{eq:bratu_lil}
\hat{u}^{(k)}(\mathbf{x}) = \sum_{p=1}^{P}{ \beta_p^{(k)} \phi_p(\mathbf{x})},
\end{equation}
and substituting into~\eqref{eq:bratu_linearized} gives
\begin{equation}
\label{eq:bratu_substituted}
\begin{cases}
\sum_{p=1}^{P}{\beta_p^{(k)} \bigl( \Delta \phi_p + \lambda e^{\hat{u}^{(k-1)}} \phi_p \bigr) } = \lambda e^{\hat{u}^{(k-1)}} \bigl( \hat{u}^{(k-1)} - 1 \bigr), & \mathbf{x} \in \Omega, \\[4pt]
\sum_{p=1}^{P}{ \beta_p^{(k)} \phi_p(\mathbf{x})} = u_b\left(\mathbf{x}\right), & \mathbf{x} \in \partial\Omega,
\end{cases}, \qquad k=1,2,\ldots.
\end{equation}
Evaluating~\eqref{eq:bratu_substituted} at the collocation points $\mathbf{x}\in\mathbb{R}^N$ produces the rectangular linear algebraic system, 
\[ \mathbf{A}^{(k-1)} \boldsymbol{\beta}^{(k)} = \mathbf{f}^{(k-1)},\] 
where the (typically dense) coefficient matrix $\mathbf{A}^{(k-1)}\in \mathbb{R}^{N \times P}$ depends on the previous iterate with elements defined as,
\begin{equation}
\label{eq:bratu_matrix}
A_{ip}^{(k-1)} = \begin{cases}
\Delta \phi_p(\mathbf{x}_i) + \lambda e^{\hat{u}^{(k-1)}(\mathbf{x}_i)} \phi_p(\mathbf{x}_i) & \mathbf{x}_i \in \Omega \\
\phi_p(\mathbf{x}_i) & \mathbf{x}_i \in \partial \Omega
\end{cases}, \qquad p=1,\ldots,P, \; i=1,\ldots,N.
\end{equation}
The right hand side is defined as,
\begin{equation}
\label{eq:bratu_rhs}
f_i^{(k-1)} = \begin{cases}
\lambda e^{\hat{u}^{(k-1)}(\mathbf{x}_i)} \bigl( \hat{u}^{(k-1)}(\mathbf{x}_i) - 1 \bigr) & \mathbf{x}_i \in \Omega \\
u_b\left(\mathbf{x}_i\right) & \mathbf{x}_i \in \partial \Omega
\end{cases}, \qquad i=1,\ldots,N.
\end{equation}
So, starting with an initial guess $\boldsymbol{\beta}^{(0)}$, the process solves the sequence of problems,
\begin{equation*}
\boldsymbol{\beta}^{(k)} = \arg\min_{\boldsymbol{\beta} \in \mathbb{R}^P } \left( \mathbf{A}^{(k-1)} \boldsymbol{\beta} - \mathbf{f}^{(k-1)} \right)^T 
\left( \mathbf{A}^{(k-1)} \boldsymbol{\beta} - \mathbf{f}^{(k-1)} \right), \qquad k=1,2,\ldots.
\end{equation*}
The linear least squares normal equations provide a pathway for direct solution~\cite{dlls}, i.e,
\[ {\mathbf{A}^{(k-1)}}^T \mathbf{A}^{(k-1)} \boldsymbol{\beta}^k = {\mathbf{A}^{(k-1)}}^T \mathbf{f}^{(k-1)}. \]
Alternatively, one may opt for a preconditioned iterative method~\cite{itlls}.

\subsection{Properties of the LiL-Q Iteration}
\label{sec:properties}

As noted in Section~\ref{sec:linear_ls} the linearized residual $\mathbf{R}_{\mathrm{lin}}^{(k)}$ at each outer iteration is generally nonzero and its magnitude influences the convergence behavior of the overall scheme.
This subsection formalizes that relationship and quantifies the computational cost of the iteration.
The main results establish that the \textsc{LiL-Q} iteration converges to a stationary point at which the nonlinear PDE residual is controlled by the approximation capacity of the chosen basis. When the basis cannot represent the linearized right-hand side exactly ($\varepsilon_{L} > 0$), the residual is bounded away from zero by an amount proportional to~$\varepsilon_{L}$, and when the exact solution lies in the span of the basis ($\varepsilon_{L} = 0$), the method recovers it to machine precision.

\subsubsection{Setting and assumptions}
\label{sec:setting}

Let $X$ and $Y$ be Banach spaces and let $\mathcal{N} \colon D(\mathcal{N}) \subset X \to Y$ be the nonlinear differential operator from~\eqref{eq:general_pde}, with an isolated root $u^{*} \in D(\mathcal{N})$ satisfying $\mathcal{N}(u^{*}) = 0$.
The \textsc{LiL-Q} iteration generates a sequence $\{u^{(k)}\}_{k \geq 0}$ by solving the quasilinearized subproblem~\eqref{eq:quasilinear_subproblem} in the least-squares sense over the finite-dimensional subspace $V_{P} = \mathrm{span}\{\phi_{1}, \ldots, \phi_{P}\} \subset X$. To connect this iteration to the classical Newton-Kantorovich framework, we rewrite~\eqref{eq:quasilinear_subproblem} in terms of the increment $\delta u^{(k)} = u^{(k+1)} - u^{(k)}$.
Since $\mathcal{N}'(u^{(k)})$ is a linear operator, we obtain
\begin{equation}
\label{eq:projected_newton}
\mathcal{N}'(u^{(k)})(\delta u^{(k)}) + \mathcal{N}(u^{(k)}) =: r^{(k)},
\end{equation}
where $r^{(k)} \in Y$ is the residual remaining after the least-squares solve.
When evaluated at the collocation points, this quantity corresponds to the discrete linearized residual $\mathbf{R}_{\mathrm{lin}}^{(k)} = \mathbf{A}^{(k)}\boldsymbol{\beta}^{(k+1)} - \mathbf{f}^{(k)}$ introduced in Section~\ref{sec:linear_ls}.
Because the solution is constrained to the finite-dimensional subspace~$V_{P}$, the linearized equation cannot in general be satisfied exactly, and the residual $r^{(k)}$ is bounded away from zero by an amount reflecting the best-approximation error of the basis. Fix an open neighborhood $U \subset D(\mathcal{N})$ containing $u^{*}$.
The following regularity conditions are imposed on $\mathcal{N}$ and on the iteration.

\begin{enumerate}[label=\textbf{(A\arabic*)},leftmargin=2.5em]
\item \label{assump:smooth} \textit{Smoothness.}
$\mathcal{N}$ is continuously Fr\'{e}chet differentiable on $U$, i.e., $\mathcal{N} \in C^{1}(U; Y)$.

\item \label{assump:lipschitz} \textit{Lipschitz continuity of the derivative.}
There exists a constant $L > 0$ such that,
\begin{equation}
\label{eq:lipschitz}
\|\mathcal{N}'(u) - \mathcal{N}'(v)\|_{\mathcal{L}(X,Y)} \leq L\,\|u - v\|_{X}
\quad \text{for all } u,\, v \in U.
\end{equation}

\item \label{assump:invertible} \textit{Uniform invertability.}
For all $u \in U$, the linearized operator $\mathcal{N}'(u) \colon X \to Y$ is bijective with
\begin{equation}
\label{eq:invertibility}
\|\mathcal{N}'(u)^{-1}\|_{\mathcal{L}(Y,X)} \leq \Lambda
\quad \text{for some } \Lambda > 0.
\end{equation}

\item \label{assump:bounded} \textit{Uniform boundedness.}
There exists a $\Gamma > 0$ such that
\begin{equation}
\label{eq:bounded}
\|\mathcal{N}'(u)\|_{\mathcal{L}(X,Y)} \leq \Gamma
\quad \text{for all } u \in U.
\end{equation}

\item \label{assump:containment} \textit{Iterate containment.}
The iterates remain in $U$: $u^{(k)} \in U$ for all $k \geq 0$.
\end{enumerate}

Assumptions~\ref{assump:smooth} and~\ref{assump:lipschitz} are standard in Newton-type convergence theory~\cite{Deuflhard2011}. The Lipschitz condition~\ref{assump:lipschitz} enters the analysis through the error bounds of Theorem~\ref{thm:error_bounds}. Assumption~\ref{assump:invertible} ensures that the continuous linearized problems are well-posed and it is needed only for the error bounds in Theorem~\ref{thm:error_bounds}. Assumption~\ref{assump:bounded} provides a uniform bound on the forward operator and is needed for all three main results. Finally, assumption~\ref{assump:containment} can be verified a posteriori or ensured by globalization strategies. Under the \textit{smallness condition} of  Theorem~\ref{thm:nk_convergence} it holds automatically and the iterates remain in a ball about~$u^{*}$ as a consequence of local convergence rather than as a standing hypothesis.

\subsubsection{Stationarity and residual bounds}
\label{sec:stationarity_bounds}

The first result establishes that when the projection residual is bounded away from zero, the \textsc{LiL-Q} iteration cannot converge to the exact root.

\begin{theorem}
\label{thm:stationary_residual}
Suppose~\ref{assump:smooth},~\ref{assump:bounded}, and~\ref{assump:containment} hold and that $\|r^{(k)}\|_{Y} \geq \varepsilon_{L} > 0$ for all $k \geq 0$.
If the sequence $\{u^{(k)}\}$ converges to a limit $\hat{u} \in U$, then
\begin{equation}
\label{eq:stationary_result}
\|\mathcal{N}(\hat{u})\|_{Y} = \lim_{k \to \infty} \|r^{(k)}\|_{Y} \geq \varepsilon_{L} > 0.
\end{equation}
In particular, $\hat{u} \neq u^{*}$.
\end{theorem}

\begin{proof}
Since $u^{(k)} \to \hat{u}$, the increments satisfy $\delta u^{(k)} \to 0$.
Rearranging~\eqref{eq:projected_newton} gives
\begin{equation}
\label{eq:thm1_rearrange}
\mathcal{N}(u^{(k)}) = r^{(k)} - \mathcal{N}'(u^{(k)})(\delta u^{(k)}).
\end{equation}
Applying~\ref{assump:bounded} to the second term on the right-hand side yields
\begin{equation}
\label{eq:thm1_vanish}
\|\mathcal{N}'(u^{(k)})(\delta u^{(k)})\|_{Y} \;\leq\; \Gamma\,\|\delta u^{(k)}\|_{X} \;\to\; 0
\quad \text{as } k \to \infty.
\end{equation}
By continuity of $\mathcal{N}$~\ref{assump:smooth}, the left-hand side of~\eqref{eq:thm1_rearrange} converges to $\mathcal{N}(\hat{u})$, so
\begin{equation}
\label{eq:thm1_limit}
\mathcal{N}(\hat{u}) = \lim_{k \to \infty} r^{(k)}.
\end{equation}
The lower bound $\|r^{(k)}\|_{Y} \geq \varepsilon_{L}$ passes to the limit, giving $\|\mathcal{N}(\hat{u})\|_{Y} \geq \varepsilon_{L} > 0$.
\end{proof}

Theorem~\ref{thm:stationary_residual} demonstrates that the limiting nonlinear residual is constrained by the accuracy of the linearized subproblems.
Because $r^{(k)}$ is the linearized residual minimized at each iteration, the approximation error of the basis, manifested through~$\varepsilon_{L}$, propagates directly to the nonlinear PDE residual at convergence.
When $\varepsilon_{L} > 0$, this residual floor represents an accuracy limit that cannot be overcome by further iterations. Reducing it requires enriching the subspace~$V_{P}$, and hence, a different approximator. In the limiting case $\varepsilon_{L} = 0$, the hypothesis of Theorem~\ref{thm:stationary_residual} is not satisfied, the floor vanishes, and the iteration can converge to the exact root~$u^{*}$. The linear elasticity experiment in Section~\ref{sec:elasticity} illustrates this scenario.

The second result provides two-sided bounds on the nonlinear residual at any iteration, and not just at the limit.

\begin{theorem}
\label{thm:residual_bounds}
Suppose~\ref{assump:smooth},~\ref{assump:bounded}, and~\ref{assump:containment} hold and that the projection residual satisfies $0 < \varepsilon_{L} \leq \|r^{(k)}\|_{Y} \leq \varepsilon_{U}$ for all $k \geq 0$.
Then
\begin{equation}
\label{eq:residual_two_sided}
\varepsilon_{L} - \Gamma\,\|\delta u^{(k)}\|_{X}
\;\leq\; \|\mathcal{N}(u^{(k)})\|_{Y}
\;\leq\; \varepsilon_{U} + \Gamma\,\|\delta u^{(k)}\|_{X}.
\end{equation}
If the sequence converges so that $\delta u^{(k)} \to 0$, the limiting nonlinear residual satisfies $\varepsilon_{L} \leq \|\mathcal{N}(\hat{u})\|_{Y} \leq \varepsilon_{U}$.
\end{theorem}

\begin{proof}
From~\eqref{eq:projected_newton},
\begin{equation}
\label{eq:thm2_identity}
\mathcal{N}(u^{(k)}) = r^{(k)} - \mathcal{N}'(u^{(k)})(\delta u^{(k)}).
\end{equation}

\noindent\textit{Upper bound.}
The triangle inequality and~\ref{assump:bounded} give
\begin{align}
\|\mathcal{N}(u^{(k)})\|_{Y}
&\leq \|r^{(k)}\|_{Y} + \|\mathcal{N}'(u^{(k)})(\delta u^{(k)})\|_{Y} \notag \\
&\leq \|r^{(k)}\|_{Y} + \Gamma\,\|\delta u^{(k)}\|_{X} \notag \\
&\leq \varepsilon_{U} + \Gamma\,\|\delta u^{(k)}\|_{X}.
\label{eq:thm2_upper}
\end{align}

\noindent\textit{Lower bound.}
The reverse triangle inequality gives
\begin{align}
\|\mathcal{N}(u^{(k)})\|_{Y}
&\geq \|r^{(k)}\|_{Y} - \|\mathcal{N}'(u^{(k)})(\delta u^{(k)})\|_{Y} \notag \\
&\geq \|r^{(k)}\|_{Y} - \Gamma\,\|\delta u^{(k)}\|_{X} \notag \\
&\geq \varepsilon_{L} - \Gamma\,\|\delta u^{(k)}\|_{X}.
\label{eq:thm2_lower}
\end{align}
The asymptotic bounds follow by taking $k \to \infty$.
\end{proof}

During early iterations when the increments are large, the bounds~\eqref{eq:residual_two_sided} may be loose. As the iteration approaches stationarity and the increments contract, the bounds tighten around (sandwich) the true nonlinear residual. A practical consequence is that the linearized residual $\mathbf{R}_{\mathrm{lin}}^{(k)}$ serves as a computationally inexpensive proxy for the true PDE residual. When this quantity stagnates, the nonlinear residual has also stagnated within the band $[\varepsilon_{L},\, \varepsilon_{U}]$, and further iterations cannot improve the solution. Basis enrichment is required for a more accurate result.

\subsubsection{Error bounds}
\label{sec:error_bounds}

The residual bounds of Theorem~\ref{thm:residual_bounds} constrain the nonlinear residual in the codomain~$Y$.
The following result bounds the iteration error $e^{(k)} = u^{(k)} - u^{*}$ in the domain~$X$, establishing that the iterates are confined to an annular region around~$u^{*}$.
We first record an identity relating the error at iteration $k+1$ to the projection residual and the nonlinear remainder.

\begin{lemma}
\label{lem:error_identity}
Let $e^{(k)} = u^{(k)} - u^{*}$ and write
\begin{equation}
\label{eq:taylor_expansion}
\mathcal{N}(u^{(k)}) = \mathcal{N}'(u^{(k)})(e^{(k)}) + w^{(k)}.
\end{equation}
Under~\ref{assump:smooth} and~\ref{assump:lipschitz}, the remainder obeys
\begin{equation}
\label{eq:remainder_bound}
\|w^{(k)}\|_{Y} \leq \frac{L}{2}\,\|e^{(k)}\|_{X}^{2},
\end{equation}
and the iteration error satisfies
\begin{equation}
\label{eq:error_identity}
\mathcal{N}'(u^{(k)})(e^{(k+1)}) = r^{(k)} - w^{(k)}.
\end{equation}
\end{lemma}

\begin{proof}
Equation~\eqref{eq:taylor_expansion} defines $w^{(k)}$.
By the fundamental theorem of calculus and $\mathcal{N}(u^{*}) = 0$,
\begin{equation}
\label{eq:ftc}
\mathcal{N}(u^{(k)}) = \mathcal{N}(u^{(k)}) - \mathcal{N}(u^{*}) = \int_{0}^{1} \mathcal{N}'\bigl(u^{*} + t\,e^{(k)}\bigr)(e^{(k)})\,dt,
\end{equation}
so that
\begin{equation}
\label{eq:remainder_integral}
w^{(k)} = \int_{0}^{1} \bigl[\mathcal{N}'\bigl(u^{*} + t\,e^{(k)}\bigr) - \mathcal{N}'(u^{(k)})\bigr](e^{(k)})\,dt.
\end{equation}
Since $u^{*} + t\,e^{(k)} - u^{(k)} = (t-1)\,e^{(k)}$, assumption~\ref{assump:lipschitz} gives $\|\mathcal{N}'(u^{*} + t\,e^{(k)}) - \mathcal{N}'(u^{(k)})\|_{\mathcal{L}(X,Y)} \leq L\,(1-t)\,\|e^{(k)}\|_{X}$, and therefore
\begin{equation}
\label{eq:remainder_estimate}
\|w^{(k)}\|_{Y} \leq L\,\|e^{(k)}\|_{X}^{2}\int_{0}^{1}(1-t)\,dt = \frac{L}{2}\,\|e^{(k)}\|_{X}^{2},
\end{equation}
which is~\eqref{eq:remainder_bound}.
Substituting~\eqref{eq:taylor_expansion} into the projection equation~\eqref{eq:projected_newton} and using $\delta u^{(k)} + e^{(k)} = e^{(k+1)}$ together with the linearity of $\mathcal{N}'(u^{(k)})$ yields~\eqref{eq:error_identity}.
\end{proof}

Before bounding the iteration error, we make precise the quantity~$\varepsilon_{L}$ that has appeared informally as a lower bound on the projection residual.
Because the right-hand side of the quasilinearized subproblem~\eqref{eq:quasilinear_subproblem} depends on the current iterate through $\mathcal{N}'(u^{(k)})$ and $\mathcal{N}(u^{(k)})$, the quantity that the least-squares solve over $V_{P}$ can attain varies from one outer iteration to the next.
To obtain bounds that hold uniformly along the trajectory, we track the best-approximation residual achieved at each iterate and take its smallest and largest values over the sequence, denoted $\varepsilon_{L}$ and $\varepsilon_{U}$ respectively.

\begin{definition}
\label{def:eps_L}
For a fixed basis spanning $V_{P}$ and a trajectory $\{u^{(k)}\}$ remaining in $U$, define the smallest residual~\eqref{eq:projected_newton} attainable at iteration~$k$ by any element of~$V_{P}$, written through the increment $w - u^{(k)}$ as,
\begin{equation}
\eta^{(k)} \;\coloneqq\; \inf_{w \in V_{P}}\; \bigl\| \mathcal{N}'(u^{(k)})(w - u^{(k)}) + \mathcal{N}(u^{(k)}) \bigr\|_{Y},
\end{equation}
the lowest attained residual at iteration k,
\begin{equation}
\label{eq:eps_L_definition}
\varepsilon_{L}(P) \;\coloneqq\; \inf_{k \geq 0} \eta^{(k)},
\end{equation}
and the largest attained residual at iteration k,
\begin{equation}
\label{eq:eps_U_definition}
\varepsilon_{U}(P) \;\coloneqq\; \sup_{k \geq 0} \eta^{(k)}.
\end{equation}
Since the least-squares solve (discretely) realizes $\|r^{(k)}\|_{Y} = \eta^{(k)}$ at each step, we have $\varepsilon_{L}(P) \leq \|r^{(k)}\|_{Y} \leq \varepsilon_{U}(P)$ for all~$k$.
\end{definition}

Note that the supremum $\varepsilon_{U}$ is finite whenever the trajectory remains in a bounded region of~$U$, a property guaranteed by the convergence result established next. Next we develop bounds on the iterate error itself.

\begin{theorem}
\label{thm:error_bounds}
Suppose~\ref{assump:smooth}--\ref{assump:containment} hold and that $0 < \varepsilon_{L} \leq \|r^{(k)}\|_{Y} \leq \varepsilon_{U}$ for all $k \geq 0$.
Then
\begin{equation}
\label{eq:error_two_sided}
\frac{1}{\Gamma}\Bigl(\varepsilon_{L} - \frac{L}{2}\,\|e^{(k)}\|_{X}^{2}\Bigr)
\;\leq\; \|e^{(k+1)}\|_{X}
\;\leq\; \Lambda\Bigl(\varepsilon_{U} + \frac{L}{2}\,\|e^{(k)}\|_{X}^{2}\Bigr).
\end{equation}
The lower bound is informative once the iterate is close enough that $\tfrac{L}{2}\,\|e^{(k)}\|_{X}^{2} < \varepsilon_{L}$.
When the quadratic remainder terms are negligible, the error is asymptotically confined to the annular region
\begin{equation}
\label{eq:error_annulus}
\frac{\varepsilon_{L}}{\Gamma}
\;\leq\; \|e^{(k+1)}\|_{X}
\;\leq\; \Lambda\,\varepsilon_{U}.
\end{equation}
\end{theorem}

\begin{proof}
Both bounds follow from the error identity~\eqref{eq:error_identity}.

\noindent\textit{Upper bound.}
Applying $\mathcal{N}'(u^{(k)})^{-1}$ to both sides of~\eqref{eq:error_identity} and taking norms gives
\begin{align}
\|e^{(k+1)}\|_{X}
&= \|\mathcal{N}'(u^{(k)})^{-1}\bigl(r^{(k)} - w^{(k)}\bigr)\|_{X} \notag \\
&\leq \|\mathcal{N}'(u^{(k)})^{-1}\|_{\mathcal{L}(Y,X)}\bigl(\|r^{(k)}\|_{Y} + \|w^{(k)}\|_{Y}\bigr) \notag \\
&\leq \Lambda\Bigl(\varepsilon_{U} + \frac{L}{2}\,\|e^{(k)}\|_{X}^{2}\Bigr),
\label{eq:thm3_upper}
\end{align}
where the last inequality uses~\ref{assump:invertible} and Lemma~\ref{lem:error_identity}.

\noindent\textit{Lower bound.}
Taking norms on~\eqref{eq:error_identity} and applying the reverse triangle inequality gives
\begin{align}
\|\mathcal{N}'(u^{(k)})(e^{(k+1)})\|_{Y}
&\geq \|r^{(k)}\|_{Y} - \|w^{(k)}\|_{Y} \notag \\
&\geq \varepsilon_{L} - \frac{L}{2}\,\|e^{(k)}\|_{X}^{2}.
\label{eq:thm3_lower_step1}
\end{align}
Since $\|\mathcal{N}'(u^{(k)})(e^{(k+1)})\|_{Y} \leq \Gamma\,\|e^{(k+1)}\|_{X}$ by~\ref{assump:bounded}, dividing by $\Gamma$ yields
\begin{equation}
\label{eq:thm3_lower}
\|e^{(k+1)}\|_{X}
\;\geq\; \frac{1}{\Gamma}\Bigl(\varepsilon_{L} - \frac{L}{2}\,\|e^{(k)}\|_{X}^{2}\Bigr).
\end{equation}
The asymptotic annulus~\eqref{eq:error_annulus} is obtained by evaluating the bounds when $\frac{L}{2}\|e^{(k)}\|_{X}^{2}$ is negligible relative to $\varepsilon_{L}$ and~$\varepsilon_{U}$.
\end{proof}

The upper and lower bounds in Theorem~\ref{thm:error_bounds} rely on different operator constants. While the upper bound uses the invertibility constant~$\Lambda$ from~\ref{assump:invertible}, the lower bound uses the forward operator bound~$\Gamma$ from~\ref{assump:bounded}. In general $\Lambda \neq 1/\Gamma$, so the annulus has nontrivial width. In the discrete formulation (collocation system), these constants are related to singular value bounds on the collocation matrix. That is, $\Gamma$ corresponds to $\sigma_{\max}(\mathbf{A}^{(k)})$ and $\Lambda$ corresponds to $1/\sigma_{\min}(\mathbf{A}^{(k)})$, and the width of the annulus is governed by the condition number $\kappa(\mathbf{A}^{(k)})$.

\subsubsection{Local convergence of the iteration}
\label{sec:local_convergence}

Theorems~\ref{thm:stationary_residual} and~\ref{thm:residual_bounds} characterize the behavior of the iteration \emph{under the hypothesis} that the sequence $\{u^{(k)}\}$ converges, and Assumption~\ref{assump:containment} likewise presumes that the iterates remain in~$U$.
We now remove these presumptions so that under the smoothness and regularity conditions~\ref{assump:smooth}-\ref{assump:bounded} together with a \textit{smallness condition} relating the operator constants to the subspace accuracy, the iteration converges locally and the iterates remain in~$U$ automatically.
The argument is an inexact Newton-Kantorovich analysis~\cite{Deuflhard2011,Kantorovich1982,DemboEisenstatSteihaug1982} in which the inexactness is precisely the projection residual~$r^{(k)}$ left by the finite-dimensional least-squares solve.

\begin{theorem}
\label{thm:nk_convergence}
Suppose~\ref{assump:smooth}--\ref{assump:bounded} hold on a ball $B(u^{*}, \rho) \subset U$, that the projection residual satisfies $\|r^{(k)}\|_{Y} \leq \varepsilon_{U}$ for every iterate in this ball, and that the smallness condition,
\begin{equation}
\label{eq:nk_smallness}
2\Lambda^{2} L\, \varepsilon_{U} \;\leq\; 1
\end{equation}
holds. Let $0 < t_{-} \leq t_{+}$ denote the two roots of the scalar quadratic $\tfrac{\Lambda L}{2} t^{2} - t + \Lambda\varepsilon_{U}$. If the initial error satisfies $\|e^{(0)}\|_{X} < t_{+}$ and $t_{+} \leq \rho$, then every iterate remains in $B(u^{*}, \rho)$, and the error decreases monotonically,
\begin{equation}
\label{eq:nk_monotone}
\|e^{(k+1)}\|_{X} \leq \|e^{(k)}\|_{X}
\end{equation}
 to the stable fixed point,
\begin{equation}
\label{eq:nk_limit}
\lim_{k \to \infty} \|e^{(k)}\|_{X} \;\leq\; t_{-} \;\leq\; 2\Lambda\varepsilon_{U}.
\end{equation}
 
\end{theorem}

\begin{proof}
Applying $\mathcal{N}'(u^{(k)})^{-1}$ to the error identity~\eqref{eq:error_identity}, taking the $X$-norm, and using~\ref{assump:invertible},~\ref{assump:lipschitz}, Lemma~\ref{lem:error_identity}, and $\|r^{(k)}\|_{Y} \leq \varepsilon_{U}$ gives, exactly as in the upper bound~\eqref{eq:thm3_upper} of Theorem~\ref{thm:error_bounds},
\begin{equation}
\label{eq:nk_recurrence}
\|e^{(k+1)}\|_{X}
\;\leq\; \Lambda\Bigl(\|r^{(k)}\|_{Y} + \tfrac{L}{2}\|e^{(k)}\|_{X}^{2}\Bigr)
\;\leq\; a\,\|e^{(k)}\|_{X}^{2} + b,
\end{equation}
where,
\[ a \coloneqq \tfrac{\Lambda L}{2}, \]
and,
\[ b \coloneqq \Lambda\varepsilon_{U}.\]
This is a quadratic majorant of the standard Newton-Kantorovich form $t_{k+1} \leq a t_{k}^{2} + b$. The smallness condition~\eqref{eq:nk_smallness} implies exactly that $4ab \leq 1$, under which the majorant map $t \mapsto a t^{2} + b$ has the two fixed points $t_{\pm}$ and, by the classical scalar majorant argument~\cite{Kantorovich1982,Deuflhard2011,DemboEisenstatSteihaug1982}, the sequence majorizing $\|e^{(k)}\|_{X}$ decreases monotonically to $t_{-}$ from any starting value below $t_{+}$.
Consequently $\|e^{(k)}\|_{X} < t_{+} \leq \rho$ for all $k$, so the iterates remain in $B(u^{*},\rho) \subset U$. The limit bound $t_{-} \leq 2\Lambda\varepsilon_{U}$ follows from $t_{-} = (1 - \sqrt{1 - 4ab})/(2a)$ and $1 - \sqrt{1-s} \leq s$ for $s \in [0,1]$.
\end{proof}

Theorem~\ref{thm:nk_convergence} characterizes the convergence behavior that Theorems~\ref{thm:stationary_residual} and~\ref{thm:residual_bounds} had assumed. The iteration contracts quadratically until the error reaches the plateau~$t_{-} \leq 2\Lambda\varepsilon_{U}$, which coincides with the outer radius~$\Lambda\varepsilon_{U}$ of the error annulus~\eqref{eq:error_annulus} established above. The limiting error is therefore squeezed between the basis-imposed floor~$\varepsilon_{L}/\Gamma$ and this plateau, consistent across all three results.
The smallness condition~\eqref{eq:nk_smallness} couples the operator constants $\Lambda$ and $L$ with the inner-solve accuracy~$\varepsilon_{U}$. Subsequently, a sufficiently rich subspace, for which $\varepsilon_{U}$ is small, guarantees local convergence, whereas an impoverished subspace with large~$\varepsilon_{U}$ may violate~\eqref{eq:nk_smallness} and forfeit the guarantee. This dependence explains why the choice of representation, examined empirically in Section~\ref{sec:numerical_experiments}, governs not only the attainable accuracy but the reliability of convergence itself. The single-digit outer-iteration counts observed throughout the experiments are the practical signature of the quadratic phase predicted here.

\begin{remark}
Collocation matrices built from global spectral bases are typically ill-conditioned, and $\kappa(\mathbf{A}^{(k)})$ grows with the number of parameters~$P$.
This growth bears on the round-off contribution to the accuracy rather than on the approximation-theoretic floor. A backward-stable least-squares solve yields an error controlled by $\varepsilon_{L}$ together with a round-off perturbation of order $\kappa(\mathbf{A}^{(k)})\,\epsilon_{\mathrm{mach}}$, where $\epsilon_{\mathrm{mach}}$ is the unit round-off. The attainable accuracy is therefore set by whichever of these two terms is larger. When the collocation matrix is full column rank (the generic case), the least-squares minimizer is unique, and if the exact solution lies in the span of the basis ($\varepsilon_{L}=0$) while $\kappa(\mathbf{A}^{(k)})\,\epsilon_{\mathrm{mach}}$ is small, machine-precision recovery follows. When $\varepsilon_{L}>0$ dominates the round-off term, accuracy is governed by the basis rather than by conditioning. In the extreme case that the matrix becomes numerically rank-deficient, the column-pivoted QR factorization remains well-defined, and its rank-revealing pivoting returns the minimum-norm solution and confines the influence of the trailing singular values that fall below the rank tolerance. The empirical condition numbers, ranks, and their dependence on~$P$ across all benchmarks are examined in Section~\ref{sec:conditioning_results}.
\end{remark}

The three theorems above collectively establish that enriching the basis~$V_{P}$ is the mechanism for improving accuracy. A richer (more representative) basis reduces~$\varepsilon_{L}$, which lowers both the residual floor (Theorem~\ref{thm:residual_bounds}) and the inner radius of the error annulus (Theorem~\ref{thm:error_bounds}), without inflating the number of outer iterations required to reach the floor. Additional outer iterations beyond stationarity cannot improve the solution. These predictions are validated empirically in the residual band plots of Section~\ref{sec:numerical_experiments}.

\subsubsection{Accuracy and the choice of basis}
\label{sec:basis_accuracy}

The preceding results show that the attainable accuracy is governed by~$\varepsilon_{L}(P)$, the best-approximation residual of the subspace.
We now establish how~$\varepsilon_{L}(P)$ behaves as the subspace is enriched, and identify the property of the representation that controls the rate of improvement.
This is the point at which the \textsc{LiL} framework, which admits any representation linear in its trainable parameters, exposes a design choice that is invisible to the convergence analysis but decisive for performance.

\begin{proposition}
\label{prop:monotone}
Let $\{\phi_{p}\}_{p \geq 1}$ generate a nested family of subspaces $V_{1} \subset V_{2} \subset \cdots \subset V_{P} \subset V_{P+1} \subset \cdots$.
Then $\varepsilon_{L}(P+1) \leq \varepsilon_{L}(P)$ for every~$P$.
\end{proposition}

\begin{proof}
For each fixed iterate~$u^{(k)}$, the best-approximation residual $\eta^{(k)}$ in~\eqref{eq:eps_L_definition} is an infimum over a larger set when $V_{P}$ is replaced by $V_{P+1} \supseteq V_{P}$, so it cannot increase. Since $\varepsilon_{L}(P) = \inf_{k} \eta^{(k)}$ is monotone in the pointwise values $\eta^{(k)}$, the inequality is preserved.
\end{proof}

Proposition~\ref{prop:monotone} requires only nestedness which is a property shared by orthogonal polynomial families (Chebyshev, Legendre), trigonometric (Fourier) systems, and hierarchical spline bases, for example. It is independent of orthogonality.
It guarantees that enriching such a basis never degrades the attainable accuracy, and through Theorems~\ref{thm:residual_bounds} and~\ref{thm:error_bounds} that the residual floor~$\varepsilon_{L}$ and the inner radius~$\varepsilon_{L}/\Gamma$ of the error annulus are non-increasing in~$P$. This monotonicity is not automatic for arbitrary \textsc{LiL} representations. For instance, a family of randomly sampled features need not be nested as the feature count grows, and the corresponding best-approximation error need not decrease monotonically. This is a point we revisit in the basis-sensitivity study of Section~\ref{sec:burgers}. The \emph{rate} at which $\varepsilon_{L}(P)$ decreases depends on the smoothness of the linearized solution relative to the approximation power of the basis, and here the choice of an orthogonal or trigonometric basis is decisive.

\begin{remark}
When the basis is an orthogonal polynomial or trigonometric (spectral) family and the linearized solution remains smooth along the iteration, classical spectral approximation theory~\cite{CanutoHussainiQuarteroniZang2006,Trefethen2000} gives best-approximation errors that decay rapidly with $P$, i.e.,
\begin{equation}
\label{eq:spectral_rate}
\varepsilon_{L}(P) \;\lesssim\; P^{-s} 
\end{equation}
for solutions in $H^{s}$, and,
\begin{equation}
\label{eq:spectral_rate2}
\varepsilon_{L}(P) \;\lesssim\; \rho^{-P}\ (\rho > 1) 
\end{equation}
for analytic solutions. Subsequently, the residual floor and the inner radius of the error annulus inherit the same algebraic or, in the analytic case, geometric decay.
This rate is contingent on smoothness. When the linearized solution develops steep internal layers such as for advection-dominated problems, the decay degrades, although the monotonicity guaranteed by Proposition~\ref{prop:monotone} still holds.
Randomly sampled features possess neither this spectral rate nor the monotonicity of Proposition~\ref{prop:monotone}.
\end{remark}

\subsubsection{Computational complexity}
\label{sec:flop_analysis}

At each outer iteration~$k$, the \textsc{LiL-Q} method assembles the collocation matrix $\mathbf{A}^{(k)} \in \mathbb{R}^{N \times P}$ and right-hand side $\mathbf{f}^{(k)} \in \mathbb{R}^{N}$, then solves the resulting linear least-squares problem.
Assembling one column of $\mathbf{A}^{(k)}$ requires evaluating the linearized operator on a single basis function at all $N$ collocation points, costing $\mathcal{O}(N)$ operations per column and $\mathcal{O}(NP)$ for the full matrix. If a Householder QR factorization is used to solve the least squares problem, $2NP^{2} - \tfrac{2}{3}P^{3}$ floating-point operations are required. Since typically $N \gg P$ in the overdetermined setting, the leading term is $2NP^{2}$. The cost over $K$ outer iterations is therefore
\begin{equation}
\label{eq:flop_lilq}
\mathcal{F}_{\textsc{LiL-Q}} = \mathcal{O}(K N P^{2}).
\end{equation}
Comparing~\eqref{eq:flop_lilq} with the \textsc{NiL-N} cost from~\eqref{eq:flop_niln} gives the ratio
\begin{equation}
\label{eq:flop_ratio}
\frac{\mathcal{F}_{\textsc{LiL-Q}}}{\mathcal{F}_{\textsc{NiL-N}}} \;\approx\; \mathcal{O}\!\left(\frac{K_{\textsc{LiL-Q}}\, P}{K_{\textsc{NiL-N}}}\right)
\end{equation}
for \textbf{equal parameter and collocation counts}.
Consequently, \textsc{LiL-Q} is cheaper whenever $K_{\textsc{LiL-Q}}\, P < K_{\textsc{NiL-N}}$. The additional factor of~$P$ per iteration reflects the cost of the direct QR solve relative to a single gradient step. In our computational examples, it is offset by the reduction from thousands of gradient steps to single-digit (occasionally low-double-digit) outer iterations.

\section{Numerical Experiments}
\label{sec:numerical_experiments}

We present a sequence of numerical experiments using seven benchmark problems spanning a range of complexity from illustrative model problems to highly heterogeneous Darcy flow and coupled nonlinear systems. We include alongside \textbf{\textsc{LiL-Q}} and the standard PINN (\textbf{\textsc{NiL-N}}) two intermediate formulations that each address only one of the two sources of nonlinearity, enabling their individual contributions to be distinguished. Our primary comparison axis is with other PINN formulations, which constitute the natural baseline for a PINN training methodology. A finite-volume reference solution is included for the Darcy benchmark, where a direct comparison against a classical solver is most informative, and the structural relationship between \textsc{LiL-Q} and classical mesh-based methods is examined separately in the Discussion. Broader head-to-head comparisons against classical solvers across the full benchmark suite are deferred to future work.
 
\begin{enumerate}
    \item \textbf{\textsc{NiL-Q}}: The governing PDE is quasilinearized as in \textsc{LiL-Q}, but the solution at each iteration is represented by a \textsc{NiL} network. Subsequently, the collocation residual equations are architecturally nonlinear and the loss is minimized via gradient-based optimization. This formulation removes the physical nonlinearity within each iteration, but, since the network outputs are nonlinear in $\boldsymbol{\theta}$, the stationarity condition remains a nonlinear system requiring iterative optimization.
    
    \item \textbf{\textsc{LiL-N}}: The solution is represented by a \textsc{LiL} representation, but the original nonlinear operator $\mathcal{N}$ is retained without quasilinearization. The residual $\mathbf{R}(\boldsymbol{\beta}) = \mathcal{N}(\hat{u}(\cdot;\,\boldsymbol{\beta}))$ evaluated at collocation points defines a nonlinear loss despite the linear architectural parameterization, because $\mathcal{N}$ is nonlinear.
\end{enumerate}

\subsection{Experimental setup}
\label{ssec:experiment_setup}
Unless otherwise stated, \textsc{NiL} implementations use fully connected feedforward networks with two hidden layers and hyperbolic tangent activation functions, while \textsc{LiL} implementations use tensor-product expansions of univariate basis functions whose specific configuration varies by problem and is detailed in each experiment. The \textsc{NiL} and \textsc{LiL} representations are constructed with approximately equal parameter counts for each experiment (fixing the parameter $P$ that appears in the FLOP accounting of Section~\ref{sec:flop_analysis}). The three nonlinear formulations (\textsc{NiL-N}, \textsc{NiL-Q}, \textsc{LiL-N}) are optimized using L-BFGS with strong Wolfe line search conditions. \textsc{LiL-Q} subproblems are solved via column-pivoted QR with orthogonal factorization. For \textsc{NiL-Q}, the network parameters from outer iteration $k$ serve as the initial guess for iteration $k{+}1$. All methods are initialized to the zero function. Collocation points are distributed with an approximate $10{:}1$ ratio of samples to parameters, allocated roughly $85\%$ to the interior and $15\%$ to auxiliary conditions.
 
In practice, different weight factors are assigned to the PDE, boundary, and initial-condition residuals to balance their contributions. We define the weighted mean-squared error
\begin{equation}
\label{eq:mse_metric}
\operatorname{MSE}
  = \lambda_{\mathrm{pde}}\,\frac{1}{N_{\mathrm{pde}}} \sum_{i \in \mathcal{I}_{\mathrm{pde}}} R_i^2
  \;+\; \lambda_{\mathrm{bc}}\,\frac{1}{N_{\mathrm{bc}}} \sum_{i \in \mathcal{I}_{\mathrm{bc}}} R_i^2
  \;+\; \lambda_{\mathrm{ic}}\,\frac{1}{N_{\mathrm{ic}}} \sum_{i \in \mathcal{I}_{\mathrm{ic}}} R_i^2,
\end{equation}
where $\mathcal{I}_{\mathrm{pde}}$, $\mathcal{I}_{\mathrm{bc}}$, and $\mathcal{I}_{\mathrm{ic}}$ denote the index sets of collocation points for each residual type, $N_{\mathrm{pde}}$, $N_{\mathrm{bc}}$, and $N_{\mathrm{ic}}$ are their respective cardinalities, and $R_i = \mathcal{N}(\hat{u})(\mathbf{x}_i)$ is the pointwise residual from~\eqref{eq:loss_function}.
Unless otherwise stated, the weights are $\lambda_{\mathrm{pde}} = 1$, $\lambda_{\mathrm{bc}} = 10$, $\lambda_{\mathrm{ic}} = 10$.
The target MSE for each experiment comparing the four PINN training formulations is set to the value achieved by \textsc{LiL-Q}, enabling matched-accuracy comparisons. Training is terminated when $\operatorname{MSE}$ falls below the target tolerance, when the coefficient update norm $\|\boldsymbol{\beta}^{(k+1)} - \boldsymbol{\beta}^{(k)}\|_2$ falls below a specified tolerance (for quasilinear methods), or when the iteration budget is exhausted. 

We record the following performance metrics: the final MSE attained, the number of iterations required to reach the target tolerance, and the wall-clock runtime. Additionally, we collect LiL coefficient matrix condition numbers and numerical rank.
All experiments were implemented in Python using PyTorch for the gradient-based methods, executed on an NVIDIA RTX~5080 GPU with double-precision arithmetic.
For the \textsc{LiL-Q} solves, we used an Intel Core Ultra~9~275HX CPU ($8$ performance and $16$ efficiency cores, $24$ threads) with $32$~GB system RAM, using double-precision (\texttt{float64}) arithmetic. Each quasilinear step calls the \texttt{scipy.linalg.lstsq} routine with the \texttt{gelsy} driver and multithreaded OpenBLAS.
Because the gradient-based methods and \textsc{LiL-Q} run on different hardware, we treat hardware-agnostic metrics (iteration counts, parameter counts, and the floating-point operation analysis of Section~\ref{sec:flop_analysis}) as the primary basis for comparing computational cost, and report wall-clock runtimes as corroborating, indicative figures. We note that this division is conservative for \textsc{LiL-Q} as its direct solves execute on CPU while the gradient-based baselines use GPU acceleration. It is expected that a GPU-accelerated implementation could widen the reported runtime gaps.

\subsection{2D Bratu}
\label{sec:bratu}
 
The Bratu equation, introduced as the worked example in Section~\ref{sec:bratu_example}, is a nonlinear elliptic PDE that serves as a standard steady model problem. It admits two parameter-dependent solution branches that coalesce for the critical coefficient value $\lambda_c \approx 6.808$, and admits no solution for $\lambda >\lambda_c$~\cite{Boyd2001}. The problem is posed on the unit square $\Omega = (0,1)^2$ with homogeneous Dirichlet boundary conditions and we fix $\lambda = 6.2$.

The \textsc{LiL} representations are \emph{linear Fourier networks} comprised of a single-hidden-layer with frozen frequencies and an output-layer with trained weights. Concretely, for a selected $P$, let $p_d^2 = P$ and define the univariate features,
\[
  \{\phi_k(\xi)\}_{k=1}^{p_d}
  = \{\cos(n\pi\xi)\}_{n=0}^{n_c-1}
    \;\cup\;
    \{\sin(n\pi\xi)\}_{n=1}^{n_s},
\]
where,
\[
  n_c = \Big\lfloor\frac{p_d-1}{2}\Big\rfloor + 1,
\]
and $n_s = p_d - n_c$. The network output for a given $\left(x,y\right) \in \left[0,1\right]^2$ is the tensor product,
\begin{equation}\label{eq:bratu-fourier-network}
  \hat u(x,y;\boldsymbol\beta)
  = \sum_{k=1}^{p_d}\sum_{\ell=1}^{p_d}
    \beta_{k\ell}\,\phi_k(x)\,\phi_\ell(y),
\end{equation}
with $\boldsymbol\beta\in\mathbb{R}^{P}$ and $P=p_d^2$. We conduct simulations using $p_d\in\{5,10,15\}$, and hence $P\in\{25,100,225\}$.

\begin{figure}
    \centering
    \includegraphics[width=\linewidth]{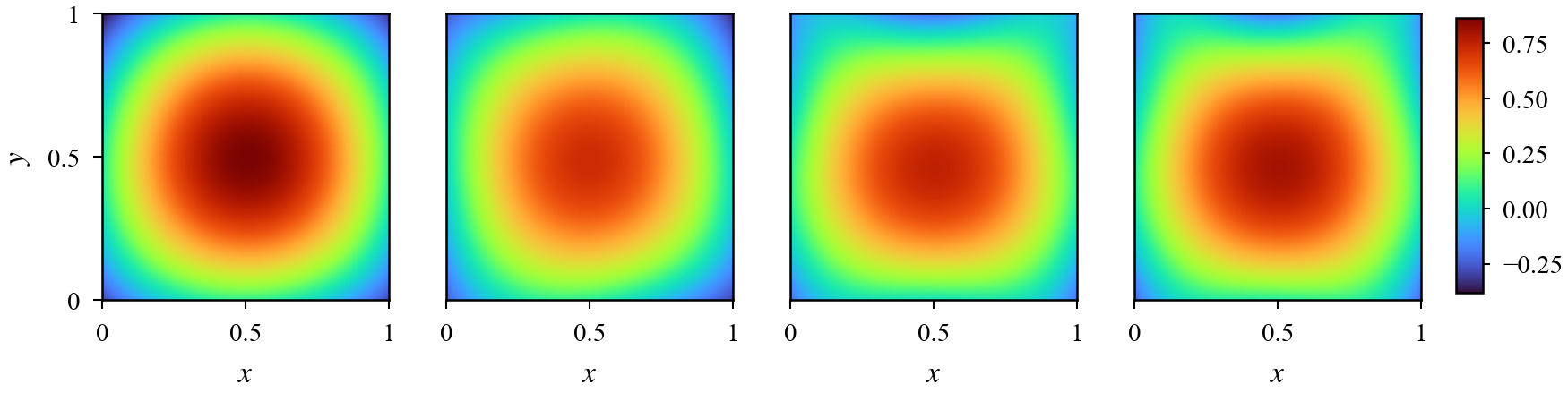} \\
    \includegraphics[width=\linewidth]{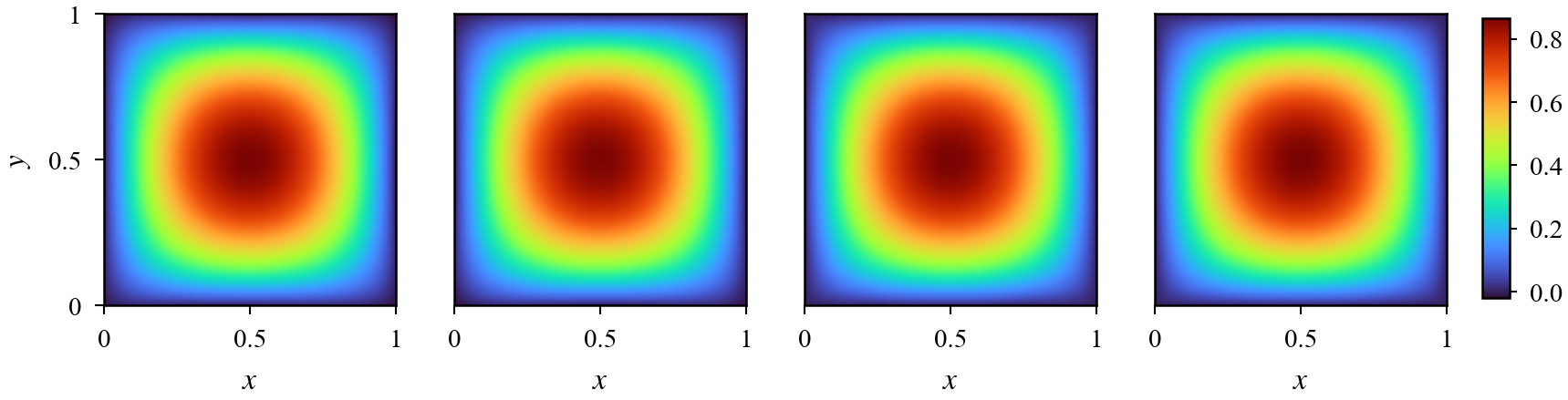}
    \caption{Bratu equation: simulated solution fields using $P = 25$ (top row) and $P = 225$ (bottom row) for \textsc{NiL-N}, \textsc{NiL-Q}, \textsc{LiL-N} and \textsc{LiL-Q} from left to right.}
    \label{fig:bratu_solutions}
\end{figure}

\begin{table}
\centering
\caption{Bratu equation: iteration counts and runtimes for various methods and number of network parameters $P$. An asterisk (*) indicates no convergence after the maximum allowable number of iterations.}
\label{tab:bratu_results}
\small
\begin{tabular}{@{}ccrrrrrrrr@{}}
\toprule
& & \multicolumn{4}{c}{Iterations} & \multicolumn{4}{c}{Runtime (s)} \\
\cmidrule(lr){3-6} \cmidrule(lr){7-10}
$P$ & $\operatorname{MSE}_{\mathrm{tol}}$ & \textsc{NiL-N} & \textsc{NiL-Q} & \textsc{LiL-N} & \textsc{LiL-Q} & \textsc{NiL-N} & \textsc{NiL-Q} & \textsc{LiL-N} & \textsc{LiL-Q} \\
\midrule
25  & $2.5 \times 10^{-1}$   & 76    & 309     & 33    & 2 & 1.11 & 6.49  & 0.91  & 0.0004 \\
100 & $1 \times 10^{-4}$   & 10000*  & 10000*  & 1597   & 3 & 220.9   & 237.7   & 24.63  & 0.032 \\
225 & $2.5 \times 10^{-7}$ & 10000*  & 10000*  & 7748  & 3 & 279.6   & 273.4   & 163.7   & 0.14 \\
\bottomrule
\end{tabular}
\end{table}

\begin{figure}
\centering
\includegraphics[width=\linewidth]{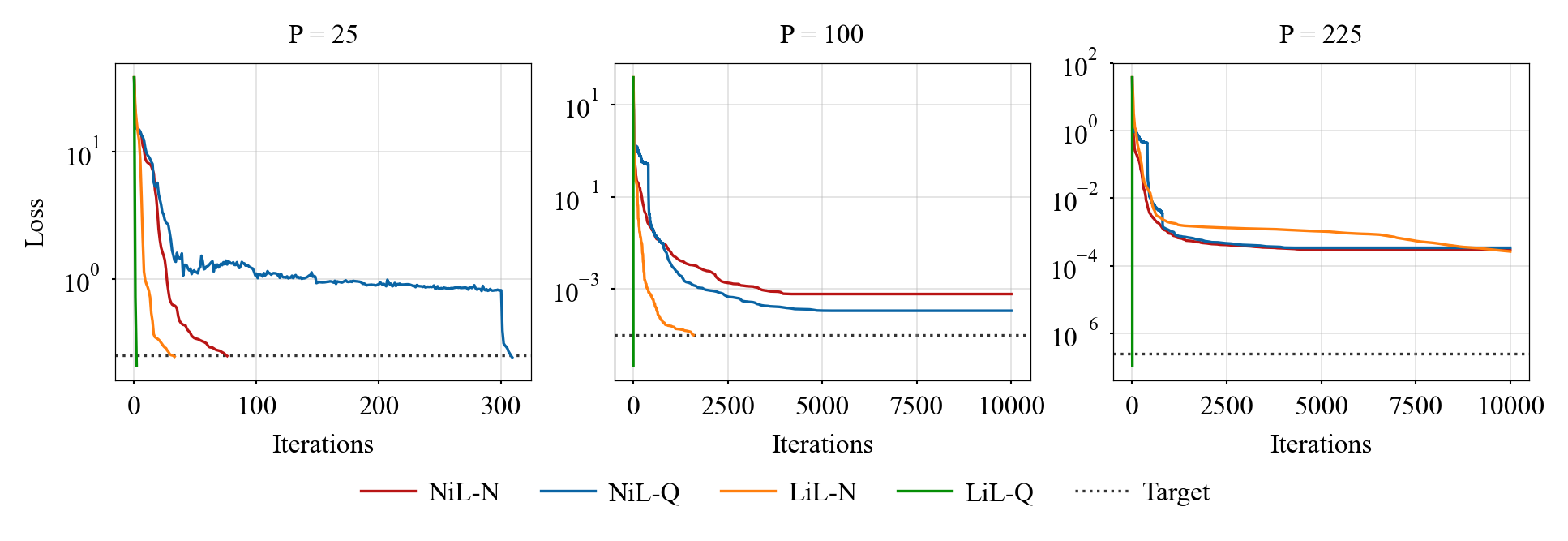}
\caption{Bratu equation: convergence histories for various methods using networks with $P \in \{25, 100, 225\}$ parameters.}
\label{fig:bratu_convergence_by_size}
\end{figure}

Fig.~\ref{fig:bratu_solutions} presents solution fields obtained using $P = 25$ and $225$ and Table~\ref{tab:bratu_results} summarizes the performance results. The convergence histories are presented in Fig.~\ref{fig:bratu_convergence_by_size}. With $P = 25$ trainable parameters, all four methods reach the target tolerance, yet the iteration counts across the various methods span two orders of magnitude: \textsc{LiL-Q} converges in 2 iterations, \textsc{LiL-N} in 33, \textsc{NiL-N} in 76, and \textsc{NiL-Q} in 309. As the number of parameters is increased to $P = 100$ and $P = 225$, both \textsc{NiL} formulations exhaust the 10{,}000-iteration budget without satisfying the convergence criterion due to stagnation which in practice, is typically addressed by modifying the learning rate schedule (not done here, to maintain a fair comparison across methods without heuristic interventions). \textsc{LiL-Q} converges in 3 iterations at both enriched basis sizes, while \textsc{LiL-N} converges but with a rapidly growing iteration count (1{,}597 and 7{,}748 respectively). This suggests that the physical nonlinearity preserves optimization difficulty, even when the architectural nonlinearity has been removed. The wall-clock runtimes corroborate this iteration-count disparity with \textsc{LiL-Q} converging in subseconds across all three configurations which is two to three orders of magnitude faster than the \textsc{NiL} formulations. As noted above, this is quantitatively highly implementation and platform dependent, but this margin is likely to be conservative given that \textsc{LiL-Q} runs on CPU while the \textsc{NiL} baselines run on GPU.

\begin{figure}
\centering
\includegraphics[width=\linewidth]{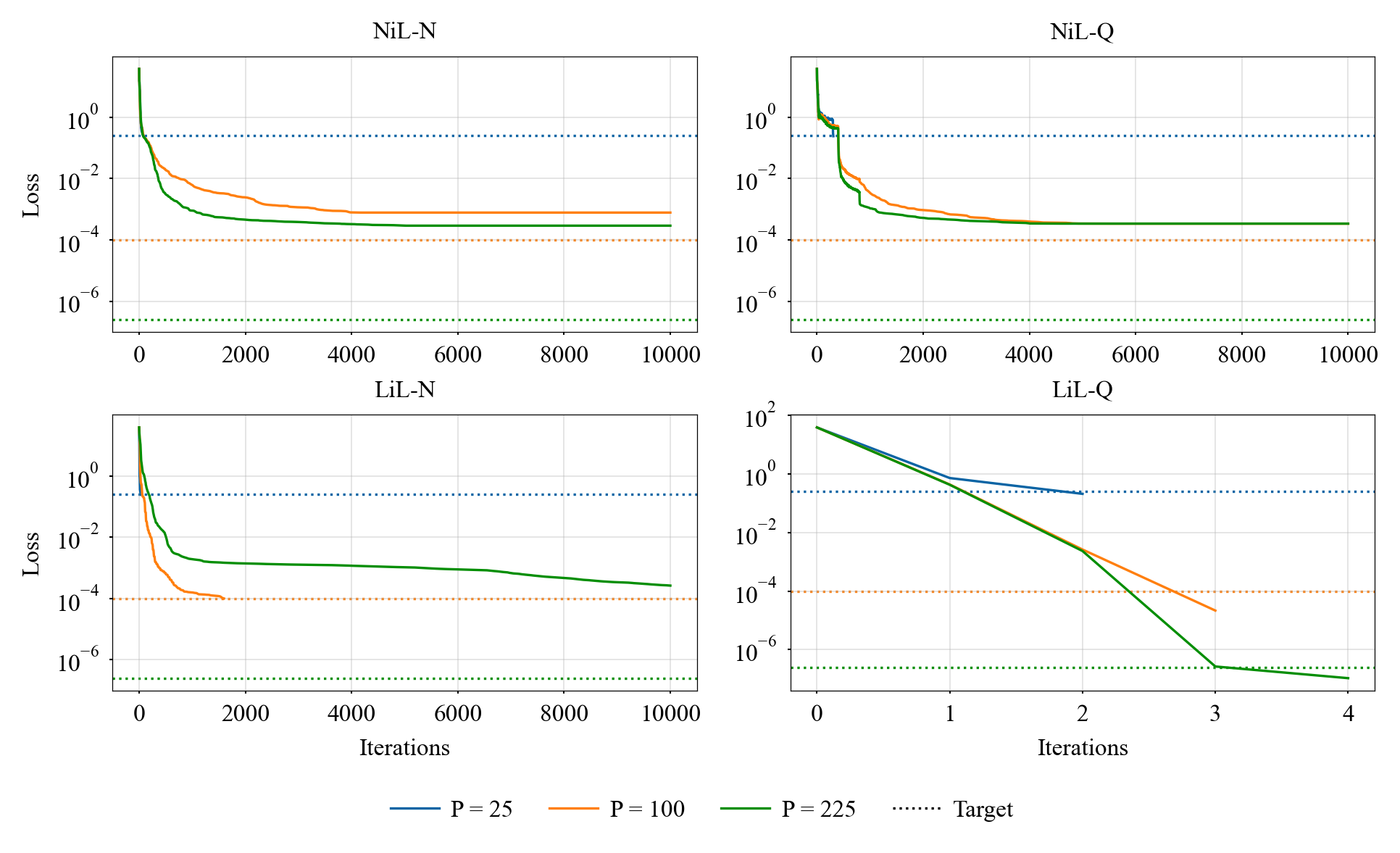}
\caption{Bratu equation: convergence histories organized by method.}
\label{fig:bratu_convergence_by_method}
\end{figure}

Fig.~\ref{fig:bratu_convergence_by_method} presents the convergence histories organized by method. For \textsc{NiL-N} and \textsc{NiL-Q}, increasing the basis size from $P = 25$ to $P = 225$ progressively slows convergence where the loss curves flatten and develop extended plateaus characteristic of increasingly ill-conditioned nonconvex landscapes. The \textsc{LiL-Q} panel presents a qualitatively different picture where the convergence trajectories for all three basis sizes are nearly superimposed, each outer iteration reducing the loss by several orders of magnitude, with the iteration count remaining at 2 to 3 regardless of $P$ while the loss floor decreases from $3 \times 10^{-2}$ at $P = 25$ to $1.5 \times 10^{-6}$ at $P = 225$. This confirms that basis enrichment lowers the achievable accuracy floor without necessarily inflating the number of iterations required to reach it.

\begin{figure}
    \centering
    \includegraphics[width=\linewidth]{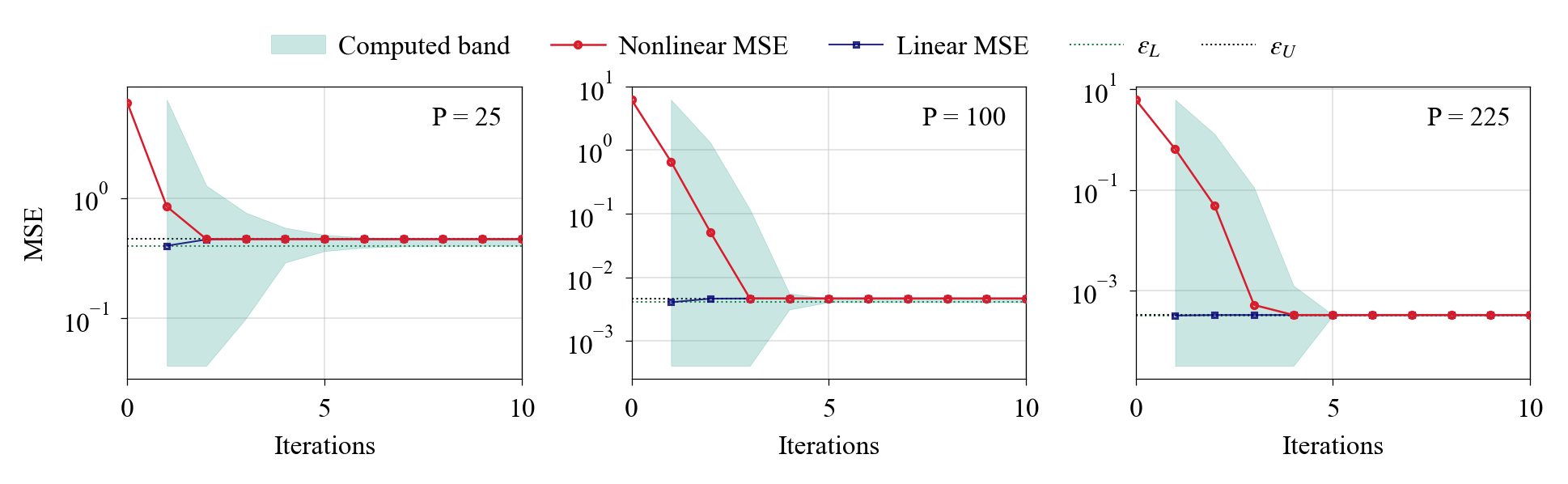}
    \caption{Bratu equation: comparison of linearized and nonlinear residual norms for \textsc{LiL-Q} at $P \in \{25, 100, 225\}$. The asymptotic agreement between the two quantities corroborates the convergence analysis of Section~\ref{sec:properties}.}
    \label{fig:bratu_residual_bounds}
\end{figure}

Fig.~\ref{fig:bratu_residual_bounds} presents the evolution of the linearized collocation residual norm $\|\mathbf{R}_{\mathrm{lin}}^{(k)}\|$ and the nonlinear collocation residual norm $\|\mathbf{R}^{(k)}\|$ across quasilinear iterations for each basis size (LiL-Q).
To qualitatively illustrate the convergence theory of Section~\ref{sec:properties}, we also compute estimates of the bounds in~\eqref{eq:residual_two_sided} at each iteration as well as the minimum and maximum linear residual square error. The nonlinear residual consistently falls within the estimated band and as the updates contract, the nonlinear residual tracks the linearized residual and both quantities settle within the predicted bounds. At stagnation, further iterations cannot improve the solution, and basis enrichment is required.

\subsection{Viscous Burgers}
\label{sec:burgers}
 
Next, we consider the time-dependent advection-diffusion problem with Burgers flux function,
\begin{equation}
\label{eq:burgers}
\begin{cases}
    \displaystyle\frac{\partial u}{\partial t} + u\frac{\partial u}{\partial x} = \nu \frac{\partial^2 u}{\partial x^2}, & (x,t) \in (-1,1) \times (0,1], \\
    u(x,t) = -\sin(\pi x), \; & x\in (-1,1), t=0, \\
    u(x,t) = 0 & x = \pm 1, \; t = 0
\end{cases},
\end{equation}
and $\nu = 0.1$. The Bellman-Kalaba quasilinearization yields the sequence of linear subproblems defined by,

\begin{equation}
\label{eq:quasilin_burgers}
\begin{cases}
    \frac{\partial u^{(k+1)}}{\partial t} + u^{(k)}\frac{\partial u^{(k+1)}}{\partial x} + u^{(k+1)}\frac{\partial u^{(k)}}{\partial x} - \nu \frac{\partial^2 u^{(k+1)}}{\partial x^2} = u^{(k)}\frac{\partial u^{(k)}}{\partial x}, & (x,t) \in (-1,1) \times (0,1], \\
    u^{(k+1)}(x,t) = -\sin(\pi x), \; & x\in (-1,1), t=0, \\
    u^{(k+1)}(x,t) = 0 & x = \pm 1, \; t = 0
\end{cases},
\end{equation}
 
The primary \textsc{LiL} representation for the four-method comparison is a linear Fourier network constructed as a tensor product of $\sin(n\pi\xi_x)$ activation for $n = 1, \ldots, p_x$ in the spatial direction and the mixed Fourier set $\{\cos(m\pi\xi_t)\, \cup\, \sin(m\pi\xi_t)\}$ in the temporal direction, where $\xi_x$ and $\xi_t$ map the physical coordinates to the unit interval.
The sine spatial activations vanish at $x = \pm 1$, enforcing the homogeneous Dirichlet boundary conditions by construction.
All four methods were tested at $p_x \cdot p_t = P \in \{25, 100, 225, 400, 625\}$ trainable parameters.
 
\begin{table}
\centering
\caption{Burgers equation: iteration counts and runtimes. An asterisk (*) indicates no convergence.}
\label{tab:burgers_iterations}
\small
\begin{tabular}{@{}ccrrrrrrrr@{}}
\toprule
& & \multicolumn{4}{c}{Iterations} & \multicolumn{4}{c}{Runtime (s)} \\
\cmidrule(lr){3-6} \cmidrule(lr){7-10}
$P$ & Target $\operatorname{MSE}$ & \textsc{NiL-N} & \textsc{NiL-Q} & \textsc{LiL-N} & \textsc{LiL-Q} & \textsc{NiL-N} & \textsc{NiL-Q} & \textsc{LiL-N} & \textsc{LiL-Q} \\
\midrule
25  & $6 \times 10^{-2}$ & 214    & 99    & 487   & 2 & 4.32  & 11.8  & 1.37  & 0.0005 \\
100 & $1 \times 10^{-3}$ & 1047   & 1796   & 93   & 3 & 26.5  & 45.6  & 2.19  & 0.036 \\
225 & $2 \times 10^{-5}$ & 4946   & 7500*   & 1863*  & 3 & 164.1  & 231.8  & 31.9  & 0.12 \\
400 & $1 \times 10^{-7}$ & 5902*  & 10000*  & 1322*  & 4 & 216.4  & 338.5   & 22.5  & 0.51 \\
625 & $5 \times 10^{-9}$   & 5578*  & 10000*  & 3373*  & 4 & 142.1  & 232.2   & 159.5  & 1.26 \\
\bottomrule
\end{tabular}
\end{table}
 
Table~\ref{tab:burgers_iterations} reports iteration counts and runtimes.
Similar to the observations for the Bratu problem, both \textsc{NiL} formulations exhaust the iteration budget with increasing network size, and \textsc{LiL-N} also fails to converge beyond $P = 100$, whereas \textsc{LiL-Q} converges in 2 to 4 iterations at every basis size.
Fig.~\ref{fig:burgers_convergence_by_method} reinforces this. The \textsc{NiL-N} and \textsc{NiL-Q} curves spread with increasing $P$, developing extended plateaus at larger basis sizes, while the \textsc{LiL-Q} curves remain nearly superimposed across all five basis dimensions.
Fig.~\ref{fig:burgers_residual_bounds} confirms that the linearized and nonlinear residuals track each other within the predicted band at each $P$, consistent with observed behavior in Bratu. The solution fields for $P = 625$ are shown in Fig.~\ref{fig:burgers_comparison}.
 
\begin{figure}
    \centering
    \includegraphics[width=0.9\linewidth]{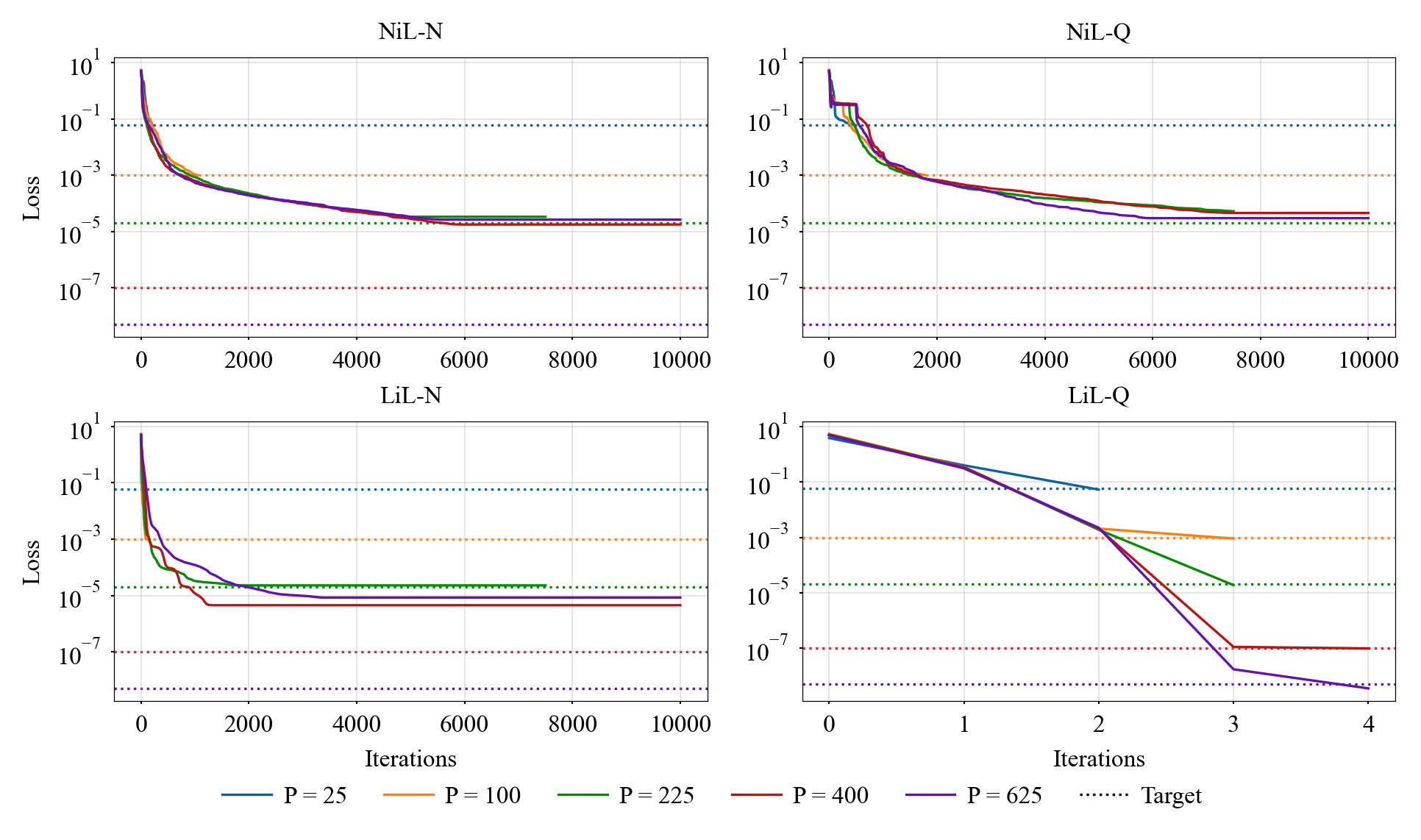}
    \caption{Burgers equation: convergence histories organized by method.
    Each panel shows all five basis sizes for a single formulation.}
    \label{fig:burgers_convergence_by_method}
\end{figure}

\begin{figure}
    \centering
    \includegraphics[width=0.9\linewidth]{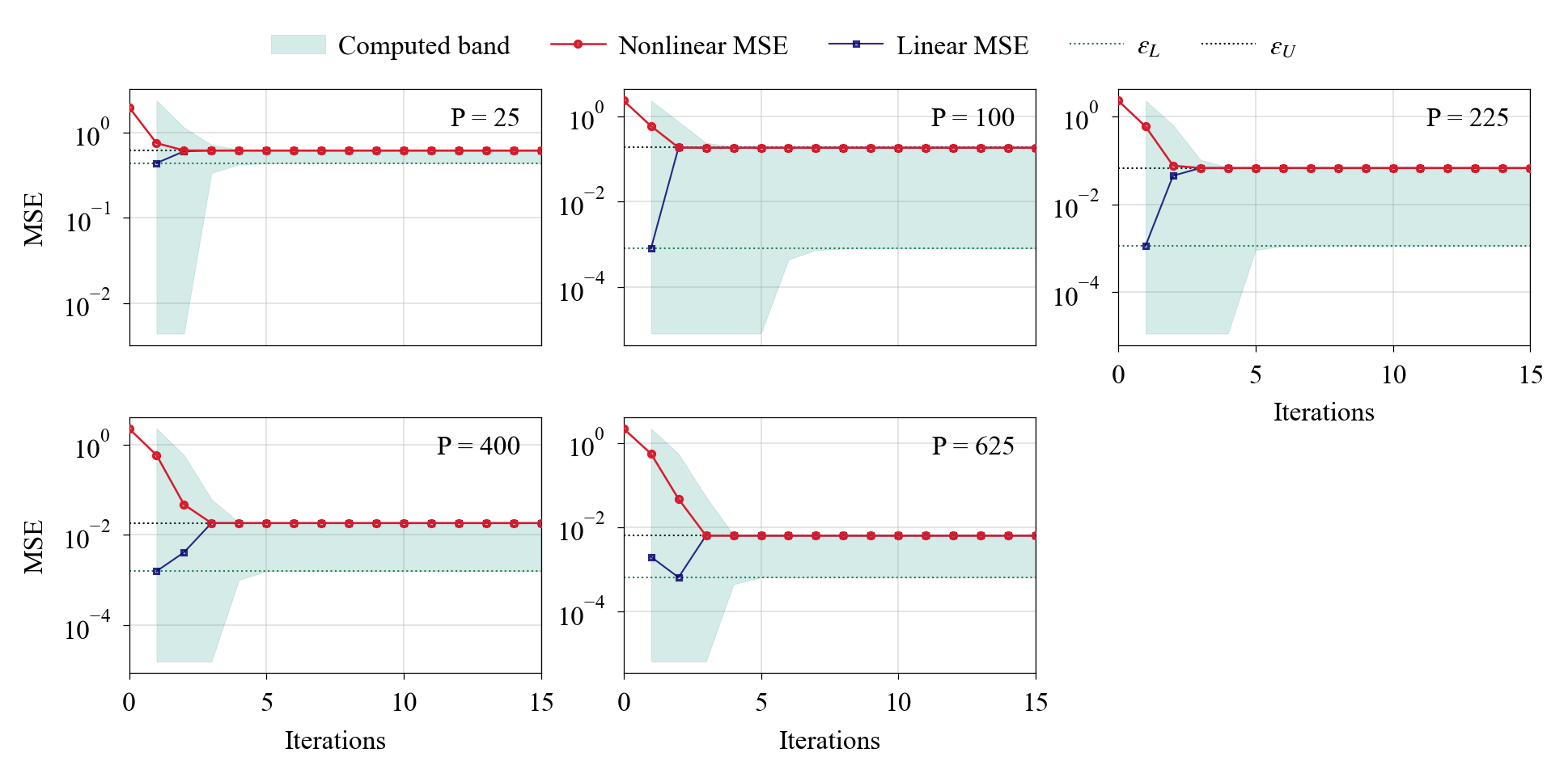}
    \caption{Burgers equation: comparison of linearized and nonlinear residual norms for \textsc{LiL-Q} at $P \in \{25, 100, 225, 400, 625\}$.}
    \label{fig:burgers_residual_bounds}
\end{figure}
 
\begin{figure}
    \centering
    \setlength{\tabcolsep}{2pt}
    \begin{tabular}{@{}cccc@{}}
        \includegraphics[width=0.25\linewidth]{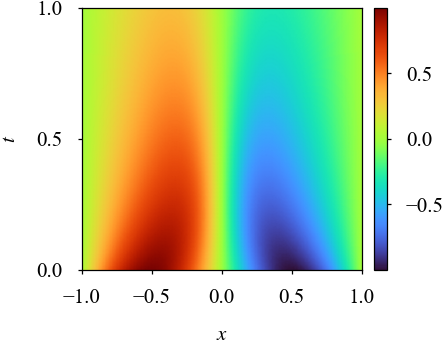} &
        \includegraphics[width=0.25\linewidth]{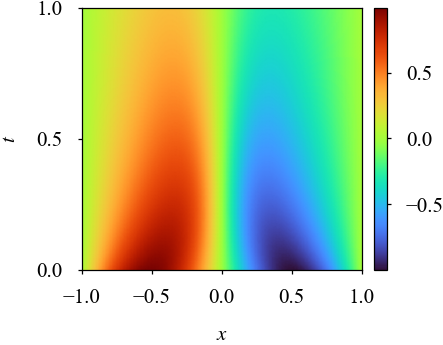} &
        \includegraphics[width=0.25\linewidth]{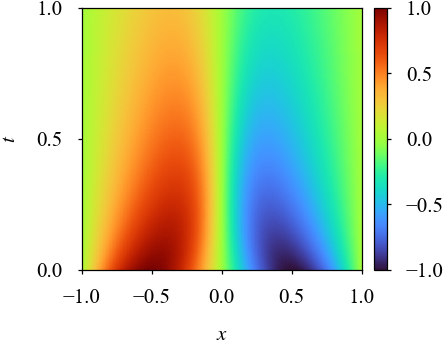} &
        \includegraphics[width=0.25\linewidth]{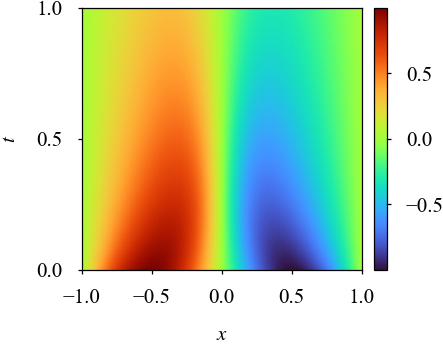} \\[-0.1cm]
        \small (a) \textsc{NiL-N} & \small (b) \textsc{NiL-Q} & \small (c) \textsc{LiL-N} & \small (d) \textsc{LiL-Q}
    \end{tabular}
    \caption{Burgers equation: predicted solution fields at $P = 625$.}
    \label{fig:burgers_comparison}
\end{figure}
 
The four-method comparison demonstrates that \textsc{LiL-Q} converges reliably at every basis size. The theory of Section~\ref{sec:properties} predicts that the nonlinear residual floor is set by the best-approximation linear solution residual $\varepsilon_L$ of the chosen basis. That in turn, depends on how well the basis functions span the solution manifold. To test this, we compare eight \textsc{LiL} configurations on the viscous Burgers equation, all with fixed $P = 625$, and spanning linear Fourier networks with different activation selections, Chebyshev polynomial expansions, and an extreme learning machine (ELM) with random $\tanh$ activations.
Each configuration was run for five outer iterations which was sufficient for all to reach stagnation. Table~\ref{tab:basis_configs} reports the final $\operatorname{MSE}$ values in ascending order and Fig.~\ref{fig:basis_convergence} shows the convergence histories.

The final $\operatorname{MSE}$ values span over eight orders of magnitude, ranging from $1.7 \times 10^{-8}$ (Sin$\times$Cheb) to $5.0$ (Sin$\times$Sin).
The worst-performing configuration (Sin$\times$Sin) is due to a lack of ability to represent the initial condition. That is because $\sin(m\pi t)$ vanishes at $t = 0$ for all $m$, making the initial condition $u(x,0) = -\sin(\pi x)$ unrepresentable. Subsequently, the smallness condition of Theorem~\ref{thm:nk_convergence} is likely violated. The two best-performing configurations comprise Sine activations in $x$ that vanish at $x = \pm 1$, enforcing the homogeneous Dirichlet conditions by construction, while Chebyshev polynomials in $t$ retain the flexibility to represent the non-zero initial condition.
Boundary condition alignment alone does not suffice, however. The Cheb$\times$Cheb approximator can represent arbitrary boundary values through superposition, yet it stagnates three orders of magnitude above Sin$\times$Cheb because its column space does not efficiently cover the solution manifold globally.
Fourier$\times$Cheb and Fourier$\times$Fourier achieve nearly identical losses ($1.3 \times 10^{-5}$), indicating that for this problem the spatial activations drive the performance hierarchy while the temporal choice is secondary.
These results support the theoretical prediction that within the \textsc{LiL-Q} framework, the basis governs the achievable accuracy. Incorporating problem knowledge (e.g., boundary behavior or expected smoothness) into the basis selection influences $\varepsilon_L$ for a given parameter count.
 
\begin{table}
\centering
\caption{\textsc{LiL} configurations for the viscous Burgers equation ($P = 625$, $\nu = 0.1$), ordered by final $\operatorname{MSE}$.
Each row specifies the spatial and temporal activations whose tensor product forms the two-dimensional network.}
\label{tab:basis_configs}
\small
\begin{tabular}{@{}lllc@{}}
\toprule
Configuration & Spatial ($x$) & Temporal ($t$) & Final $\operatorname{MSE}$ \\
\midrule
Sin $\times$ Cheb
    & $\sin(n\pi x)$, $n{=}1{:}25$
    & $T_j(t)$, $j{=}0{:}24$
    & $1.7 \times 10^{-8}$ \\[4pt]
AugSin $\times$ Cheb
    & $\{1,\, x,\, \sin(n\pi x)_{n=1}^{23}\}$
    & $T_j(t)$, $j{=}0{:}24$
    & $4.3 \times 10^{-8}$ \\[4pt]
Fourier $\times$ Cheb
    & $\{\cos,\sin\}(n\pi x)$, $n{=}0{:}12$
    & $T_j(t)$, $j{=}0{:}24$
    & $1.3 \times 10^{-5}$ \\[4pt]
Fourier $\times$ Fourier
    & $\{\cos,\sin\}(n\pi x)$, $n{=}0{:}12$
    & $\{\cos,\sin\}(m\pi t)$, $m{=}0{:}12$
    & $1.3 \times 10^{-5}$ \\[4pt]
Cheb $\times$ Cheb
    & $T_i(x)$, $i{=}0{:}24$
    & $T_j(t)$, $j{=}0{:}24$
    & $7.8 \times 10^{-5}$ \\[4pt]
ELM ($\tanh$)
    & \multicolumn{2}{l}{625 random $\tanh$ neurons}
    & $5.0 \times 10^{-2}$ \\[4pt]
Cos $\times$ Cheb
    & $\cos(n\pi x)$, $n{=}0{:}24$
    & $T_j(t)$, $j{=}0{:}24$
    & $1.7 \times 10^{-1}$ \\[4pt]
Sin $\times$ Sin
    & $\sin(n\pi x)$, $n{=}1{:}25$
    & $\sin(m\pi t)$, $m{=}1{:}25$
    & $5.0$ \\
\bottomrule
\end{tabular}
\end{table}
 
\begin{figure}
\centering
\includegraphics[width=0.7\linewidth]{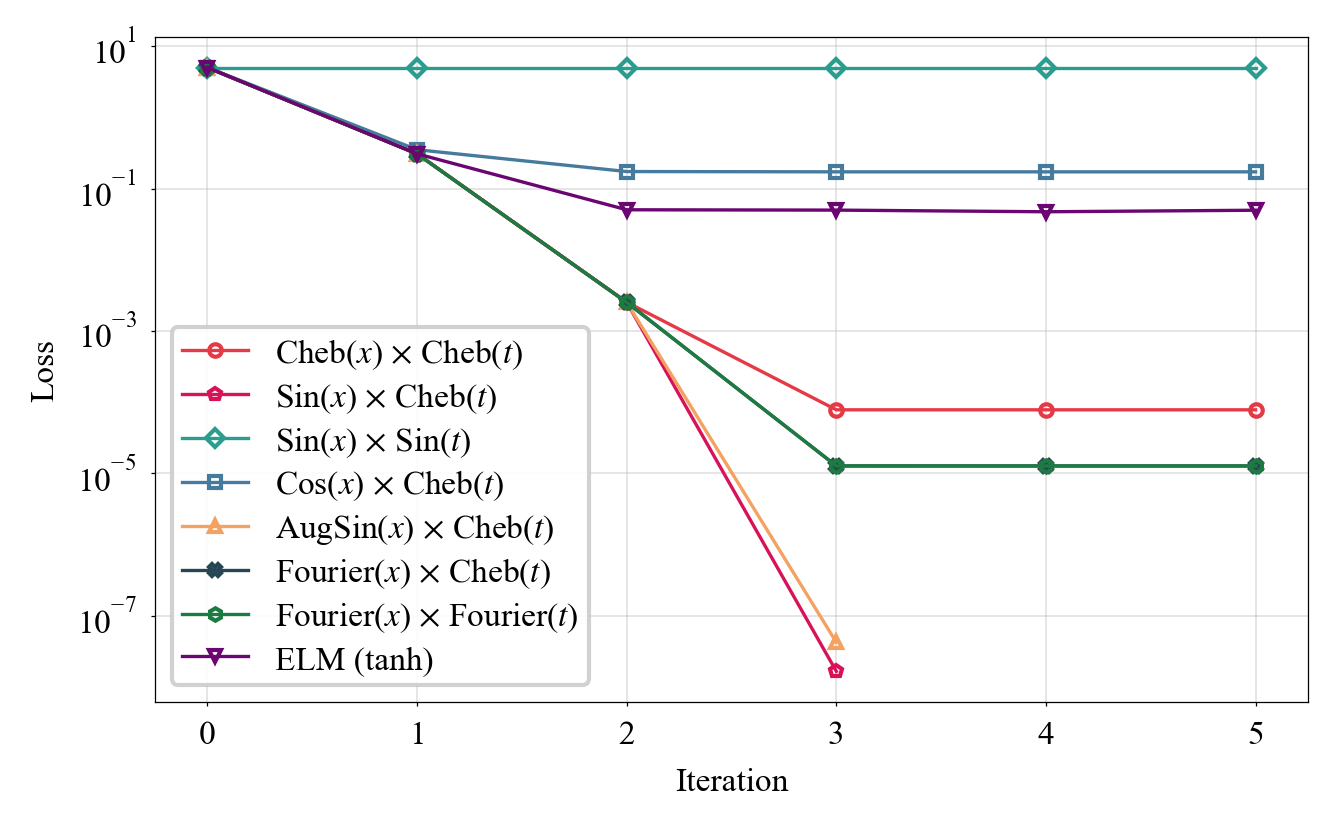}
\caption{Convergence histories for eight \textsc{LiL} configurations on the viscous Burgers equation ($P = 625$).}
\label{fig:basis_convergence}
\end{figure}

\subsection{Buckley-Leverett}
\label{sec:buckley_leverett}
 
The Buckley-Leverett equation models two-phase immiscible displacement in porous media~\cite{Buckley1942}. The \textit{fractional flow} flux function introduces stiff nonlinearity characterized by a mobility ratio dependent inflection point and sonic points (maxima or minima) with non-zero gravity number~\cite{nlnbl}. 
Following Fraces and Tchelepi~\cite{FracesTchelepi2021}, we regularize the equation with an artificial diffusion term:
\begin{equation}
\label{eq:bl_governing}
\frac{\partial S}{\partial t} + \frac{\partial f(S)}{\partial x} - \varepsilon\frac{\partial^2 S}{\partial x^2} = 0, \qquad (x,t) \in (0,1) \times (0,T],
\end{equation}
with $S(0,t) = 1$ and $S(1,t) = 0$.
The fractional flow function incorporates viscous and gravitational effects through
\begin{equation}
\label{eq:bl_flux}
f(S) = \frac{S^2\bigl(1 - N_g(1-S)^2\bigr)}{S^2 + M(1-S)^2},
\end{equation}
where $M$ is the viscosity ratio and $N_g$ is the gravity number~\cite{LiTchelepi2015}.
We consider two parameter sets that produce solutions of increasing complexity: a viscous flow case ($N_g = 0$, $M = 0.5$, $\varepsilon = 0.1$, $T = 0.4$, $S(x,0) = e^{-10x}$) and a gravity-influenced case ($N_g = -5$, $M = 1.0$, $\varepsilon = 0.1$, $T = 0.175$, $S(x,0) = 1 - (1 + e^{-100(x-0.4)})^{-1}$).
Quasilinearization of~\eqref{eq:bl_governing} yields, for both cases,
\begin{equation}
\label{eq:bl_linearized}
\frac{\partial S^{(k+1)}}{\partial t} + \frac{\partial}{\partial x}\!\left(f'(S^{(k)})\, S^{(k+1)}\right) - \varepsilon \frac{\partial^2 S^{(k+1)}}{\partial x^2} = \frac{\partial}{\partial x}\!\left[f'(S^{(k)})\, S^{(k)} - f(S^{(k)})\right].
\end{equation}
 
This experiment tests whether the convergence behavior observed on Bratu and Burgers extends to the qualitatively more severe rational nonlinearity of the Buckley-Leverett flux, and examines the accuracy limitations imposed by steep saturation fronts.
The \textsc{LiL} representation for both cases is a linear Fourier network with full Fourier activations ($\{\cos, \sin\}$) in both spatial and temporal directions.
All four formulations were tested at $P \in \{64, 256, 576, 1024\}$ trainable parameters, corresponding to 8, 16, 24, and 32 hidden nodes per coordinate direction.
 
\begin{table}
\centering
\caption{Buckley-Leverett (viscous flow, $N_g = 0$): iteration counts and runtimes.
An asterisk (*) indicates the iteration budget was exhausted without convergence.}
\label{tab:bl_case1_iterations}
\small
\begin{tabular}{@{}ccrrrrrrrr@{}}
\toprule
& & \multicolumn{4}{c}{Iterations} & \multicolumn{4}{c}{Runtime (s)} \\
\cmidrule(lr){3-6} \cmidrule(lr){7-10}
$P$ & Target $\operatorname{MSE}$ & \textsc{NiL-N} & \textsc{NiL-Q} & \textsc{LiL-N} & \textsc{LiL-Q} & \textsc{NiL-N} & \textsc{NiL-Q} & \textsc{LiL-N} & \textsc{LiL-Q} \\
\midrule
64    & $8 \times 10^{-2}$ & 179 & 342 & 259 & 11    & 5.11 & 9.32  & 5.41  & 0.052 \\
256   & $1.5 \times 10^{-2}$ & 926    & 972    & 1107 & 4   & 32.49   & 34.38   & 19.1  & 0.26 \\
576   & $7.5 \times 10^{-4}$   & 1785  & 3463  & 15000*    & 4      & 86.95   & 151.1   & 264   & 1.33 \\
1024  & $7.5 \times 10^{-5}$   & 20000*  & 20000*  & 20000*  & 4 & 999.6   & 798.3   & 375.8   & 4.16 \\
\bottomrule
\end{tabular}
\end{table}
 
Table~\ref{tab:bl_case1_iterations} reports the iteration counts and runtimes for the viscous displacement case. Convergence histories are shown in Fig.~\ref{fig:bl_convergence_by_method}.
\textsc{NiL-Q} requires more iterations than \textsc{NiL-N} at every basis size, consistent with the observation from Burgers that linearizing the PDE alone does not simplify the remaining nonconvex training problem and may compound the cost by resetting optimizer state at each outer step.
\textsc{LiL-N} converges at $P = 64$ and $P = 256$ but exhausts the budget at $P \geq 576$, confirming that physical nonlinearity introduces increasingly severe optimization difficulty as the basis grows.
Only \textsc{LiL-Q} converges at every size, requiring 4 to 11 outer iterations.
Fig.~\ref{fig:bl_residual_bounds} confirms that the linearized and nonlinear residuals track each other within the predicted band at each $P$, with the stagnation floor decreasing from approximately $3 \times 10^{-1}$ at $P = 64$ to $10^{-3}$ at $P = 1{,}024$.
 
\begin{figure}
    \centering
    \includegraphics[width=0.9\linewidth]{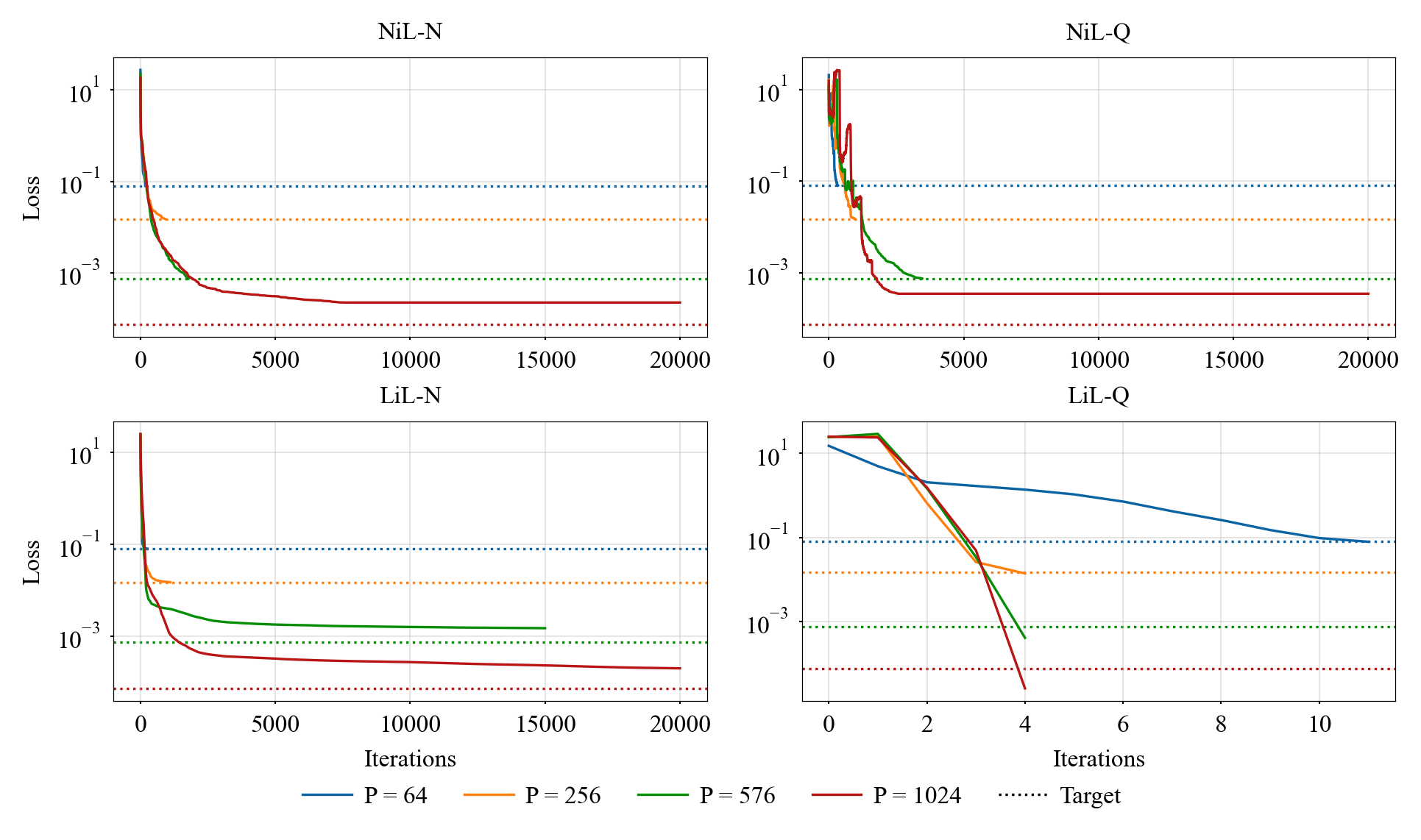}
    \caption{Buckley-Leverett (viscous flow, $N_g = 0$): convergence histories organized by method.
    Each panel shows all four basis sizes for a single formulation.}
    \label{fig:bl_convergence_by_method}
\end{figure}

\begin{figure}
    \centering
    \includegraphics[width=0.95\linewidth]{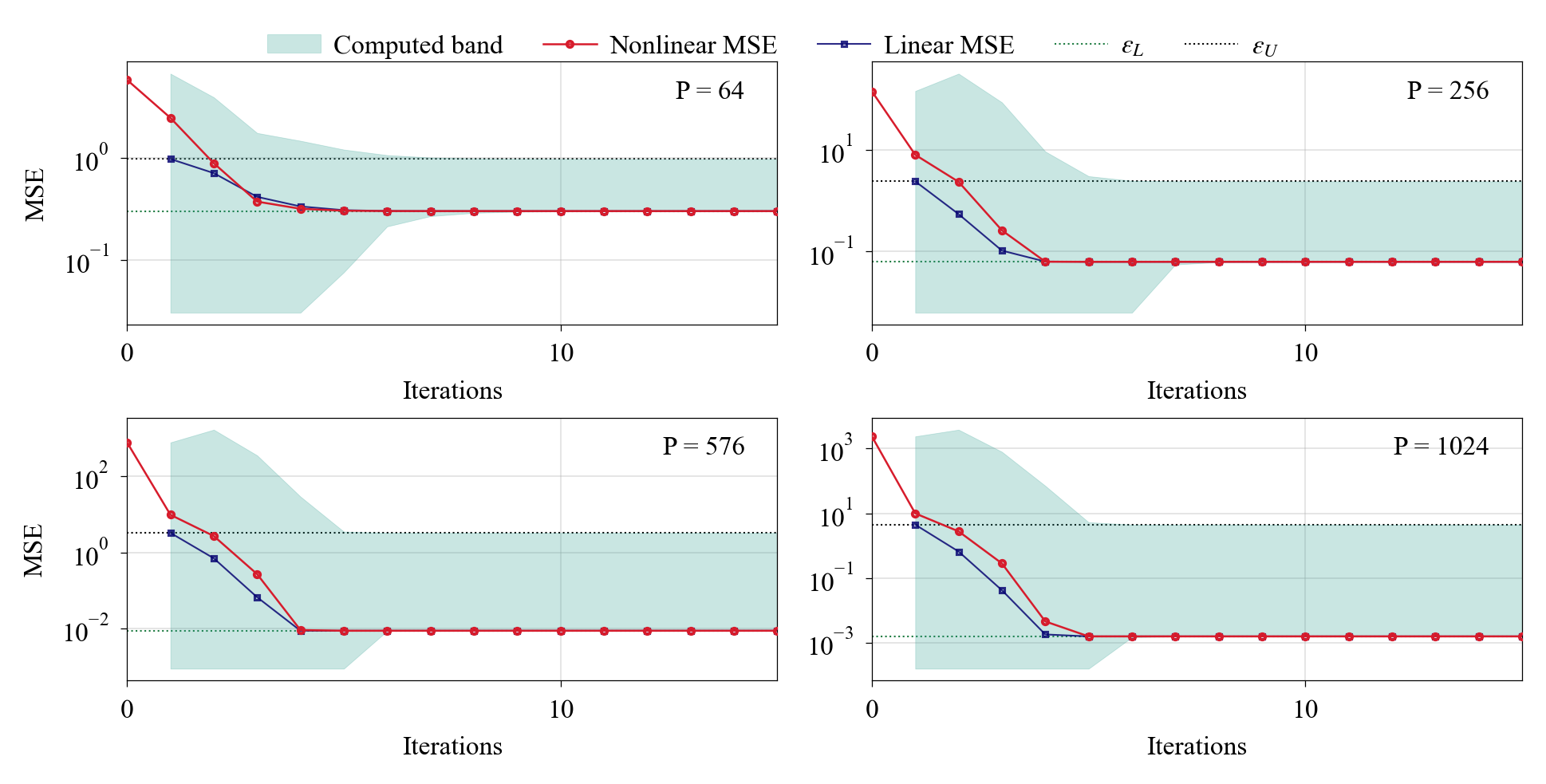}
    \caption{Buckley-Leverett (viscous flow, $N_g = 0$): comparison of linearized and nonlinear residual norms for \textsc{LiL-Q} at $P \in \{64, 256, 576, 1{,}024\}$.}
    \label{fig:bl_residual_bounds}
\end{figure}
 
\begin{table}
\centering
\caption{Buckley-Leverett (gravity-influenced flow, $N_g = -5$): iteration counts and runtimes. An asterisk (*) indicates no convergence.}
\label{tab:bl_case2_iterations}
\small
\begin{tabular}{@{}ccrrrrrrrr@{}}
\toprule
& & \multicolumn{4}{c}{Iterations} & \multicolumn{4}{c}{Runtime (s)} \\
\cmidrule(lr){3-6} \cmidrule(lr){7-10}
$P$ & Target $\operatorname{MSE}$ & \textsc{NiL-N} & \textsc{NiL-Q} & \textsc{LiL-N} & \textsc{LiL-Q} & \textsc{NiL-N} & \textsc{NiL-Q} & \textsc{LiL-N} & \textsc{LiL-Q} \\
\midrule
64    & $2.5 \times 10^{-1}$ & 229    & 213    & 95    & 9  & 7.7  & 5.75  & 2.72  & 0.05 \\
256   & $1.5 \times 10^{-1}$ & 680   & 1804   & 8086   & 7 & 28.08  & 63.71  & 198.3  & 0.49 \\
576   & $7.5 \times 10^{-2}$ & 1002   & 1692   & 15000*  & 10  & 58.94  & 83.9  & 437.9  & 3.16 \\
1024  & $3.5 \times 10^{-2}$   & 11384  & 7178   & 20000*  & 8  & 867   & 457.7   & 568.5   & 9.18 \\
\bottomrule
\end{tabular}
\end{table}
 
The gravity-dominated case introduces counter-current flow where shocks can travel in opposite directions. This introduces substantially more complex solution structure, and it is reflected in the more permissive $\operatorname{MSE}$ targets in Table~\ref{tab:bl_case2_iterations} (e.g., $3.5 \times 10^{-2}$ at $P = 1{,}024$ compared to $7.5 \times 10^{-5}$ in the viscous case).
Once again, neither partial elimination strategy provides reliable gains. \textsc{NiL-Q} yields inconsistent results, marginally faster than \textsc{NiL-N} at $P = 64$ but $2.7\times$ slower at $P = 256$. \textsc{LiL-N} helps at $P = 64$ but inverts sharply, exhausting the budget at $P \geq 576$ as the more complex flux couples more coefficients through the nonlinearity.
Only \textsc{LiL-Q} converges consistently, requiring 7 to 10 outer iterations across all four basis sizes.
The convergence histories in Fig.~\ref{fig:bl_gravity_convergence_by_method} reveal that the \textsc{NiL-Q} loss curves exhibit transient upward spikes at the start of each quasilinear iteration: updating the linearization point changes the subproblem, and the optimizer state accumulated for the previous subproblem is initially counterproductive on the new one.
This effect is negligible on the milder Bratu and Burgers nonlinearities but becomes pronounced on the strongly nonlinear BL flux, contributing to \textsc{NiL-Q} requiring more iterations than \textsc{NiL-N} at most basis sizes.
\textsc{LiL-Q} is immune to this effect because no inner optimization state is carried across outer iterations.
Fig.~\ref{fig:bl_gravity_residual_bounds} confirms that the convergence theory continues to hold under the non-monotonic flux. The nonlinear residual tracks the linearized residual and settles within the predicted band at each $P$, though the stagnation floors are higher than in the viscous case, reflecting the greater difficulty of approximation with smooth basis functions. As will be detailed in Section~\ref{sec:conditioning_results}, at the largest parameter counts the collocation matrix becomes numerically rank-deficient ($\kappa\sim10^{16}$, rank deficit of order $10^{2}$), reflecting the difficulty of resolving the sharp saturation front with a smooth global basis. The column-pivoted QR solve handles this regime stably in this case, however.

The solution fields in Fig.~\ref{fig:bl_comparison} compare the \textsc{NiL-N} and \textsc{LiL-Q} predictions at $P = 1{,}024$ for both the viscous and gravity-influenced cases. Both methods capture the qualitative structure of the counter-current saturation waves despite the substantial differences in iteration count and runtime.
 
\begin{figure}
    \centering
    \includegraphics[width=0.9\linewidth]{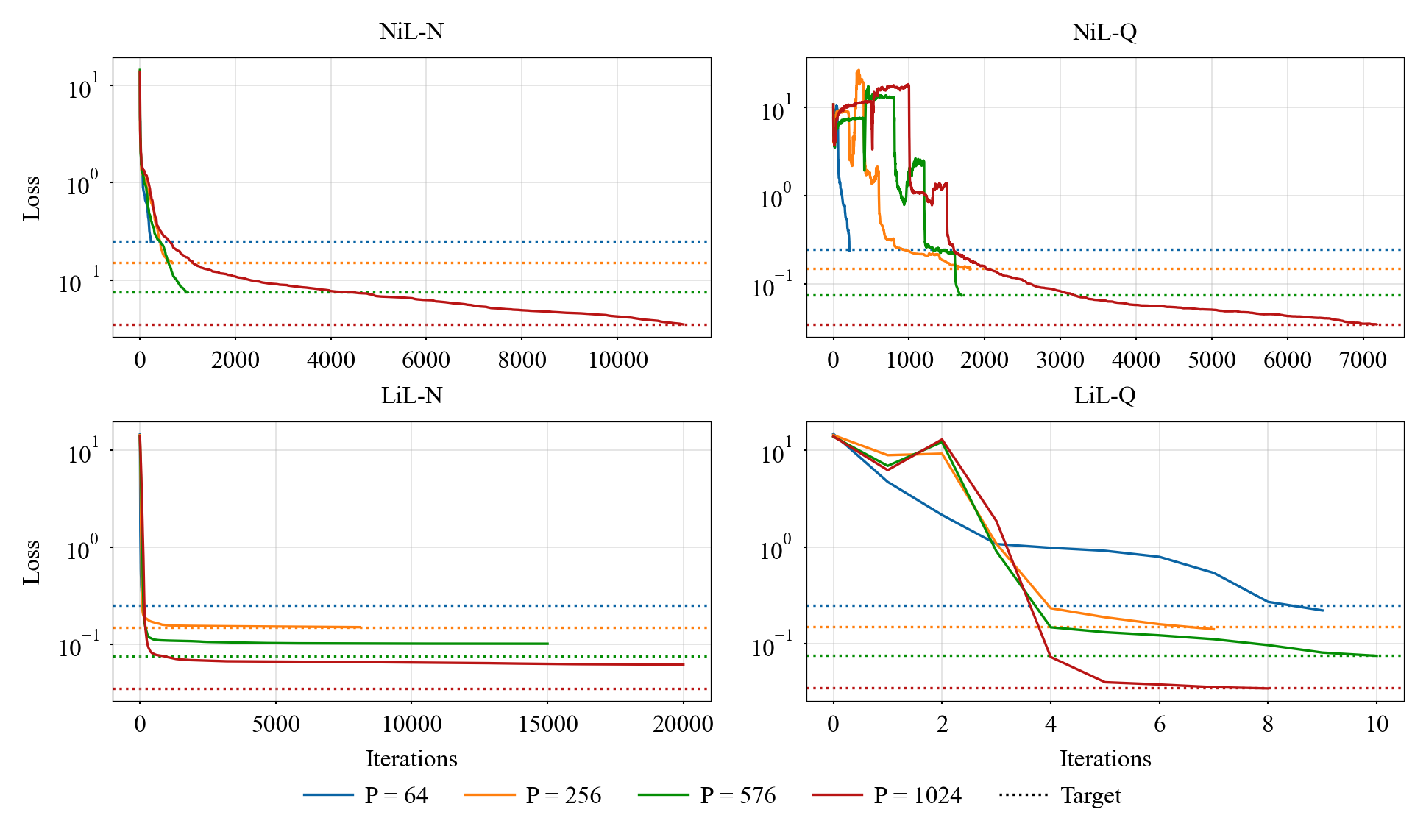}
    \caption{Buckley-Leverett (gravity-influenced flow, $N_g = -5$): convergence histories organized by method.
    The \textsc{NiL-Q} panel exhibits order-of-magnitude loss spikes not observed in any other experiment.}
    \label{fig:bl_gravity_convergence_by_method}
\end{figure}

\begin{figure}
    \centering
    \includegraphics[width=0.95\linewidth]{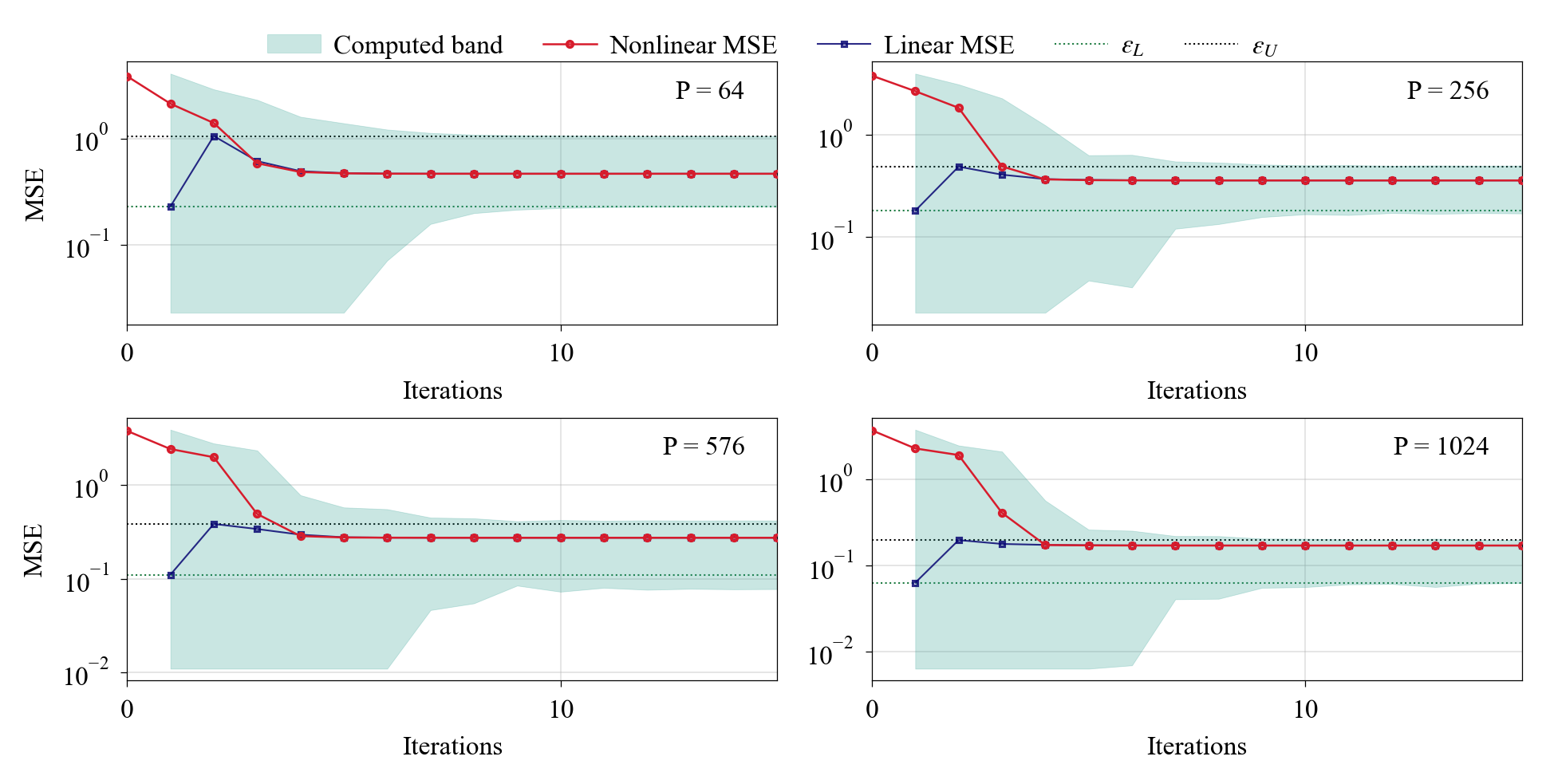}
    \caption{Buckley-Leverett (gravity-influenced flow, $N_g = -5$): comparison of linearized and nonlinear residual norms for \textsc{LiL-Q} at $P \in \{64, 256, 576, 1{,}024\}$.}
    \label{fig:bl_gravity_residual_bounds}
\end{figure}
 
\begin{figure}
    \centering
    \setlength{\tabcolsep}{2pt}
    \begin{tabular}{@{}cc@{}}
        \includegraphics[width=0.35\linewidth]{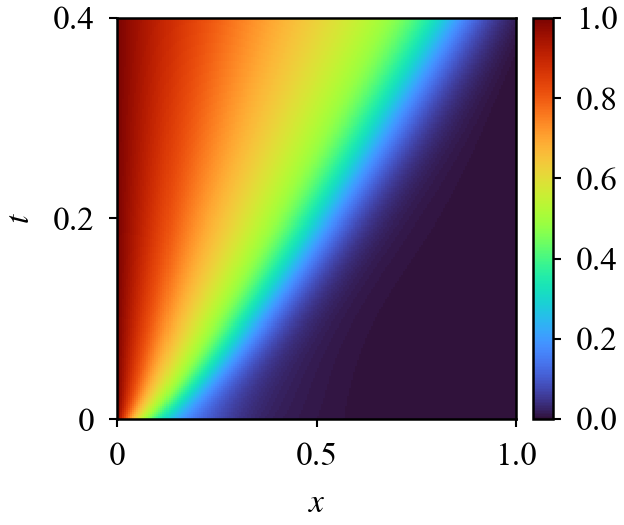} &
        \includegraphics[width=0.35\linewidth]{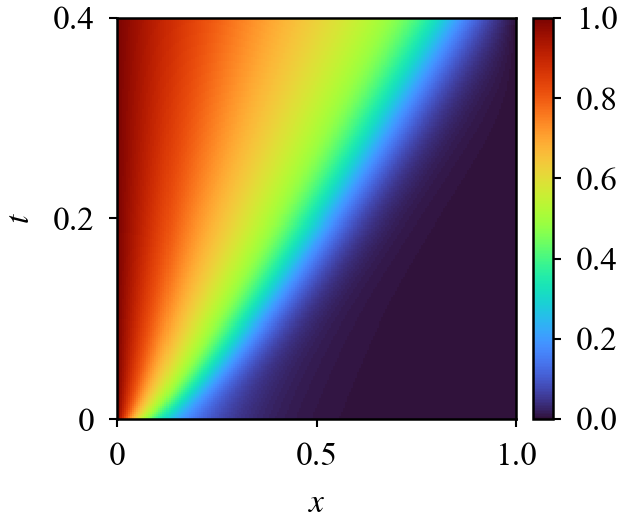} \\[-0.1cm]
        {\small (a) \textsc{NiL-N}, $N_g = 0$} &
        {\small (b) \textsc{LiL-Q}, $N_g = 0$} \\[0.2cm]
        \includegraphics[width=0.35\linewidth]{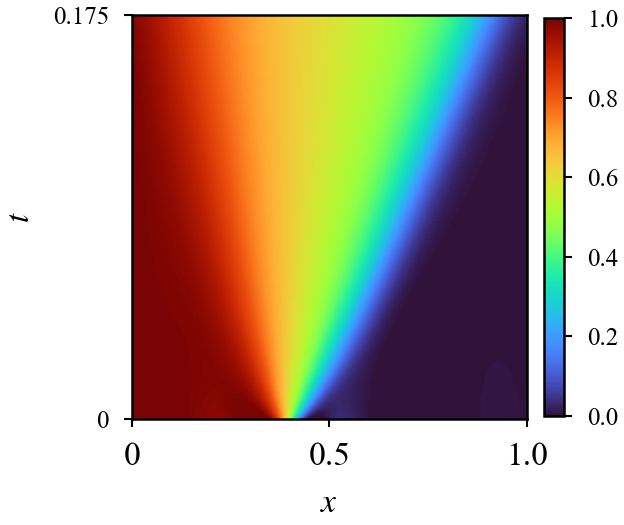} &
        \includegraphics[width=0.35\linewidth]{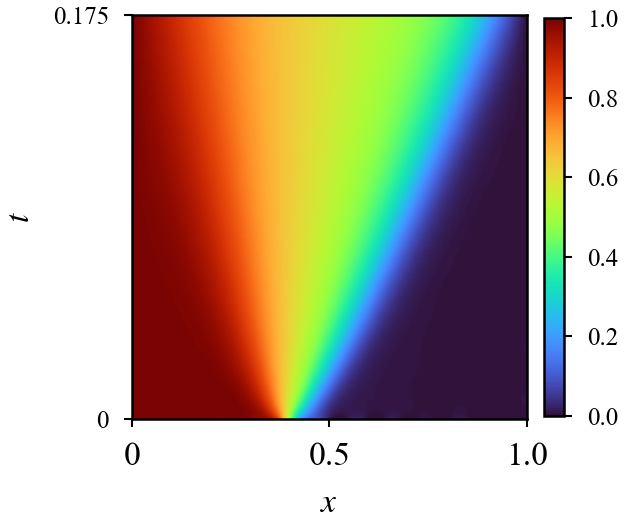} \\[-0.1cm]
        {\small (c) \textsc{NiL-N}, $N_g = -5$} &
        {\small (d) \textsc{LiL-Q}, $N_g = -5$}
    \end{tabular}
    \caption{Buckley-Leverett solutions at $P = 1{,}024$: \textsc{NiL-N} (left) and \textsc{LiL-Q} (right) for viscous displacement (top) and gravity-influenced displacement (bottom).}
    \label{fig:bl_comparison}
\end{figure}

\subsection{Coupled Linear Elasticity}
\label{sec:elasticity}

The remaining experiments extend the scope from scalar nonlinear PDEs to coupled systems and rough heterogeneity.
We consider the plane-strain Navier equations of linear elasticity on $\Omega = (0,1)^2$, seeking the displacement fields $u_x(x,y)$ and $u_y(x,y)$ satisfying,
\begin{equation}
\label{eq:elasticity_system}
(\lambda + 2\mu)\,\frac{\partial^2 u_x}{\partial x^2} + \mu\,\frac{\partial^2 u_x}{\partial y^2} + (\lambda + \mu)\,\frac{\partial^2 u_y}{\partial x\,\partial y} = -f_x,
\end{equation}
and
\begin{equation}
\label{eq:elasticity_system_2}
\mu\,\frac{\partial^2 u_y}{\partial x^2} + (\lambda + 2\mu)\,\frac{\partial^2 u_y}{\partial y^2} + (\lambda + \mu)\,\frac{\partial^2 u_x}{\partial x\,\partial y} = -f_y,
\end{equation}
where $\lambda$ and $\mu$ are the Lam\'{e} parameters and $f_x$, $f_y$ are prescribed body forces.
The coupling between the two displacement fields enters through the mixed derivative $(\lambda + \mu)\,\partial^2/\partial x\,\partial y$, which prevents the system from being solved as two independent scalar problems and gives rise to the block-structured collocation system described below.
To verify the solver against a known reference, we follow the method of manufactured solutions used in~\cite{Haghighat2021}, prescribing the displacement fields,
\begin{equation}
\label{eq:elasticity_exact}
u_x^*(x,y) = \cos(2\pi x)\,\sin(\pi y), \qquad
u_y^*(x,y) = \sin(\pi x)\,\frac{Q\,y^4}{4},
\end{equation}
with $\lambda = 1$, $\mu = 0.5$, and $Q = 4$, and computing the body forces by evaluating the differential operators on the left-hand side of~\eqref{eq:elasticity_system} at these displacements.
Following~\cite{Haghighat2021}, we apply mixed boundary conditions. The bottom face is clamped ($u_x = u_y = 0$), the top face imposes $u_x = 0$ together with a prescribed normal traction $\sigma_{yy} = (\lambda + 2\mu)\,Q\sin(\pi x)$, and on the lateral faces we impose zero normal traction ($\sigma_{xx} = 0$) together with $u_y = 0$.

Since~\eqref{eq:elasticity_system} is linear in the displacements, no quasilinearization is needed in LiL methods and the system is solved directly in a single solve.
Each displacement field is approximated by an independent expansion with its own basis, i.e.,
\begin{equation}
\label{eq:elasticity_lil}
u_x(\mathbf{x}) \approx \sum_{p=1}^{P_u} \beta_p^{(u_x)}\,\phi_p^{(u_x)}(\mathbf{x})
\end{equation}
and,
\begin{equation}
\label{eq:elasticity_lil_2}
u_y(\mathbf{x}) \approx \sum_{q=1}^{P_v} \beta_q^{(u_y)}\,\phi_q^{(u_y)}(\mathbf{x}),
\end{equation}
and substituting these into~\eqref{eq:elasticity_system} at the collocation points produces a linear system with a block-column structure:
\begin{equation}
\label{eq:elasticity_block}
\begin{bmatrix}
\mathbf{A}_{x\text{-mom}}^{(u_x)} & \mathbf{A}_{x\text{-mom}}^{(u_y)} \\[4pt]
\mathbf{A}_{y\text{-mom}}^{(u_x)} & \mathbf{A}_{y\text{-mom}}^{(u_y)} \\[4pt]
\mathbf{A}_{\mathrm{BC}}^{(u_x)} & \mathbf{A}_{\mathrm{BC}}^{(u_y)}
\end{bmatrix}
\begin{bmatrix}
\boldsymbol{\beta}^{(u_x)} \\[4pt] \boldsymbol{\beta}^{(u_y)}
\end{bmatrix}
=
\begin{bmatrix}
\mathbf{f}_{x\text{-mom}} \\[4pt] \mathbf{f}_{y\text{-mom}} \\[4pt] \mathbf{f}_{\mathrm{BC}}
\end{bmatrix}.
\end{equation}
The off-diagonal blocks $\mathbf{A}_{x\text{-mom}}^{(u_y)}$ and $\mathbf{A}_{y\text{-mom}}^{(u_x)}$ arise from the coupling term and are what distinguishes this from two independent scalar solves.
The bases are chosen to align with the structure of the manufactured displacement fields. For $u_x^*$, we use tensor products of cosine functions in $x$ and sine functions in $y$, and for $u_y^*$, we use sine functions in $x$ and Chebyshev polynomials in $y$ (the latter representing the quartic $y^4$ exactly at degree $\geq 4$).
With these choices, the exact solution is contained in the span of the basis at every resolution tested, placing the experiment in the $\varepsilon_L = 0$ regime of Theorem~\ref{thm:stationary_residual}.
Experiments were conducted at five basis sizes with $p_d \in \{5, 10, 15, 20, 25\}$ modes per dimension per field, yielding total parameter counts $P \in \{50, 200, 450, 800, 1{,}250\}$ (since $P_u = P_v = p_d^2$ and $P = P_u + P_v$); the collocation grid uses a $10{:}1$ ratio of samples to parameters, with $85\%$ allocated to interior points and $15\%$ to boundary conditions.

Table~\ref{tab:elasticity_results} reports the relative $L_2$ errors and solve times.
At every basis size, the displacement errors are at $\mathcal{O}(10^{-16})$ and the stress errors (computed from the displacement derivatives) are at $\mathcal{O}(10^{-15})$ to $\mathcal{O}(10^{-16})$, confirming recovery to machine precision.
The PDE residual MSE ranges from $6.6 \times 10^{-29}$ at $P = 50$ to $6.1 \times 10^{-27}$ at $P = 1{,}250$. The slight increase with $P$ reflects accumulated round-off in the larger matrix operations rather than any degradation in solution quality.
This behavior is consistent with the conditioning analysis of Section~\ref{sec:properties} where we show that the condition number grows only modestly, from $\kappa\approx1.1\times10^{2}$ at $P=50$ to $\kappa\approx4.4\times10^{4}$ at $P=1{,}250$, and the collocation matrix remains full column rank at every $P$ (Section~\ref{sec:conditioning_results}).
The minimizer is therefore unique, and because the exact solution lies in the span of the basis ($\varepsilon_{L}=0$) while $\kappa\,\epsilon_{\mathrm{mach}}\approx10^{-11}$ stays far below any practical floor, the error remains at machine precision throughout.
Two distinct facts underlie this result and are worth separating.
First, the matrix is full rank and all $P$ basis functions contribute linearly independent collocation responses, so no direction is discarded.
Second, the \emph{solution} is sparse in this basis such that across all tested sizes, the number of coefficients exceeding $10^{-6}$ in magnitude remains constant at five modes per dimension per field, with every additional coefficient driven to zero by the least-squares fit.
Together these confirm the $\varepsilon_{L}=0$ regime of the theory. The basis spans the solution exactly with a handful of modes, and enriching it further adds well-conditioned but inactive degrees of freedom, leaving the accuracy unchanged.
Fig.~\ref{fig:elasticity_fields} displays the exact and predicted displacement fields alongside the pointwise absolute errors at $P = 50$. Even with only 25 basis functions per field and a QR solve completing in under 2~milliseconds, the predicted fields are visually indistinguishable from the analytical solution, with discrepancies uniformly at the level of floating-point arithmetic throughout the domain.

\begin{table}
\centering
\caption{Linear elasticity: relative $L_2$ errors and QR solve times across five basis sizes.
All errors are at machine precision, confirming the $\varepsilon_L = 0$ regime.}
\label{tab:elasticity_results}
\small
\begin{tabular}{@{} r r r c c c c c @{}}
\toprule
$p_d$ & $P$ & Time (s) & $\epsilon_{u_x}$ & $\epsilon_{u_y}$ & $\epsilon_{\sigma_{xx}}$ & $\epsilon_{\sigma_{yy}}$ & $\epsilon_{\sigma_{xy}}$ \\
\midrule
 5 &   50 & 0.007 & $6.5 \times 10^{-16}$ & $1.0 \times 10^{-15}$ & $2.4 \times 10^{-16}$ & $3.7 \times 10^{-16}$ & $8.6 \times 10^{-16}$ \\
10 &  200 & 0.029 & $7.1 \times 10^{-16}$ & $9.3 \times 10^{-16}$ & $4.3 \times 10^{-16}$ & $6.7 \times 10^{-16}$ & $9.0 \times 10^{-16}$ \\
15 &  450 & 0.106 & $1.2 \times 10^{-15}$ & $9.1 \times 10^{-16}$ & $1.1 \times 10^{-15}$ & $1.1 \times 10^{-15}$ & $1.4 \times 10^{-15}$ \\
20 &  800 & 0.283 & $1.6 \times 10^{-15}$ & $1.6 \times 10^{-15}$ & $1.4 \times 10^{-15}$ & $1.0 \times 10^{-15}$ & $1.2 \times 10^{-15}$ \\
25 & 1{,}250 & 1.03  & $3.0 \times 10^{-15}$ & $2.2 \times 10^{-15}$ & $2.6 \times 10^{-15}$ & $2.2 \times 10^{-15}$ & $2.3 \times 10^{-15}$ \\
\bottomrule
\end{tabular}
\end{table}

\begin{figure}
\centering
 \includegraphics[width=0.95\linewidth]{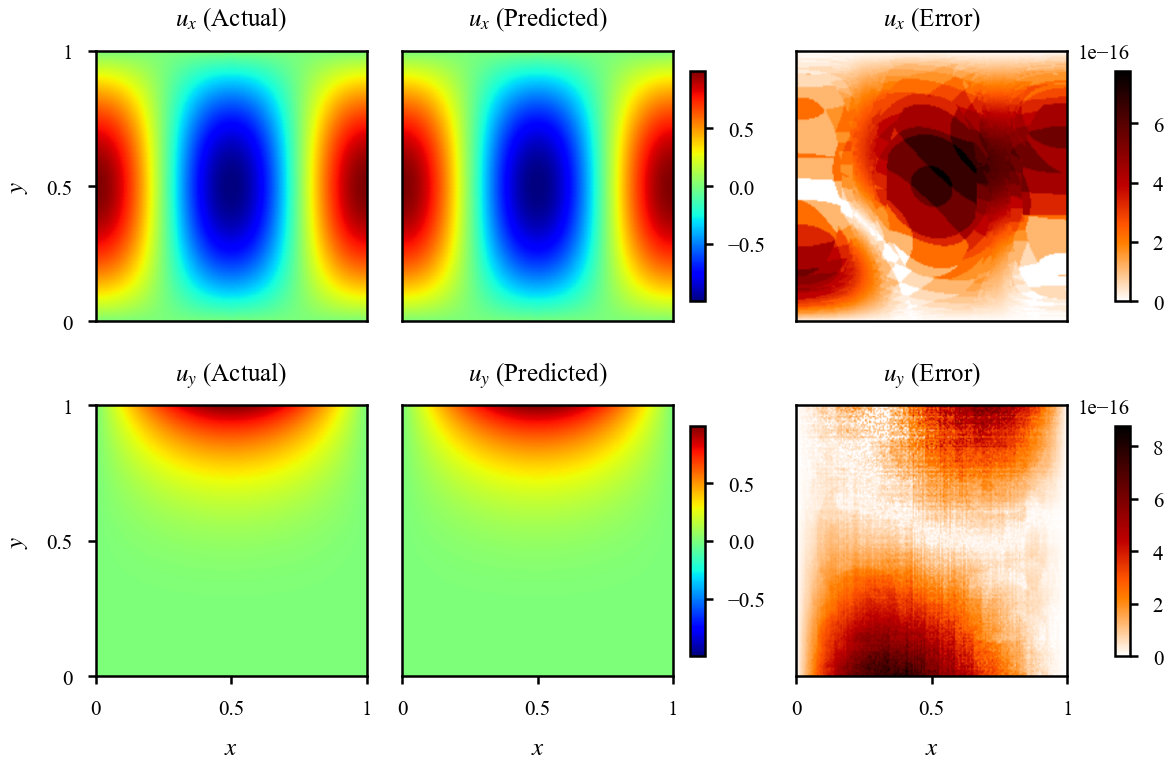}
\caption{Linear elasticity at $P = 50$: exact displacement fields (left), \textsc{LiL} predictions (center), and pointwise absolute errors (right).
Errors are at $\mathcal{O}(10^{-15})$ throughout the domain, consistent with machine-precision recovery.}
\label{fig:elasticity_fields}
\end{figure}

Table~\ref{tab:elasticity_comparison} places these results alongside the SciANN boundary-collocation PINN of~\cite{Haghighat2021}, which uses the same manufactured displacements~\eqref{eq:elasticity_exact}, material parameters, and mixed boundary conditions.
We emphasize that Table~\ref{tab:elasticity_comparison} is intended to provide only a qualitative comparative view.
The SciANN configuration trains five $4\!\times\!40$ \texttt{tanh} networks for the two displacements and three stress components, together with two learnable Lam\'{e} parameters, for a total of $25{,}407$ trainable weights; training runs for $2{,}872$ Adam epochs and $361$~s of wall-clock time.
This is a mixed forward/inverse formulation (stresses and material parameters are learned jointly with displacements), so the training objectives are not directly comparable to the displacement-only least-squares solve in \textsc{LiL}.
Post-training evaluation on the same $200\times200$ grid yields relative $L_2$ displacement errors $\epsilon_{u_x} \approx 2.8\times 10^{-2}$ and $\epsilon_{u_y} \approx 4.5\times 10^{-2}$.
Even at $P = 50$, \textsc{LiL} recovers both displacement fields to $\mathcal{O}(10^{-15})$ in a single QR solve, with a total wall time (including collocation assembly) of $0.07$~s.
At $P = 1{,}250$, which still uses roughly 20 times fewer degrees of freedom than SciANN, the total wall time is $4.11$~s, roughly $90$ times shorter, while the accuracy advantage exceeds 13 orders of magnitude.
 
\begin{table}
\centering
\caption{Linear elasticity: comparison with the published SciANN PINN solver~\cite{Haghighat2021}.
Same manufactured solution, material parameters, and boundary conditions.
\textsc{LiL} runtimes include collocation assembly and QR solve.
A dash indicates unreported data.}
\label{tab:elasticity_comparison}
\small
\begin{tabular}{@{} l r l r c c @{}}
\toprule
Method & $P$ & Training & Time (s) & $\epsilon_{u_x}$ & $\epsilon_{u_y}$ \\
\midrule
SciANN (PINN)~\cite{Haghighat2021}
    & 25{,}407 & 2{,}872 Adam epochs & 361
    & $2.8 \times 10^{-2}$ & $4.5 \times 10^{-2}$ \\
\midrule
    & 50    & 1 QR solve & 0.07
    & $6.5 \times 10^{-16}$ & $1.0 \times 10^{-15}$ \\[3pt]
\textsc{LiL}
    & 800   & 1 QR solve & 1.82
    & $1.6 \times 10^{-15}$ & $1.6 \times 10^{-15}$ \\[3pt]
    & 1{,}250 & 1 QR solve & 4.11
    & $3.0 \times 10^{-15}$ & $2.2 \times 10^{-15}$ \\
\bottomrule
\end{tabular}
\end{table}

\subsection{Kovasznay Flow}
\label{sec:kovasznay}

Next, we consider the Kovasznay flow 2D steady solution to the incompressible Navier-Stokes equations governed by
\begin{equation}
\label{eq:kovasznay_ns}
\begin{aligned}
u\,\frac{\partial u}{\partial x} + v\,\frac{\partial u}{\partial y} &= -\frac{\partial p}{\partial x} + \frac{1}{Re}\left(\frac{\partial^2 u}{\partial x^2} + \frac{\partial^2 u}{\partial y^2}\right), \\[4pt]
u\,\frac{\partial v}{\partial x} + v\,\frac{\partial v}{\partial y} &= -\frac{\partial p}{\partial y} + \frac{1}{Re}\left(\frac{\partial^2 v}{\partial x^2} + \frac{\partial^2 v}{\partial y^2}\right), \\[4pt]
\frac{\partial u}{\partial x} + \frac{\partial v}{\partial y} &= 0,
\end{aligned}
\end{equation}
where $u(x,y)$ and $v(x,y)$ are the velocity components, and $p(x,y)$ is the pressure.
Its analytical solution is
\begin{equation}
\label{eq:kovasznay_exact}
\begin{aligned}
u &= 1 - e^{\zeta x}\cos(2\pi y), \\[4pt]
v &= \frac{\zeta}{2\pi}\,e^{\zeta x}\sin(2\pi y), \\[4pt]
p &= \tfrac{1}{2}(1 - e^{2\zeta x}),
\end{aligned}
\end{equation}
where $\zeta = \mathrm{Re}/2 - \sqrt{\mathrm{Re}^2/4 + 4\pi^2}$.
We set $\mathrm{Re} = 40$ on the domain $\Omega = [-0.5, 1] \times [-0.5, 1.5]$ and impose Dirichlet boundary conditions from the analytical solution on all four faces.
This problem has been widely used to benchmark PINN-based Navier-Stokes solvers~\cite{Jin2021, Lee2024, Gu2024}.

Unlike the linear elasticity system of Section~\ref{sec:elasticity}, the quadratic convective terms in~\eqref{eq:kovasznay_ns} introduce a nonlinearity that requires the Bellman-Kalaba quasilinearization.
At each outer iteration $k$, the products of unknowns are replaced by first-order Taylor expansions about the current iterate, yielding a system that is linear in $(u^{(k+1)}, v^{(k+1)}, p^{(k+1)})$ with coefficients that depend on $(u^{(k)}, v^{(k)})$.
Each field is expanded in a tensor-product Chebyshev basis with $p_d$ modes per spatial dimension, giving $P = 3\,p_d^2$ total trainable parameters, and the resulting collocation system has the same block-column structure as~\eqref{eq:elasticity_block} but with three column blocks (for $u$, $v$, $p$) and row blocks from the $x$-momentum, $y$-momentum, continuity, and boundary equations.
Each linearized subproblem is solved via QR factorization, and the process is repeated until the coefficient update norm falls below a prescribed tolerance.
Experiments were conducted at $p_d \in \{5, 10, 15, 20, 25\}$, corresponding to $P \in \{75, 300, 675, 1{,}200, 1{,}875\}$, with collocation points sampled at $80\%$ interior, $10\%$ boundary, and $10\%$ continuity at a $2{:}1$ sample-to-parameter ratio per field. The pressure is pinned at one corner to remove the gauge ambiguity.
 
\begin{table}
\centering
\caption{Kovasznay flow ($\mathrm{Re} = 40$): outer iterations, wall-clock time, and relative $L_2$ errors across five basis sizes.}
\label{tab:kovasznay_results}
\small
\begin{tabular}{@{} r r r r c c c @{}}
\toprule
$p_d$ & $P$ & Iters & Time (s) & $\epsilon_{u}$ & $\epsilon_{v}$ & $\epsilon_{p}$ \\
\midrule
5  & 75     & 16 & 0.06  & $3.8 \times 10^{-1}$  & $1.2 \times 10^{0\phantom{-}}$   & $8.8 \times 10^{-1}$ \\
10 & 300    & 10 & 0.22  & $2.9 \times 10^{-2}$  & $1.3 \times 10^{-1}$  & $4.8 \times 10^{-1}$ \\
15 & 675    &  6 & 0.67  & $2.0 \times 10^{-5}$  & $1.3 \times 10^{-4}$  & $8.1 \times 10^{-4}$ \\
20 & 1{,}200 &  6 & 1.87  & $7.2 \times 10^{-9}$  & $2.6 \times 10^{-8}$  & $1.4 \times 10^{-7}$ \\
25 & 1{,}875 &  6 & 5.58  & $7.2 \times 10^{-13}$ & $5.3 \times 10^{-12}$ & $1.3 \times 10^{-11}$ \\
\bottomrule
\end{tabular}
\end{table}

Table~\ref{tab:kovasznay_results} summarizes the results.
At $P = 75$, the Chebyshev basis lacks the resolution to represent the exponential boundary layer in the Kovasznay solution, and the errors remain $\mathcal{O}(1)$ despite convergence of the quasilinearization.
As $P$ increases, the solution errors decrease rapidly, reaching $\mathcal{O}(10^{-13})$ at $P = 1{,}875$ in 5.6~seconds.
The outer iteration count stabilizes at 6 for $P \geq 675$, consistent with the convergence rate behavior observed in the scalar experiments.
The exact and predicted fields at $P = 1{,}875$, displayed in Fig.~\ref{fig:kovasznay_fields}, are visually indistinguishable, with pointwise errors at $\mathcal{O}(10^{-12})$ for the velocity components and $\mathcal{O}(10^{-11})$ for the pressure.

\begin{figure}
\centering
\includegraphics[width=0.95\linewidth]{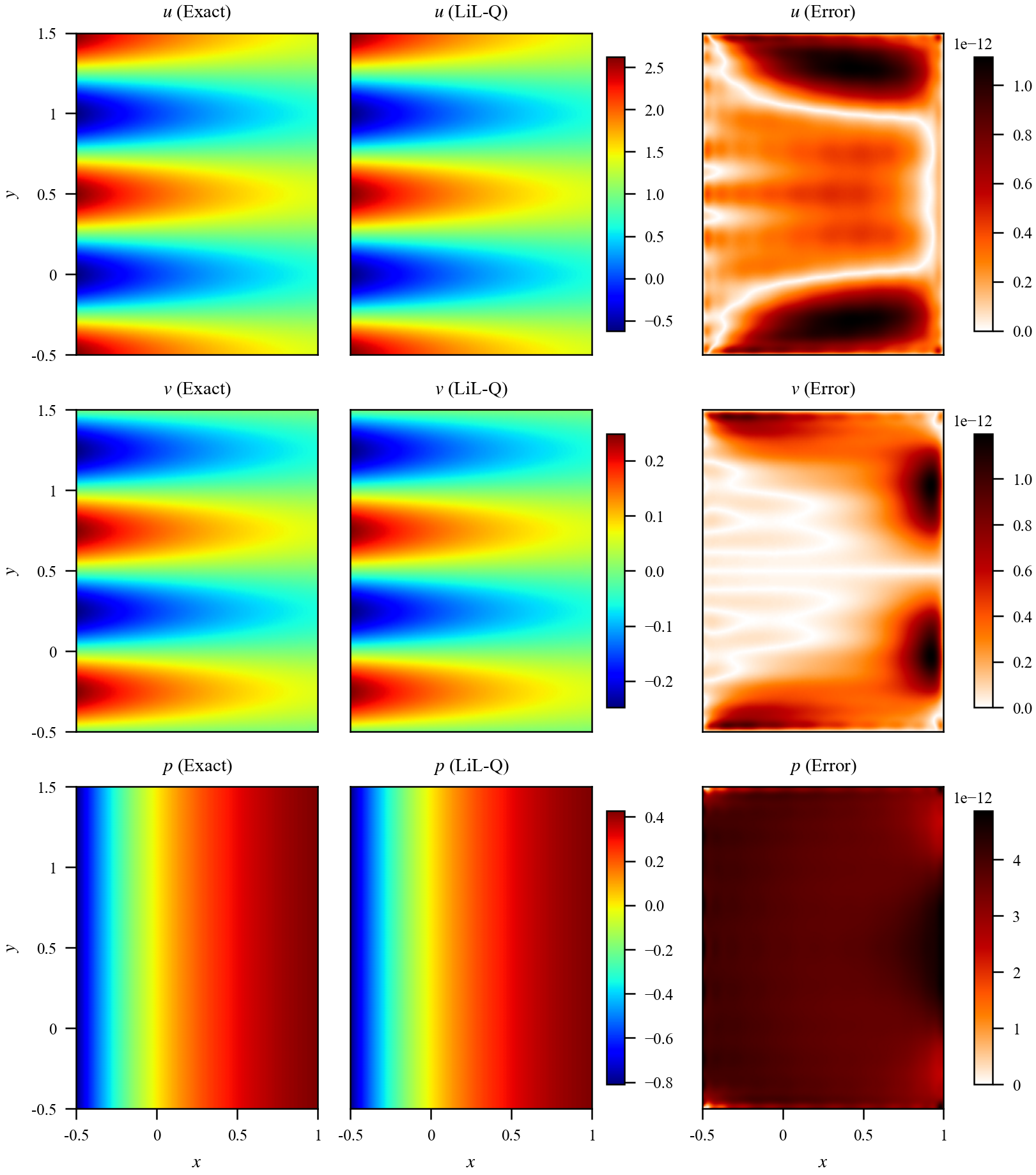}
\caption{Kovasznay flow at $P = 1{,}875$: exact fields (left), \textsc{LiL-Q} predictions (center), and pointwise absolute errors (right) for $u$, $v$, and $p$.}
\label{fig:kovasznay_fields}
\end{figure}

Table~\ref{tab:kovasznay_comparison} places these results alongside published PINN solutions for the same problem at $\mathrm{Re} = 40$. We emphasize that Table~\ref{tab:kovasznay_comparison} is intended to only provide a qualitative comparative view.
The VV-NSFnet of Jin et al.~\cite{Jin2021} achieves the closest velocity accuracy among the baselines ($\epsilon_u = 2.4 \times 10^{-5}$), but requires ${\sim}90{,}000$ trainable parameters and over $3 \times 10^4$ gradient iterations. \textsc{LiL-Q} at $P = 675$ matches this accuracy ($\epsilon_u = 2.0 \times 10^{-5}$) with 133 times fewer parameters, 6 outer iterations, and 0.67~seconds of wall-clock time.
Among Lee's~\cite{Lee2024} configurations, the ADA-F ($4\!\times\!60$) yields $\epsilon_u = 7.2 \times 10^{-5}$ using ${\sim}11{,}000$ parameters. \textsc{LiL-Q} at $P = 675$ produces a $3.6\times$ smaller velocity error with roughly $16\times$ fewer parameters and in 6 outer iterations. The reported wall-clock times differ by a factor of order $10^{3}$ (4{,}089~s versus 0.67~s), though this figure compares results obtained on different hardware and we report it only as an indication of the magnitude of the difference.
At $P = 1{,}875$, the errors drop to $\mathcal{O}(10^{-12})$ across all field variables, exceeding every reported baseline by three to four orders of magnitude in 5.58~seconds.
 
\begin{table}
\centering
\caption{Kovasznay flow ($\mathrm{Re} = 40$): Summary of published PINN solver results.
Each row presents a reported configuration from the corresponding reference.
A dash indicates unreported data.}
\label{tab:kovasznay_comparison}
\small
\begin{tabular}{@{} l r l r c c c @{}}
\toprule
Method & $P$ & Training & Time (s) & $\epsilon_{u}$ & $\epsilon_{v}$ & $\epsilon_{p}$ \\
\midrule
VV-NSFnet~\cite{Jin2021}
    & ${\sim}$90{,}000  & 30K Adam + L-BFGS  & --
    & $2.4 \times 10^{-5}$ &  $1.4 \times 10^{-4}$ & -- \\[3pt]
DD-PINN~\cite{Gu2024}
    & ${\sim}$32{,}000  & 30K Adam + L-BFGS      & --
    & $8.2 \times 10^{-5}$ & $7.1 \times 10^{-4}$ & $3.0 \times 10^{-4}$ \\[3pt]
PINN~\cite{Lee2024}
    & ${\sim}$11{,}300  & Adam + L-BFGS      & $1174.8$
    & $9.7 \times 10^{-5}$ & $1.1 \times 10^{-3}$ & $2.8 \times 10^{-4}$ \\[3pt]
ADA-F ($2 \times 20$)~\cite{Lee2024}
    & $547$  & Adam + L-BFGS      & $251.7$
    & $1.4 \times 10^{-4}$ & $9.7 \times 10^{-4}$ & $2.9 \times 10^{-4}$ \\[3pt]
ADA-F ($4 \times 60$)~\cite{Lee2024}
    & ${\sim}$11{,}000  & Adam + L-BFGS      & $4089.4$
    & $7.2 \times 10^{-5}$ & $5.4 \times 10^{-4}$ & $1.4 \times 10^{-4}$ \\
\midrule
    & 675               & 6 QR solves        & 0.67
    & $2.0 \times 10^{-5}$ & $1.3 \times 10^{-4}$ & $8.1 \times 10^{-4}$ \\[3pt]
\textsc{LiL-Q}
    & 1{,}200           & 6 QR solves        & 1.87
    & $7.2 \times 10^{-9}$ & $2.6 \times 10^{-8}$ & $1.4 \times 10^{-7}$ \\[3pt]
    & 1{,}875           & 6 QR solves        & 5.58
    & $7.2 \times 10^{-13}$ & $5.3 \times 10^{-12}$ & $1.3 \times 10^{-11}$ \\
\bottomrule
\end{tabular}
\end{table}

\subsection{3D Beltrami Flow}
\label{sec:beltrami}
 
We consider the unsteady three-dimensional Beltrami flow with an exact solution to the incompressible Navier-Stokes equations on $\Omega = [-1,1]^3$ over $t \in [0,1]$~\cite{Jin2021, Gu2024}.
The velocity fields are
\begin{equation}
\label{eq:beltrami_vel}
\begin{aligned}
u &= -a\bigl[\exp(ax)\sin(ay+dz) + \exp(az)\cos(ax+dy)\bigr]\exp(-d^2 t), \\
v &= -a\bigl[\exp(ay)\sin(az+dx) + \exp(ax)\cos(ay+dz)\bigr]\exp(-d^2 t), \\
w &= -a\bigl[\exp(az)\sin(ax+dy) + \exp(ay)\cos(az+dx)\bigr]\exp(-d^2 t),
\end{aligned}
\end{equation}
with pressure
\begin{equation}
\label{eq:beltrami_p}
\begin{split}
p = -\tfrac{1}{2}a^2\bigl[&\exp(2ax)+\exp(2ay)+\exp(2az) \\
  &+ 2\sin(ax+dy)\cos(az+dx)\exp\!\bigl(a(y+z)\bigr) \\
  &+ 2\sin(ay+dz)\cos(ax+dy)\exp\!\bigl(a(z+x)\bigr) \\
  &+ 2\sin(az+dx)\cos(ay+dz)\exp\!\bigl(a(x+y)\bigr)\bigr]\exp(-2d^2 t),
\end{split}
\end{equation}
where $a = d = 1$.
Dirichlet conditions from the analytical solution are imposed on all six faces and at $t = 0$.
This problem is substantially more demanding than the Kovasznay flow as it involves four spatiotemporal dimensions, four coupled field variables, and a pressure field whose exponential terms produce a dynamic range considerably wider than that of the velocity fields.
 
The \textsc{LiL-Q} formulation follows the same quasilinearization approach as in Section~\ref{sec:kovasznay}, with the block-column structure now comprising four column blocks (for $u$, $v$, $w$, $p$) and row blocks from three momentum equations, the continuity equation, boundary conditions on six faces, and the initial condition.
Each field is expanded in a 4D tensor-product Chebyshev basis. To accommodate the steeper pressure gradients, the pressure basis is enriched relative to the velocity bases.
We assign $N_{\mathrm{vel}} = 6$ modes per dimension for each velocity component ($P_{\mathrm{vel}} = 6^4 = 1{,}296$ per field) and $N_p = 8$ for pressure ($P_p = 8^4 = 4{,}096$), giving $P = 3 \times 1{,}296 + 4{,}096 = 7{,}984$ total parameters.
The collocation system comprises 4{,}096 interior points (each contributing four equation rows), 1{,}296 boundary points on the six faces (three velocity conditions each), 512 initial-condition points (three fields), and one pressure-pin row, yielding 21{,}809 rows for 7{,}984 unknowns.

The quasilinearization converges in 4 outer iterations with a total wall-clock time of approximately 5~minutes on CPU, dominated by four QR factorizations of the $21{,}809 \times 7{,}984$ system.
 
\begin{table}
\centering
\caption{Beltrami flow: relative $L_2$ errors (\%) at five time snapshots and the entire domain.
\textsc{LiL-Q} with $P = 7{,}984$ parameters, converged in 4 outer iterations (wall time $\approx 5$~min on CPU).}
\label{tab:beltrami_results}
\small
\begin{tabular}{@{} l c c c c c c @{}}
\toprule
 & $t = 0.00$ & $t = 0.25$ & $t = 0.50$ & $t = 0.75$ & $t = 1.00$ & Entire domain \\
\midrule
$\epsilon_u$(\%) & 0.014 & 0.041 & 0.037 & 0.038 & 0.034 & 0.0345 \\
$\epsilon_v$(\%) & 0.014 & 0.041 & 0.037 & 0.038 & 0.034 & 0.0345 \\
$\epsilon_w$(\%) & 0.014 & 0.041 & 0.037 & 0.038 & 0.034 & 0.0345 \\
$\epsilon_p$(\%) & 0.146 & 0.282 & 0.315 & 0.451 & 0.752 & 0.2610 \\
\bottomrule
\end{tabular}
\end{table}
 
Table~\ref{tab:beltrami_results} reports the relative $L_2$ errors at five time snapshots.
The velocity errors range from 0.014\% to 0.041\% and are identical across the three components to four significant figures. This is a consequence of the cyclic symmetry of~\eqref{eq:beltrami_vel} combined with the identical basis and collocation grids used for all three velocity fields.
The pressure errors are higher (0.15 to 0.75\%), consistent with the wider dynamic range of~\eqref{eq:beltrami_p} and the fact that pressure is not directly constrained by boundary or initial data but inferred through the incompressibility constraint.
Fig.~\ref{fig:beltrami_3d} shows the velocity and pressure fields at $t = 1.0$ on the $z = 0$ slice, confirming visual agreement between the exact and predicted solutions.
 
\begin{figure}
\centering
\includegraphics[width=0.85\linewidth]{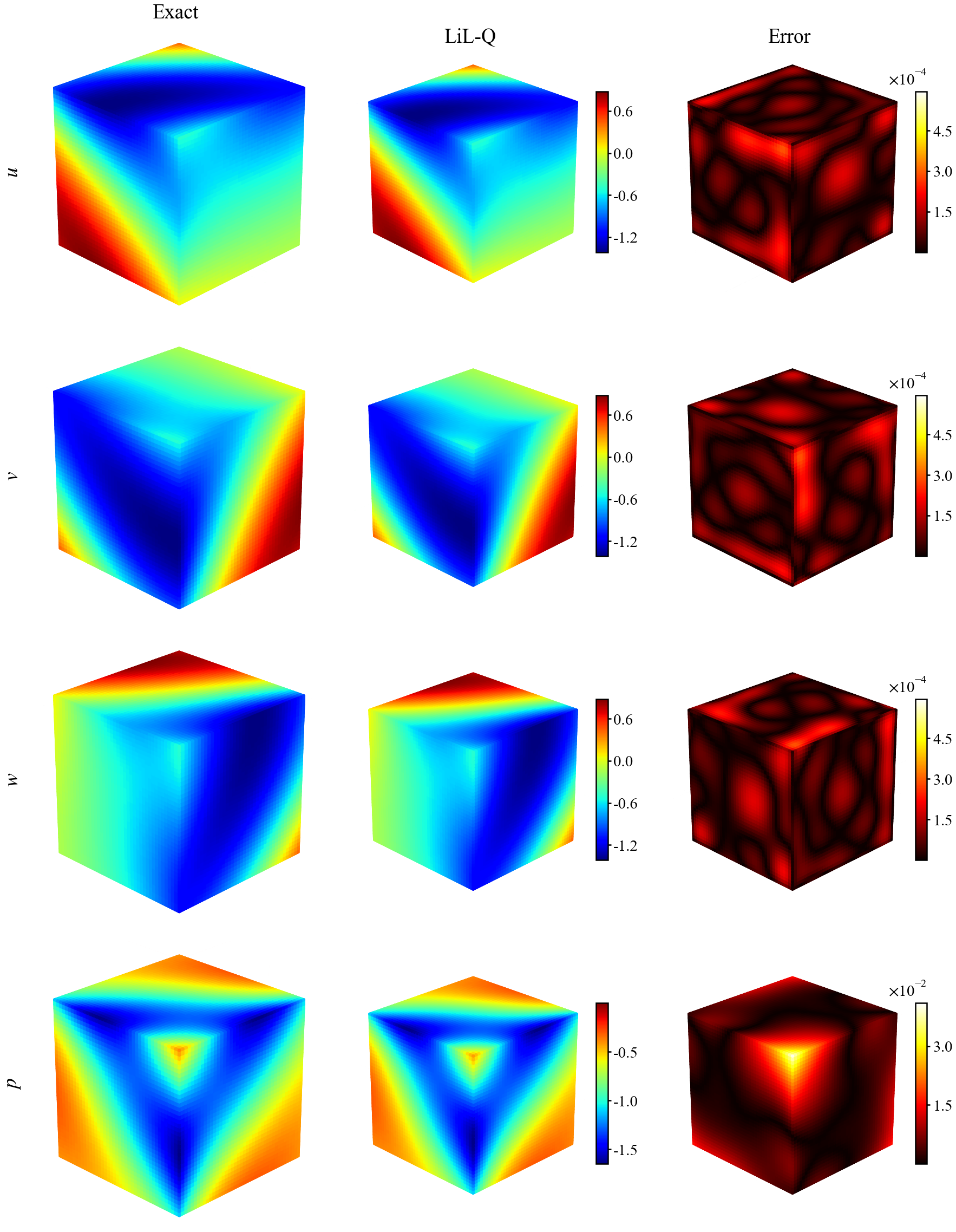}
\caption{Beltrami flow at $t = 1$: exact fields (left), \textsc{LiL-Q} predictions (middle), and pointwise absolute errors (right) for $u$, $v$, $w$, and $p$.}
\label{fig:beltrami_3d}
\end{figure}

Table~\ref{tab:beltrami_comparison} presents these results alongside published PINN implementations for the same problem, evaluated at $t = 1.0$.
Jin et al.~\cite{Jin2021} report two formulations at a comparable parameter count (${\sim}8{,}000$): the VP-NSFnet, which recovers velocity to 0.16 - 0.22\% but yields pressure errors of 8.93\%, and the VV-NSFnet, which achieves velocity errors of 0.041 to 0.045\%.
At the same parameter count (7{,}984), \textsc{LiL-Q} produces velocity errors of 0.034\% (slightly below the VV-NSFnet) while recovering pressure to 0.75\%, an order of magnitude below the VP-NSFnet. The 4 QR factorizations complete in approximately 5~minutes, replacing $3 \times 10^4$ gradient iterations.
Gu et al.~\cite{Gu2024} employ a domain-decomposed architecture with ${\sim}126{,}000$ total parameters, reporting velocity errors of 0.16--0.21\% and a pressure error of 2.11\%. \textsc{LiL-Q} achieves 5 to 6 times lower velocity error and 2.8 times lower pressure error with 16 times fewer parameters.
 
\begin{table}
\centering
\caption{Beltrami flow at $t = 1.0$: comparison with published PINN solvers.
Each row reports the best configuration from the corresponding reference.
Parameter counts for domain-decomposed methods are totals across all subdomains.
A dash indicates unreported data.}
\label{tab:beltrami_comparison}
\small
\begin{tabular}{@{} l r l r c c c c @{}}
\toprule
Method & $P$ & Training
    & $\epsilon_{u}$(\%) & $\epsilon_{v}$(\%) & $\epsilon_{w}$(\%) & $\epsilon_{p}$(\%) \\
\midrule
VP-NSFnet~\cite{Jin2021}
    & ${\sim}$8{,}000  & 30K Adam + L-BFGS  & 0.163 & 0.219 & 0.178 & 8.93 \\[3pt]
VV-NSFnet~\cite{Jin2021}
    & ${\sim}$8{,}000  & 30K Adam + L-BFGS  & 0.044 & 0.045 & 0.041 & -- \\[3pt]
DD-PINN~\cite{Gu2024}
    & ${\sim}$126{,}000 & 30K Adam + L-BFGS      & 0.162 & 0.190 & 0.213 & 2.11 \\[3pt]
\textsc{LiL-Q} (this work)
    & 7{,}984           & 4 QR solves        & 0.034 & 0.034 & 0.034 & 0.75 \\
\bottomrule
\end{tabular}
\end{table}

We note that this is the most severely conditioned problem in the study. At $P=7984$ the collocation matrix has $\kappa\approx4\times10^{17}$ and a numerical rank deficit of seven, so the accuracy reported here relies on the rank-revealing column-pivoted QR solve (Section~\ref{sec:conditioning_results}).


\subsection{Single-Phase Darcy Flow}
\label{sec:single_phase_darcy}

Steady-state single-phase flow through porous media with spatially heterogeneous permeability is governed by the elliptic pressure equation
\begin{equation}
\label{eq:darcy_pressure}
\nabla \cdot \left( \frac{K(\mathbf{x})}{\mu} \nabla P(\mathbf{x}) \right) = 0, \qquad \mathbf{x} \in \Omega,
\end{equation}
where $K(\mathbf{x})$ is the permeability field, $\mu$ is viscosity, and $P(\mathbf{x})$ is the pressure.
This equation is the computational kernel of reservoir simulation, appearing as the pressure solve within sequential implicit methods for multiphase flow~\cite{AzizSettari1979}.
Since~\eqref{eq:darcy_pressure} is linear in $P$ for any fixed $K$, there is no physical nonlinearity and quasilinearization is not required. The \textsc{NiL-Q} and \textsc{LiL-Q} formulations reduce to \textsc{NiL} and \textsc{LiL}, respectively, so any performance gap between the two is attributable entirely to architectural nonlinearity.
The experiment includes a finite volume method (FVM) solution as an independent numerical reference~\cite{AzizSettari1979}.

We adopt a mixed velocity-pressure formulation, solving simultaneously for a dimensionless pressure head $h^* = (P - P_{\mathrm{bot}})/\Delta P$ and velocity components $u^*$, $v^*$ on a rectangular reservoir of dimensions $L_x \times L_y = 1{,}200 \times 2{,}200$~ft, discretized on a $60 \times 220$ cell grid, with pressure boundary conditions on the top and bottom faces and no-flow conditions on the lateral boundaries.
Mapping coordinates to the unit square via $x^* = x/L_x$, $y^* = y/L_y$ yields the dimensionless governing equations
\begin{equation}
\label{eq:darcy_nondim}
\begin{aligned}
u^* + K^* \mathcal{R}\, \frac{\partial h^*}{\partial x^*} &= 0, & &\text{(Darcy-}x\text{)} \\[4pt]
v^* + K^* \frac{\partial h^*}{\partial y^*} &= 0, & &\text{(Darcy-}y\text{)} \\[4pt]
\mathcal{R}\, \frac{\partial u^*}{\partial x^*} + \frac{\partial v^*}{\partial y^*} &= 0, & &\text{(Continuity)}
\end{aligned}
\end{equation}
where $\mathcal{R} = L_y/L_x$ is the domain aspect ratio and $K^* = K/K_0$ is the normalized permeability. Both the aspect-ratio scaling and permeability normalization were introduced to improve conditioning after observing anisotropic convergence during preliminary experiments, and were applied uniformly to the \textsc{NiL} and \textsc{LiL} implementations.

For the \textsc{LiL} formulation, a lifting function decomposition $h^* = y^* + \tilde{h}^*$ satisfies the Dirichlet pressure conditions exactly, with the correction $\tilde{h}^*$ expanded in Fourier activations that vanish at $y^* = 0$ and $y^* = 1$.
Each field variable uses a Fourier network whose activation functions are chosen to satisfy its respective boundary conditions by construction: the pressure correction uses $\cos(n\pi x^*)\sin(m\pi y^*)$ (cosine for Neumann laterally, sine for Dirichlet vertically). The $x$-velocity uses $\sin(n\pi x^*)$ spatially (enforcing no-flow at the lateral boundaries) with augmented cosine modes in $y^*$. The $y$-velocity uses augmented cosine modes in both directions, as no boundary conditions are imposed on it directly.
With 32 Fourier modes per direction per field, the total \textsc{LiL} parameter count across all three variables is $P = 3{,}169$.
Collocation points are placed at cell centers of the permeability grid, and the resulting overdetermined system is solved via QR factorization in a single pass.
The experiments are conducted across four permeability realizations of increasing heterogeneity. Three synthetic fields (S1, S2, S3) are generated via Karhunen-Lo\`eve expansion of a Gaussian covariance kernel with increasing $\ln K$ standard deviation, and the third case is a 2D slice of the SPE10 benchmark dataset~\cite{spe10}.
Fig.~\ref{fig:perm_fields} shows the four permeability fields. Table~\ref{tab:darcy_training} summarizes the training configurations for both methods.

\begin{table}
\centering
\caption{Training configurations for the single-phase Darcy flow experiments.
Three separate networks are used for $h^*$, $u^*$, and $v^*$ in the \textsc{NiL} formulation; the reported parameter count is the total across all three.}
\label{tab:darcy_training}
\small
\begin{tabular}{@{}llll@{}}
\toprule
\multicolumn{2}{@{}l}{\textit{\textsc{NiL} configuration}} & \multicolumn{2}{l}{\textit{\textsc{LiL} configuration}} \\
\cmidrule(lr){1-2} \cmidrule(lr){3-4}
Hidden layers       & 2                    & Modes per direction & 32 \\
Neurons per layer   & 32                   & Total parameters    & 3{,}169 \\
Activation          & $\tanh$ / SiLU       & Solver              & QR (single pass) \\
Output ($h^*$)      & sigmoid              & --- & --- \\
Total parameters    & 3{,}555              & --- & --- \\
Optimizer           & Adam, StepLR         & --- & --- \\
Iterations          & 150{,}000            & --- & --- \\
\midrule
\multicolumn{4}{@{}l}{\textit{Shared:} $K$ normalization: $K_0 = 1$ (synthetic), $K_0 = K_{\mathrm{geo}}$ (SPE10); grid: $60 \times 220$ cells.} \\
\bottomrule
\end{tabular}
\end{table}

\begin{figure}
    \centering
    \includegraphics[width=1\linewidth]{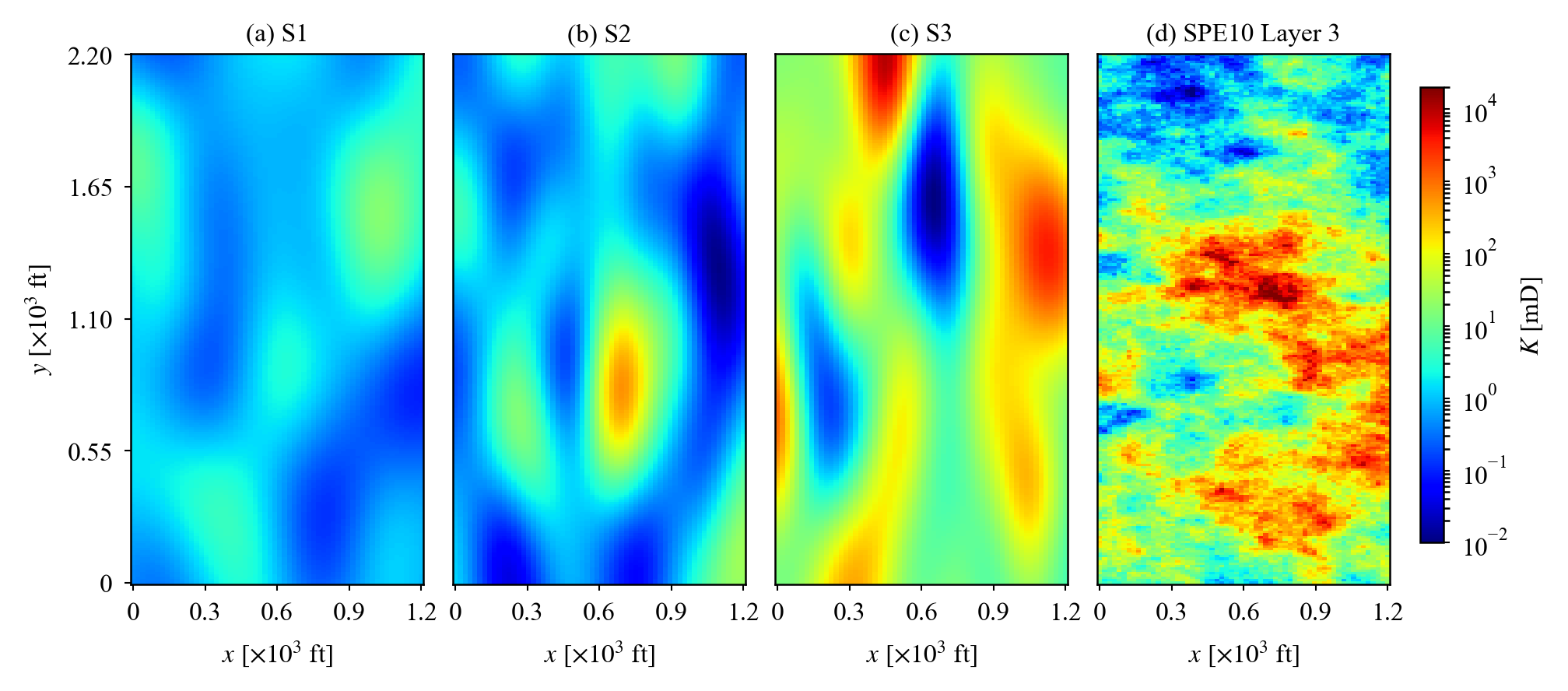}
    \caption{Permeability fields ($\log_{10} K$ in mD) for the four test cases, with heterogeneity contrast increasing from left to right.}
    \label{fig:perm_fields}
\end{figure}

Table~\ref{tab:darcy_results} reports PDE residual MSE and solve times, and Fig.~\ref{fig:pressure_fields2} compares the resulting pressure fields across all four realizations.
The \textsc{LiL} solutions are obtained from a single QR solve completing in approximately 24~s, while the \textsc{NiL} solutions require 150{,}000 gradient iterations over 960 to 1{,}579~s, despite comparable parameter counts ($P = 3{,}555$ for \textsc{NiL} versus $P = 3{,}169$ for \textsc{LiL}).
For the moderate-heterogeneity fields (S1, S2), \textsc{LiL} achieves Darcy residuals several orders of magnitude lower than \textsc{NiL} while completing in a fraction of the time.
For S3, the Darcy residuals are comparable between the two methods, and all three solutions capture the qualitative pressure distribution. This case marks the onset of the regime where increased heterogeneity in the coefficient field begins to limit the accuracy achievable with smooth global bases.
On the SPE10 field, \textsc{NiL} produces Darcy residuals an order of magnitude larger than \textsc{LiL}. The \textsc{NiL} runtime is also substantially longer (1{,}579~s versus 24~s), though, as elsewhere, this cross-hardware wall-clock comparison is indicative rather than definitive and is conservative for \textsc{LiL-Q}, which runs on CPU.
Across all four fields, the \textsc{LiL} boundary conditions are satisfied to machine precision by construction, whereas \textsc{NiL} boundary residuals are small but non-zero. The FVM solves each problem in approximately 0.1~s.

\begin{table}
\centering
\caption{PDE residual MSE and solve times for single-phase Darcy flow.
\textsc{LiL}: 32 Fourier modes per direction ($P = 3{,}169$).
\textsc{NiL}: 2 hidden layers, 32 neurons ($P = 3{,}555$), 150{,}000 iterations.}
\label{tab:darcy_results}
\small
\begin{tabular}{@{}llrrrr@{}}
\toprule
Case & Method & Darcy-$x$ MSE & Darcy-$y$ MSE & Continuity MSE & Runtime (s) \\
\midrule
\multirow{3}{*}{S1}
  & \textsc{NiL} & $8.41 \times 10^{-1}$ & $9.48$ & $1.38 \times 10^{-10}$ & 960.5 \\
  & \textsc{LiL} & $2.20 \times 10^{-5}$ & $9.47 \times 10^{-5}$ & $1.67 \times 10^{-12}$ & 24 \\
  & FVM          & ---                    & ---                    & $2.79 \times 10^{-7}$ & 0.1 \\
\midrule
\multirow{3}{*}{S2}
  & \textsc{NiL} & $3.43 \times 10^{-2}$ & $4.21 \times 10^{-2}$ & $8.05 \times 10^{-10}$ & 960.7 \\
  & \textsc{LiL} & $6.25 \times 10^{-5}$ & $1.45 \times 10^{-4}$ & $5.47 \times 10^{-12}$ & 23 \\
  & FVM          & ---                    & ---                    & $6.57 \times 10^{-7}$ & 0.1 \\
\midrule
\multirow{3}{*}{S3}
  & \textsc{NiL} & $3.32 \times 10^{-2}$ & $2.04 \times 10^{-1}$ & $4.31 \times 10^{-9}$ & 966.1 \\
  & \textsc{LiL} & $1.97 \times 10^{-1}$ & $1.36 \times 10^{-1}$ & $1.30 \times 10^{-8}$ & 23 \\
  & FVM          & ---                    & ---                    & $1.34 \times 10^{-3}$ & 0.1 \\
\midrule
\multirow{3}{*}{SPE10}
  & \textsc{NiL} & $24.98$ & $94.69$ & $2.97 \times 10^{-7}$ & 1{,}579 \\
  & \textsc{LiL} & $7.11 \times 10^{-1}$ & $5.23$ & $8.94 \times 10^{-8}$ & 24 \\
  & FVM          & ---     & ---     & $2.12 \times 10^{-1}$ & 0.1 \\
\bottomrule
\end{tabular}
\end{table}

\begin{figure}
    \centering
    \includegraphics[width=0.91\linewidth]{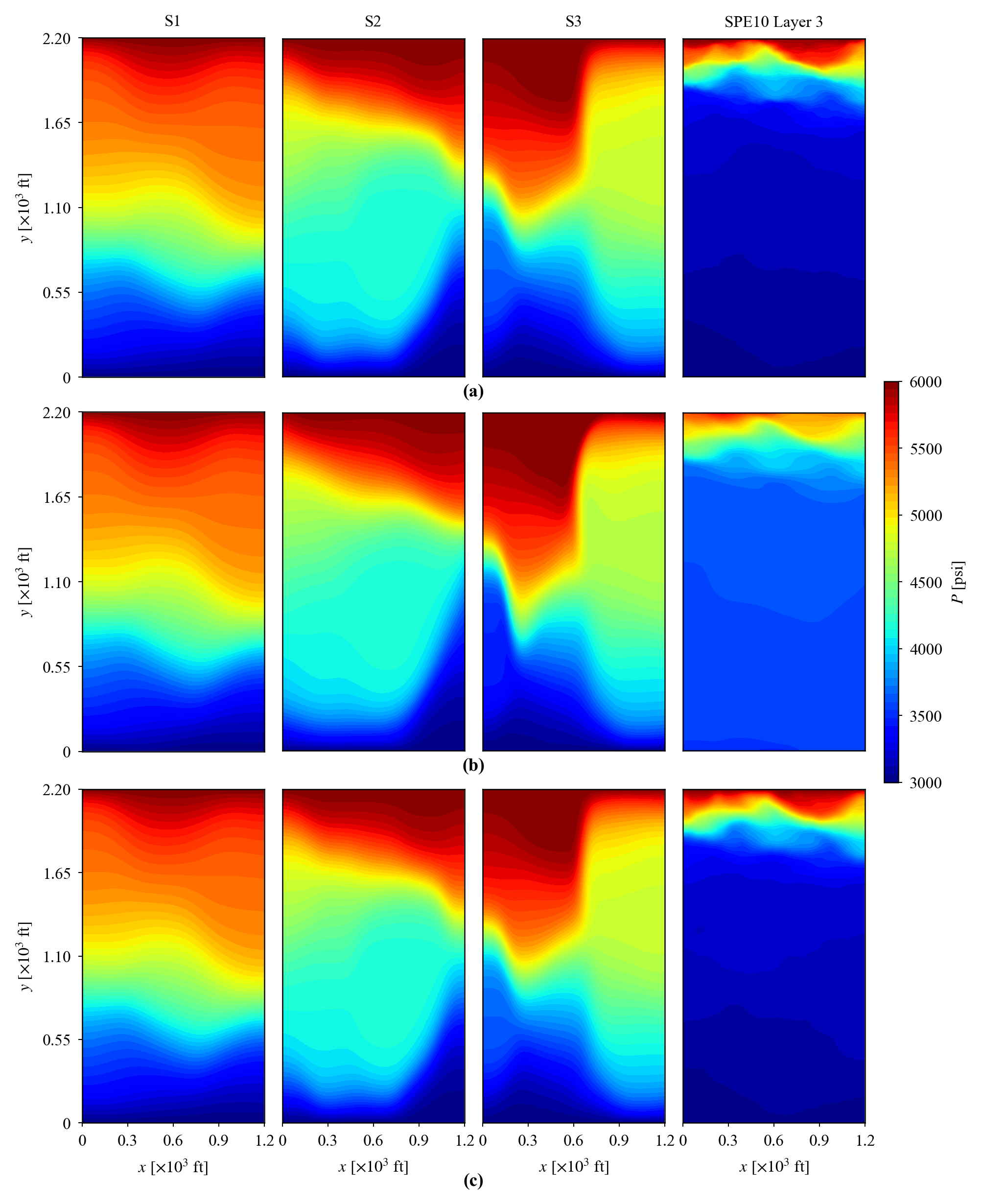}
    \caption{Pressure field solutions for (a) FVM (b) \textsc{NiL} and (c) \textsc{LiL} across the four permeability realizations.}
    \label{fig:pressure_fields2}
\end{figure}

\subsection{Conditioning of the Collocation System}
\label{sec:conditioning_results}

Having presented the individual benchmarks, we now examine the conditioning of the linear-in-learnables collocation matrix across all cases, which determines the round-off contribution to accuracy analyzed in Section~\ref{sec:properties}.
Table~\ref{tab:condition_numbers} reports, for each problem, the condition number $\kappa(\mathbf{A}^{(k)})$ at the smallest and largest parameter counts, the numerical rank at the largest parameter count, and the variation of $\kappa$ across outer iterations.
Two observations stand out.
First, $\kappa(\mathbf{A}^{(k)})$ spans an enormous range across problems from $\mathcal{O}(10^{2})$ for the smooth elasticity and Kovasznay flows to $\mathcal{O}(10^{17})$--$\mathcal{O}(10^{19})$ for the advection-dominated Buckley-Leverett and three-dimensional Beltrami problems. Yet, in every full-rank case, the reported accuracy is consistent with the round-off bound $\kappa\,\epsilon_{\mathrm{mach}}$ being dominated by the basis floor $\varepsilon_{L}$, so conditioning does not limit the results.
Second, the collocation matrix is full column rank at every parameter count tested except for the two most severely conditioned problems. The Buckley-Leverett and Beltrami problems develop rank deficits (of order $10$--$10^{2}$ out of $P$) at their largest parameter counts, and the rank-revealing column-pivoted QR factorization is essential in these cases, returning the minimum-norm solution rather than amplifying the deficient directions.
The condition numbers also stabilize after the first outer iterate, varying by less than a factor of two across the subsequent Newton sequence for the well-resolved problems, confirming that the conditioning is a property of the converged linearization.

\begin{table}
\centering
\caption{Condition number $\kappa(\mathbf{A}^{(k)})$ of the LiL-Q collocation matrix across the experiments.
Columns report the parameter range tested, $\kappa$ at the smallest and largest parameter counts (final-iterate values), the across-iteration stability factor $\max_k\kappa/\min_k\kappa$ after the first outer iterate.}
\label{tab:condition_numbers}
\small
\begin{tabular}{lccccc}
\hline
Problem & $P$ range & $\kappa$ (min $P$) & $\kappa$ (max $P$) & Iter.\ stability & Num.\ rank \\
\hline
Bratu                     & $25$--$225$   & $1.7\times10^{2}$ & $1.4\times10^{9}$  & $1.00\times$ & full ($=P$) \\
Burgers                   & $25$--$625$   & $8.9\times10^{1}$ & $6.8\times10^{9}$  & $4.30\times$ & full ($=P$) \\
Buckley--Leverett (no g.) & $64$--$1024$  & $6.9\times10^{3}$ & $3.5\times10^{16}$ & $1.20\times$ & $934$ (def.\ $90$) \\
Buckley--Leverett (g.)    & $64$--$1024$  & $7.8\times10^{3}$ & $4.5\times10^{16}$ & $1.31\times$ & $936$ (def.\ $88$) \\
Linear elasticity         & $50$--$1250$  & $1.1\times10^{2}$ & $4.4\times10^{4}$  & ---$^{\dagger}$ & full ($=P$) \\
Kovasznay                 & $75$--$1875$  & $7.4\times10^{1}$ & $1.7\times10^{5}$  & $1.05\times$ & full ($=P$) \\
Darcy / SPE10             & $3169$        & $2.6\times10^{6}$ & $1.3\times10^{7}$ & ---$^{\ddagger}$ & full ($=P$) \\
Beltrami (3D)             & $7984$        & $3.9\times10^{17}$ & $1.1 \times 10^{19}$ & $18.4\times$ & $7977$ (def.\ $7$) \\
\hline
\end{tabular}

\vspace{0.25em}
{\footnotesize $^{\dagger}$Elasticity converges in a single outer iterate ($\varepsilon_{L}=0$), so no across-iteration ratio applies. $^{\ddagger}$The four Darcy/SPE10 permeability fields share $P=3169$; the $\kappa$ range spans the four cases, each full rank.}
\end{table}

\section{Discussion}
\label{sec:discussion}

The controlled comparison across four formulations provides consistent evidence that convexity of the training problem is the primary determinant of convergence reliability in physics-informed learning.
Across all nonlinear benchmarks (Bratu, Burgers, Buckley-Leverett, Kovasznay, Beltrami), partial elimination of either nonlinearity source yields inconsistent and often negligible improvements.
The \textsc{NiL-Q} formulation linearizes the PDE at each outer iteration, yet the convergence histories show that \textsc{NiL-N} and \textsc{NiL-Q} trace similar loss trajectories at every basis dimension. This outcome is consistent with the stationarity condition~\eqref{eq:stationarity}, which remains a coupled nonlinear system whenever the network architecture introduces nested nonlinearity.
On the Buckley-Leverett problem with gravitational effects, \textsc{NiL-Q} is in fact counterproductive and the quasilinear outer iteration resets the optimizer state at each step, producing transient loss spikes that compound the cost of the already-difficult inner optimization.
The \textsc{LiL-N} formulation removes the architectural nonlinearity while retaining the nonlinear PDE operator. This helps at small $P$ but degrades as the parameter count grows (Tables~\ref{tab:bratu_results},~\ref{tab:burgers_iterations},~\ref{tab:bl_case1_iterations}), because a larger coefficient space couples more basis modes through the physical nonlinearity, rendering the loss landscape progressively harder to navigate.
Only \textsc{LiL-Q}, which eliminates both sources simultaneously, achieves convergence in single-digit outer iterations in most cases (rising to between 11 and 16 only at the coarsest, under-resolved basis sizes) across all problem sizes and nonlinearity types.

The experiments also illuminate the factors that govern the achievable accuracy of the converged \textsc{LiL-Q} solution.
Theorems~\ref{thm:stationary_residual} and~\ref{thm:residual_bounds} predict that the nonlinear PDE residual at convergence is controlled by the linearized residual, which in turn reflects the best-approximation error $\varepsilon_{L}$ of the basis.
When $\varepsilon_{L} > 0$, the residual is bounded away from zero and the method converges to a neighborhood of the exact solution whose radius shrinks with~$\varepsilon_{L}$. When $\varepsilon_{L} = 0$, the floor vanishes and the method recovers the exact solution.
The residual band plots (Figs.~\ref{fig:bratu_residual_bounds} - \ref{fig:bl_gravity_residual_bounds}) confirm the $\varepsilon_{L} > 0$ regime across all nonlinear benchmarks. Increasing $P$ does not accelerate convergence (the outer iteration count remains in single digits), but it lowers the floor to which the method converges as the richer basis reduces~$\varepsilon_{L}$.
The coupled linear elasticity experiment confirms the $\varepsilon_{L} = 0$ regime in which the basis span contains the exact displacement fields, and the method recovers the analytical solution to $\mathcal{O}(10^{-16})$ in a single direct solve.
The basis sensitivity study on the viscous Burgers equation (Table~\ref{tab:basis_configs}) further demonstrates that, for a fixed parameter count, the choice of basis functions can account for orders of magnitude differences in solution accuracy. Boundary alignment is necessary but not sufficient, and what appears to matter is how efficiently the column space of the collocation matrix captures the structure of the linearized right-hand side across the entire domain.
This study also resolves the empirical anomaly noted in the introduction.
Within the same convex linearize-then-fit framework and at the identical parameter count $P = 625$, replacing a random-feature \textsc{ELM} basis ($\operatorname{MSE} = 5.0 \times 10^{-2}$) with an orthogonal Sin$\times$Cheb basis lowers the converged error to $1.7 \times 10^{-8}$, a gap of more than six orders of magnitude attributable solely to the approximation power of the representation. Theorem~\ref{thm:stationary_residual} and the spectral-decay estimate~\eqref{eq:spectral_rate} predict that the residual floor is set by~$\varepsilon_{L}$, which a spectral basis drives down at a controlled rate while random features do not---so that the convex formulation, far from being the inaccurate option it appeared to be in the random-feature setting, attains the accuracy of the best nonconvex solvers once paired with a basis of controlled approximation power.
These observations suggest that the practitioner's primary design decision within the \textsc{LiL-Q} framework is the construction of the approximation space: problem-specific knowledge of boundary behavior, expected smoothness, and symmetries can be embedded directly into the basis to improve accuracy for a given parameter budget.

The Kovasznay and Beltrami experiments demonstrate that \textsc{LiL-Q} generalizes from scalar PDEs to coupled nonlinear systems without modification to the core algorithm. The only additional complexity is the assembly of the block-column collocation structure, which accommodates multiple field variables and their inter-equation coupling.
On the Kovasznay flow at $\mathrm{Re} = 40$, \textsc{LiL-Q} at $P = 675$ matches the velocity accuracy of Jin et al.'s~\cite{Jin2021} VV-NSFnet ($\epsilon_u \approx 2 \times 10^{-5}$) with over 100 times fewer trainable parameters, and at $P = 1{,}200$ surpasses all published baselines by three to four orders of magnitude in under two seconds (Table~\ref{tab:kovasznay_comparison}).
On the Beltrami flow, the largest system in this study ($P = 7{,}984$; coupled fields in four spatiotemporal dimensions), \textsc{LiL-Q} converges in 4 outer iterations and approximately 5~minutes on CPU, producing velocity errors of 0.034\% and pressure errors of 0.75\% at $t = 1.0$ (Table~\ref{tab:beltrami_comparison}). At comparable or smaller parameter counts, the VP-NSFnet~\cite{Jin2021} reports pressure errors of 8.93\% (12 times larger) and the DD-PINN~\cite{Gu2024} uses 16 times more parameters.
These comparisons suggest that the gains from convexifying the training problem carry over to the coupled, multi-field systems encountered in computational physics.

The computational cost analysis in Section~\ref{sec:flop_analysis} shows that \textsc{LiL-Q} performs fewer total floating-point operations than \textsc{NiL-N} whenever $K_{\textsc{LiL-Q}} P < K_{\textsc{NiL-N}}$. This condition holds across all experiments in this study because the reduction from $10^2$ to $10^4$ gradient steps to single-digit (occasionally low-double-digit) outer iterations more than offsets the additional factor of $P$ per iteration arising from the dense QR factorization.
A separate and more fundamental performance gap exists relative to classical numerical methods. The finite volume solver completes the Darcy problem in approximately 0.1~s, compared to 24~s for \textsc{LiL}, a three-order-of-magnitude discrepancy.
This gap arises not from algorithmic deficiency but from the structural properties of the resulting linear systems. The global support of the Fourier and polynomial bases used in this work produces dense collocation matrices requiring $\mathcal{O}(NP^2)$ factorization, whereas finite volume discretizations yield sparse, banded systems solvable in near-linear time (using algebraic multigrid for elliptic problems for instance).
Bridging this gap likely requires basis functions with compact or local support that would induce sparsity in the collocation matrix and permit efficient sparse solvers, without sacrificing the mesh-free character of the method.
The \textsc{LiL-Q} framework also retains the structural flexibility of physics-informed machine learning: techniques developed to scale standard PINN training, including domain decomposition, adaptive residual weighting, and dynamic collocation point sampling, are directly applicable to the convex subproblems.

\section{Conclusion}
\label{sec:conclusion}

This study develops \textsc{LiL-Q}, a solution framework for physics-informed neural networks that converts each outer iteration into a convex linear least-squares problem by combining Bellman-Kalaba quasilinearization with linear-in-learnables representations.
The method was assessed on seven benchmark problems spanning scalar nonlinear PDEs (Bratu, viscous Burgers, Buckley--Leverett), coupled linear and nonlinear systems (plane-strain elasticity, Kovasznay flow at $\mathrm{Re} = 40$, three-dimensional Beltrami flow), and steady-state flow with heterogeneous coefficients (single-phase Darcy with SPE10 permeability).
A controlled comparison against three formulations that retain one or both sources of nonlinearity showed that \textsc{LiL-Q} converges in single-digit outer iterations in most cases (and in no more than a few dozen even at the coarsest basis sizes) across all tested problem sizes, nonlinearity types, and spatial dimensions, while the competing formulations require $10^2$ to $10^4$ gradient iterations and frequently stagnate at larger basis dimensions.
On coupled Navier-Stokes benchmarks, \textsc{LiL-Q} matched or exceeded published PINN accuracy with up to two orders of magnitude fewer trainable parameters, completing in seconds to minutes rather than hours of gradient-based training.
When the exact solution lies within the span of the chosen basis, the method recovers it to machine precision in a single direct solve.

Three concrete directions follow from the present results.
First, the dense QR factorization can be ported to GPU-accelerated direct and preconditioned iterative solvers (for example, cuSOLVER), enabling the larger systems needed for more complex problems. This could include matrix free methods that do not require storage of the coefficient matrix.
Second, replacing the globally supported with compactly supported alternatives could induce sparsity in the collocation matrix, narrowing the performance gap relative to classical mesh-based solvers while retaining the mesh-free formulation. For example, ~\citet{locspej} and~\citet{loccg} show that Newton-Kantorovich updates for general advection-diffusion-reaction problems tend to be sparse; this may enable the use of a basis expansion process that augments compactly supported activations across iterations.
Third, extending the framework to inverse problems, where bilinear parameter-state coupling and complex inter-equation nonlinearities introduce structure not addressed by the current convexification, remains an open challenge.
Fourth, systematic head-to-head benchmarking against classical solvers (finite-volume, finite-element, and spectral methods) across the present benchmark suite, beyond the single comparison reported here for Darcy flow, would quantify the conditions under which the cost-and-reliability advantages of \textsc{LiL-Q} translate into wall-clock competitiveness with mature mesh-based discretizations.

\section*{Data and Code Availability}
The implementation and scripts required to reproduce the experiments reported in this paper are available at \url{https://github.com/awojinrin/lilq-pinn}. The repository currently includes runnable examples for representative scalar and coupled benchmarks.
 
\section*{Acknowledgments}
The authors acknowledge financial support from the industrial affiliates of the Computational Reservoir Engineering (CORE) industry-university consortium at Texas A\&M University.


\bibliographystyle{cas-model2-names}

\bibliography{cas-refs}

\end{document}